\documentclass[a4paper, 12pt]{article}

\usepackage{latexsym}
\usepackage{a4wide}

\usepackage{amscd,amsfonts,amsmath,amssymb,amsthm}
\usepackage{epsf}
\usepackage[latin1]{inputenc}
\usepackage{graphics}
\usepackage{enumerate}
\usepackage{calc}
\usepackage[dvipsnames]{xcolor}
\usepackage{color}
\usepackage[active]{srcltx}

\usepackage{epsfig}

\usepackage[export]{adjustbox}
\usepackage{array}
\usepackage{bigstrut} \usepackage{caption}
\usepackage{colortbl}
\usepackage{float}
\usepackage{graphpap}
\usepackage{graphicx}
\usepackage{longtable}
\usepackage{mathabx}
\usepackage{mathrsfs}
\usepackage{multirow}
\usepackage{placeins} \usepackage{rotating}
\usepackage{tikz}
\usepackage{tkz-euclide}
\usepackage{times}
\usepackage{varwidth} \usepackage{wrapfig}
\usepackage{xspace}
\usepackage{pstricks}

\usepackage{bbm}
\usepackage{esint}
\usepackage{authblk}

\usepackage{pdfpages} 
\usepackage{pgfplots}
\pgfplotsset{compat=newest}

\usepackage{mathtools}
\mathtoolsset{showonlyrefs=true}

\newcommand{\R}{\mathbb{R}}
\renewcommand{\a}{\alpha}
\renewcommand{\b}{\beta}
\newcommand{\II}{{I\!I}}
\newcommand{\FDK}{{\delta\!K}}

\newcommand{\weak}{\rightharpoonup}
\newcommand{\weaks}{\stackrel{*}{\rightharpoonup}}
\newcommand{\SetDeformations}{\mathcal{A}} 
\renewcommand{\t}{\widetilde}
\renewcommand{\o}{\overline}
\newcommand{\D}{\nabla}
\renewcommand{\d}{\partial} \newcommand{\intd}{\,\mathrm{d}}
\newcommand{\h}[1]{\widehat{#1}}
\renewcommand{\h}{\widehat}
\newcommand{\Epot}{\mathcal{F}}
\newcommand{\Estored}{\mathcal{E}}
\newcommand{\Efree}{\mathcal{W}}
\newcommand{\costFct}{\mathcal{J}}
\newcommand{\costFctExpl}{\mathcal{J}}
\newcommand{\N}{\mathbb{N}}
\newcommand{\pf}{v}
\newcommand{\AreaHardPhase}{V}

\newcommand{\penaltyPerimeter}{\eta}
\newcommand{\ModicaMortola}[1]{\mathcal{R}^{{#1}}}
\newcommand{\ModicaMortolaInterfaceWidth}{\epsilon}
\newcommand{\DoubleWell}{\Psi}
\newcommand{\displacement}{w}
\newcommand{\quadPoint}{q} 
\newcommand{\LMIso}{\lambda}
\newcommand{\IsoFct}{G}
\newcommand{\SetNodes}{\mathcal{N}_h}
\newcommand{\SetInteriorNodes}{\mathcal{N}_h^{\text{int}}}

\definecolor{optColor}{rgb}{0.7, 0.75, 0.71}

\newtheorem{theorem}{Theorem}[section]
\newtheorem{lemma}[theorem]{Lemma}
\newtheorem{proposition}[theorem]{Proposition}
\newtheorem{corollary}[theorem]{Corollary}

\DeclareMathOperator{\curl}{curl}
\DeclareMathOperator{\Tr}{tr}
\DeclareMathOperator{\cof}{cof}
\DeclareMathOperator{\Id}{Id}
\DeclareMathOperator{\id}{id}

\numberwithin{equation}{section}

\makeatletter
\newcommand{\subjclass}[1]{%
  \let\@oldtitle\@title%
  \gdef\@title{\@oldtitle\footnotetext{\emph{Mathematics subject classification.} #1}}%
}
\newcommand{\keywords}[1]{%
  \let\@@oldtitle\@title%
  \gdef\@title{\@@oldtitle\footnotetext{\emph{Key words and phrases.} #1.}}%
}
\makeatother

\begin{document}

\title{On Material Optimisation for Nonlinearly Elastic Plates and Shells}
\author{Peter Hornung\footnote{Fakult\"at Mathematik, TU Dresden, 01062 Dresden (Germany)}
, Martin Rumpf and Stefan Simon\footnote{Institut f\"ur Numerische Simulation, Universit\"at Bonn, 53115 Bonn (Germany)}}

\date{}

\maketitle

\begin{abstract}
This paper investigates the optimal distribution of hard and soft material on elastic plates.
In the class of isometric deformations stationary points of a Kirchhoff plate functional with incorporated 
material hardness function are investigated and a compliance cost functional is taken into account.
Under symmetry assumptions on the material distribution and the load it is shown that cylindrical solutions are stationary points. 
Furthermore, it is demonstrated that the optimal design of cylindrically deforming, clamped rectangular plates is non trivial, i.e. with a material distribution which is not just depending on one axial direction on the plate.
Analytical results are complemented with numerical optimization results using a suitable finite element discretization and a phase field description of the material phases. Finally, using numerical methods an outlook on the optimal design of non isometrically deforming plates and shells is given.
\end{abstract}

\section{Introduction}\label{sec:intro} 
In nonlinear models of elastic deformations of plates and shells, 
a decomposition of the material into a hard and a soft phase can be taken into account,
and a natural question to ask is how to distribute these different materials in a mechanically optimal way. 
In this paper, we study this shape optimization problem both analytically and numerically.
For the background of shape optimization of bulk material and linearized elasticity as well as the homogenization perspective, 
we refer the reader to the textbooks \cite{HaNe96,Al02}.

The situation for thin plates addressed here differs from that of bulk materials.
We will essentially make use of the fact that plates can only be deformed isometrically,
i.e., preserving local lengths and angles. This isometry constraint arises naturally
in the rigorous derivation of Kirchhoff's plate theory from nonlinear three dimensional
elasticity provided in \cite{FrJaMue02}.

The characteristic global property of isometrically deformed plates is that they are developable surfaces. 
This is shown for smooth deformations in \cite{HaNi59} 
and it remains true for deformations with finite bending energy, cf. \cite{Po73, Ki01, Pa04}.
For further results on $W^{2,2}$ isometries we refer to \cite{Ho11a, Ho11, DaHo15, H-Eberhard}.

Sprekels and Tiba \cite{SpTi99} studied a linear plate or beam model given by the PDE $\Delta \delta^3(x) \Delta u =f$, 
where $\delta$ is the variable plate thickness, $u$ the normal displacement and $f$ the load. 
They took into account a volume or a tracking type cost functional and applied duality methods to solve the resulting optimization problem numerically.
In Arnautu et al. \cite{ArLaSp00}, the numerical approximation of the deformation of clamped plates via a reformulation of a system of second order PDEs is investigated. 
For the discretization piecewise affine, continuous finite elements are taken into account.
Recently, an optimal control problem for plates with variable thickness was studied by Deckelnick et al. \cite{DeHiJo17}.
Also here, the thickness of the plate is the design function.
The authors used a variational discretization of the resulting optimal control problem
and took into account a mixed formulation of the state equation based on a lowest-order Raviart-Thomas mixed finite element approach. 
They derived estimates for the discretization and the regularization error.

Our focus in this paper is on isometric deformations and we take into account the nonlinear Kirchhoff plate functional 
$$
\Efree[B,u] 
 = \frac{1}{2}\int_S B(x) |A(x)|^2 \intd x - \int_S f(x)\cdot u(x) \intd x, 
$$
where $u$ denotes the deformation of a plate $S$, $A$ the second fundamental of the deformed plate, $f$ the load 
and $B:S\mapsto \{a,b\}$ a binary material hardness function, which describes the distribution of a hard phase $b$ and a soft phase $a$ with $b>a$. 
In \cite{HoRuSi16}, we already considered the optimal distribution of a soft and a hard material for nonlinearly elastic planar beams. 
We proved that under gravitational force the optimal distribution involves no microstructure and is ordered.
We also provided numerical simulations which confirm and extend this observation.
Now, we treat the two dimensional case with a particular focus on cylindrical deformation of rectangular plates clamped on one side.
\\
For a homogeneous material distribution, Bartels \cite{Ba11} approximated large bending isometries by making use of the discrete Kirchhoff triangle and a linearization of the isometry constraint.
At variance we implemented the exact isometry constraints on all vertices of the underlying triangulation.
To describe the distribution of hard material we again took into account a phase field model of Modica--Mortola type.
Furthermore, we investigated the optimization of a material distribution on non isometrically deforming elastic plates and shells. 

The paper is organized as follows.
In Section \ref{sec:plate} we discuss the Kirchhoff plate functional with a material hardness function.
We in particular show that under suitable symmetry assumptions 
there exists a stationary point in the class of isometric deformations which is a cylindrical solution.
Using the discrete Kirchhoff triangle scheme we investigate in Section \ref{sec:numericStateEq} 
a finite element discretization of the state equation
and compare numerically the compliance cost for different distributions of the hard material on a rectangular, clamped plate. 
Then, in Section \ref{sec:nontrivial} we compare these designs analytically and show that 
optimal designs for sufficient strong applied forces are indeed not just depending on one axis direction of the plate.
In Section \ref{sec:Developability} a mild condition on the deformed boundary is investigated under
which (reflection symmetric) isometric deformations are
already cylindrical and in Section \ref{sec:Disenyo} existence of an optimal design in the class of cylindrically deforming plates is shown using a relaxation approach.
Properties of such an optimal design and in particular the distribution profile of hard material are analyzed in Section \ref{sec:properties}. 
Section \ref{sec:numericOptDesign} confirms these findings numerically based on 
the above finite element discretization of plates with a material distribution modeled via a phase field approach.
Finally, Section \ref{sec:Mixed} generalizes the model considering elastic, in general non isometric deformations of plates and shells as minimizers of the sum of an elastic membrane and bending energy. 
It is shown that the resulting optimal material distribution is determined by fine scale structure, which 
are an indication for the onset of a microstructure in the limit of this shape 
optimization problem for vanishing regularization parameter.

\section{Stationary points of the Kirchhoff plate functional}
\label{sec:plate}

In the following, we consider a rectangular domain $S = (0, \ell)\times I$ with $I=(-\frac{1}{2}, \frac{1}{2})$
of length $\ell$ and width $1$ as the midplate of the reference configuration of a thin elastic object $S_\delta = S \times (-\delta/2,\delta/2)$.
For a deformation $U:S_\delta \to \R^3$, we consider the stored elastic energy functional
$
 \mathcal{E}_{3D}[U] = \int_{S_\delta} W(\D U) \intd x
$.
In \cite{FrJaMue02b}, under suitable assumptions on the hyperelastic density function $W$ and for a scaling $1/\delta^3$,
a $\Gamma$-convergence result was established, where the limit functional 
$
 \Estored[u] = \frac{1}{2} \int_S |A[u](x)|^2 \intd x
$
is only finite for $u \in W^{2,2}_{\text{iso}}(S)$.
More precisely, we denote by $\Id$ the $2\times 2$ identity matrix and we set
$$
W^{1,\infty}_{\text{iso}}(S) = \{u\in W^{1,\infty}(S, \R^3) : (\D u)^T(\D u) = \Id \mbox{ almost everywhere }\} \, .
$$
Then we define $W^{2,2}_{\text{iso}}(S) = W^{2,2}(S, \R^3)\cap W^{1,\infty}_{\text{iso}}(S)$.
We will always consider plates which are {\em clamped} at the `left' lateral boundary, i.e.,
deformations which belong to the admissible class
$$
\SetDeformations_S = \left\{u\in W^{2,2}_{\text{iso}}(S) : u = \id \mbox{ and }\D u = \Id
\mbox{ in the trace sense on }\{0\}\times (-\frac{1}{2}, \frac{1}{2})\right\}.
$$
Here and elsewhere $\id$ denotes the identity or the standard injection of $\R^2$ into $\R^3$.
Moreover, $A[u]$ denotes the second fundamental form of an immersion $u : S\to\R^3$.
When there is no danger of confusion we will simply write $A$.

Now, we take into account a material distribution $B : S \to [a,b]$ on the reference plate,
which models the material hardness between two positive constants $a < b$.
Then the elastic energy stored in the deformed configuration $u(S)$ is given by Kirchhoff's plate energy, weighted with the material hardness $B$. 
Moreover, an external force $f\in L^2(S, \R^3)$ is acting. 
Thus, the free elastic energy is given by
\begin{align}\label{eq:FreeElastEnergy}
 \Efree[B,u] 
 = \Estored[B,u] - \Epot[u] 
 = \frac{1}{2}\int_S B(x) |A(x)|^2 \intd x - \int_S f(x)\cdot u(x) \intd x \, .
\end{align}
Next, let us review some results about $W^{2,2}$ isometric immersions. 
Due to a result by Kirchheim \cite{Ki01}, we have $W^{2,2}_{\text{iso}}(S)\subset C^1(S, \R^3)$.
Define $C_{\D u}$ as the set of points $x\in S$ such that
$\D u$ is constant in a neighbourhood of $x$. 
By definition $C_{\D u}$ is open. As shown in \cite{Ki01, Pa04}
the deformation $u$ is developable on $S\setminus C_{\D u}$, i.e.,
for every $z\in S\setminus C_{\D u}$ there exists a unique line segment, denoted by $[z]$,
with the properties that $[z]\subset S$, both endpoints of $[z]$ lie on $\d S$,
and $\D u$ is constant on $[z]$.
\\
Being open, the set $C_{\D u}$ consists of countably many connected components. For 
each such component $U$ there exists a countable set $\Sigma\subset S\setminus C_{\D u}$
such that $S\cap\d U = \bigcup_{z\in\Sigma} [z]$, cf. \cite{Ho11a}.

To study stationary points of Kirchhoff's plate energy in the class of isometric immersions
we follow \cite{Ho17, Ho20} and consider for a given $u\in W^{2,2}_{\text{iso}}(S)$ 
a one parameter family $(u_t)_{t\in (-1,1)}\subset W^{2,2}_{\text{iso}}(S)$ such that the limit
$$
\tau = \lim_{t\to 0}\frac{1}{t}\left(u_t - u\right).
$$
exists weakly in $W^{2,2}(S, \R^3)$. In what follows such a family $(u_t)$ will be called
a bending of $u$.
\\
We can compute the first variation of $\Efree[B,\cdot]$ along the family $(u_t)$. It is given by
\begin{equation}
\label{first variation}
\frac{d}{dt}\Big|_{t = 0} \Efree[B,u_t] = \int_S B(x) A[u_t]:b(x)\intd x - \int_S f(x)\cdot\tau(x) \intd x.
\end{equation}
Here $b\in L^2(S, \R^{2\times 2})$ is the weak $L^2$ limit
\begin{equation}
\label{lin2f}
b = \lim_{t\to 0}\frac{1}{t}\left(A[u_t] - A[u]\right).
\end{equation}
Following \cite{Ho17, Ho20} we say that $u$ is stationary for $\Efree[B,\cdot]$ on some subset $\SetDeformations\subset W^{2,2}_{\text{iso}}(S)$
if \eqref{first variation} is zero for all bendings $(u_t)\subset\SetDeformations$.
Applying the direct method one can show that $\Efree[B,\cdot] : \SetDeformations_S\to [0, \infty)$ attains a minimum,
see e.g. \cite{Ho17}.
\bigskip

A deformation $u : S\to\R^3$ will be called {\em cylindrical} if
\begin{align*}
u(x_1, x_2) =
\begin{pmatrix}
u_1(x_1, 0)
\\
x_2
\\
u_3(x_1, 0)
\end{pmatrix}
\end{align*}
for all $(x_1, x_2)\in S$.
A cylindrical deformation $u$ of $S$ is determined by the planar arclength parametrised curve 
$u(\cdot, 0)$.
For isometric, cylindrical $u$ we therefore introduce as in \cite{HoRuSi16} the phase $K\in W^{1,2}(0, \ell)$
of $u' = \d_1 u(\cdot, 0)$ by
\begin{equation}
\label{uK}
u' =
\begin{pmatrix}
\cos K \\ 0 \\ \sin K
\end{pmatrix}.
\end{equation}
Once $u(0)$ is prescribed, $u$ is determined by $K$.
The normal to $u$ is given by $n = (-\sin K, 0, \cos K)$.
\\
Similar to \cite{HoRuSi16} we introduce
$$
F(x_1, x_2) = \int_{x_1}^{\ell} f(s, x_2) ds,
$$
and set $\o F(x_1) = \int_{-1/2}^{1/2} F(x_1, x_2)\intd x_2$ as well as
$\o B(x_1) = \int_{-1/2}^{1/2} B(x_1, x_2)\intd x_2$.
\\
Note that in terms of $K$, the clamped boundary condition on $u\in\SetDeformations_S$ is equivalent to $K(0) = 0$
and $u(0) = 0$. Thus $K$ belongs to the space
$$
W_{l}^{1,2}(0, \ell) = \{K\in W^{1,2}(0,\ell) : K(0) = 0\}.
$$
Define the functional $\Efree_{\text{phase}}[B,\cdot] : W_{l}^{1,2}(0, \ell)\to\R$ by
\begin{equation}
\label{en-4}
\Efree_{\text{phase}}[B,K] = \frac{1}{2}\int_0^{\ell} \o B(x_1)(K'(x_1))^2\ \intd x_1 - \int_0^{\ell} 
\o F(x_1)\cdot u'(x_1)\intd x_1.
\end{equation} 
Stationary points $K$ of \eqref{en-4} satisfy the Euler-Lagrange equation
\begin{equation}
\label{ELK-1}
-(\o BK')' = \o F\cdot n
\mbox{ in the dual space of }
W^{1,2}_l(0, \ell).
\end{equation}
Observe that \eqref{ELK-1} encodes the Dirichlet boundary condition $K(0) = 0$ and the natural boundary
condition $K'(\ell) = 0$.
\\
It is easy to construct minimisers $K$ of \eqref{en-4} within $W^{1,2}_l(0, \ell)$.
It was shown in \cite{HoRuSi16} that, if $\o f= (\cos\b, 0, \sin\b)$
for some constant $\b\in [-\pi/2, 0)$, then \eqref{en-4} admits
a unique global minimiser $K$, which can equivalently be characterized as being
the unique solution of \eqref{ELK-1}
with the additional property that
\begin{equation}
\label{statecon}
K(x_1)\in [\beta, \beta + \pi)\mbox{ for all }x_1\in (0, \ell).
\end{equation}
And this, in turn, is equivalent to the stronger condition
$K\in (\beta, 0]$ on $(0, \ell)$.
\bigskip

Consider a cylindrical deformation $u\in\SetDeformations_S$ which is stationary
within the class of cylindrical deformations, i.e., 
$$
\frac{d}{dt}\Big|_{t = 0} \Efree[B,u_t] = 0
$$
for all bendings $(u_t)_{t\in (-1,1)}\subset\SetDeformations_S$ of $u$ such that each $u_t$ is cylindrical. Its phase is clearly a solution of \eqref{ELK-1}.

The following proposition asserts that any cylindrical $u\in\SetDeformations_S$ whose phase $K$ satisfies
\eqref{ELK-1} is in fact stationary among all maps in $\SetDeformations_S$.
This is an instance of the principle of symmetric stationarity, see e.g. 
\cite{Ho17, Ho20, Ho14, DaDeGr08}.
The following arguments follow the conceptual framework developed in \cite{Ho17},
but in the present case the computations can be carried out explicitly.

\begin{proposition}\label{procyl}
Let $f\in L^2(S, \R^3)$ be such that $e_2\cdot f\equiv 0$.
Assume, moreover, that 
\begin{equation}
\label{moments}
\begin{split}
\int_{-1/2}^{1/2} x_2 B(x_1, x_2) \intd x_2 &= 0\mbox{ for almost every }x_1\in (0, \ell),
\\
\int_{-1/2}^{1/2} x_2 f(x_1, x_2) \intd x_2 &= 0\mbox{ for almost every }x_1\in (0, \ell).
\end{split}
\end{equation}
Let $u\in\SetDeformations_S$ be cylindrical and assume that its phase $K$ satisfies 
\eqref{ELK-1} and that $K'\neq 0$ almost everywhere on $(0, \ell)$.
Then $u$ is stationary for $\Efree[B,\cdot]$.
\end{proposition}

The hypothesis that $K'\neq 0$ almost everywhere is satisfied in the situations
that we are mainly interested in:
\begin{lemma}\label{K'neq0}
Let $f_0\in \R^3\setminus\{0\}$, let $f\in L^2(S, \R^3)$
and assume that $\o f$ is almost everywhere parallel $f_0$.
Let $K\in W^{1,2}_{l}(0, \ell)$ be a solution 
of \eqref{ELK-1}. If $K'$ has infinitely many zeros in $[0, \ell]$ then
$K$ is identically zero.
\end{lemma}
\begin{proof}
If the set $\{\o BK' = 0\} = \{K' = 0\}$ is not finite, then it
has an accumulation point $t_0\in [0, \ell]$. Since $\o BK' = k\in C^1([0,\ell])$
due to \eqref{ELK-1}, we see that $k(t_0) = \o F(t_0)\cdot n(t_0) = 0$.
\\
Since by the hypotheses on $f$ we know that $\o F$ is always parallel
to $\o F(t_0)$, we conclude that $\o F\cdot n(t_0)$ is identically zero.
Hence the constant map $(K(t_0), 0)$ is a solution of the first order system
\begin{align*}
K' &= \frac{k}{\o B}
\\
- k' &= \o F\cdot n
\end{align*}
associated with \eqref{ELK-1}. Hence by uniqueness we must have
$K \equiv K(t_0)$. Since $K\in W^{1,2}_l$ we conclude that $K = 0$.
\end{proof}
The proof of Proposition \ref{procyl} uses the following lemma.

\begin{lemma}\label{lecyl}
Let $u\in W^{2,2}_{\text{iso}}(S)$ be cylindrical
and let $(u_t)\subset W^{2,2}_{\text{iso}}(S)$ be a bending of $u$ and denote by $b$ the weak
$L^2$ limit of
$\frac{A[u_t] - A[u]}{t}$ as $t\to 0$.
Let $J\subset (0,\ell)$ be an open interval such that $A[u]\neq 0$ almost everywhere on $J\times (-1/2, 1/2)$.
Then there exist $\a$, $\b\in W^{2,2}(J)$ such that
\begin{equation}
\label{lecyl-1}
b(x) =
\begin{pmatrix}
\a''(x_1) + x_2\b''(x_1) & \b'(x_1)
\\
\b'(x_1) & 0
\end{pmatrix}
\mbox{ for almost every }x\in J\times (-1/2, 1/2).
\end{equation}
If, moreover, $\d_1 u_t = 0$ on $\{0\}\times (-1/2, 1/2)$ in the trace sense
for all $t$, then $\beta'(0) = 0$.
\end{lemma}
\begin{proof}
We assume without loss of generality that $J = (0, \ell)$ 
and take into account  $\curl A := (\d_1 A_{21}-\d_2 A_{11},\d_1 A_{22}-\d_2 A_{12})$ 
for a field $A$ of $2\times2$ matrices. Then, we deduce from the fact that $\curl A[u_t] = 0$ for isometries 
that $\curl b = 0$ and from this and the fact that $b$ is symmetric 
we obtain that there exists $m\in W^{2,2}(S)$ such that
$b = \D^2 m$ almost everywhere on $S$. Define
\begin{align*}
\a(x_1) &= \int_{-1/2}^{1/2} m(x_1, x_2)\ dx_2
\\
\b(x_1) &= 12\int_{-1/2}^{1/2} x_2 m(x_1, x_2)\ dx_2.
\end{align*}
Since $m\in W^{2,2}(S)$, we see that $\a$, $\b\in W^{2,2}(0, \ell)$.
On the other hand, the Gauss curvature $\det A[u_t]$ of the immersion  $u_t$ vanishes for all $t$ and therefore a differentiation with respect to $t$ implies $\cof A : b = 0$ everywhere.
As $u$ is cylindrical, this implies that
$A_{11}\d_2\d_2 m = 0$
almost everywhere on $S$. Hence $\d_2\d_2 m = 0$ almost everywhere on $S$.
This implies that 
$$
m(x) = \a(x_1) + x_2\b(x_1) \mbox{ for almost every }x\in S.
$$
Hence $\D^2 m$ equals the right-hand side of \eqref{lecyl-1}.
\\
To prove the last assertion in the statement of the lemma,
denote by $\tau$ the weak $W^{2,2}$ limit of $t^{-1}(u_t - u)$.
Since $u$ is cylindrical,
$$
\frac{1}{t}\d_1\d_2(u_t - u) = \frac{1}{t} A_{12}[u_t] n_t
= \frac{1}{t}(A_{12}[u_t] - A_{12}[u]) n_t \weak b_{12} n
$$
because $n_t\to n$ strongly in $L^2$. Here, we have used that for isometric deformations
$\d_1\d_2 u = (\d_1\d_2 u \cdot n) n$.
The left-hand side converges weakly in $L^2$
to $\d_1\d_2\tau$. Hence $\d_2\d_1\tau = \d_1\d_2\tau = \beta'n\in W^{1,2}$ is independent
of $x_2$.
\\
Finally, $\d_1 u_t = 0$ on $\{0\}\times (-1/2, 1/2)$ implies that $\d_1\tau = 0$ 
and thus also $\d_2\d_1\tau = 0$ on $\{0\}\times (-1/2, 1/2)$ in the trace sense.
From this it follows that $\beta'(0) = 0$.
\end{proof}

\begin{proof}[Proof of Proposition \ref{procyl}]
Let $(u_t)\subset\SetDeformations_S$ be a bending of $u$.
(Notice that the $u_t$ are in general not cylindrical.)
As above, denote by $\tau$ the weak $W^{2,2}$ limit of $t^{-1}(u_t - u)$ and by $b$ 
the weak $L^2$ limit of $\frac{1}{t}(A[u_t] - A[u])$ as $t\to 0$. 
\\
The map $b$ satisfies the hypotheses of Lemma \ref{lecyl}. 
Let $\a$ and $\b$ be as in the conclusion of that lemma.
Define $\eta : (0, \ell)\to\R$ by setting $\eta(x_1) = \a'(x_1) - \a'(0)$. Define $\varphi : (0, \ell)\to\R^3$
by setting
$$
\varphi(x_1) = \int_0^{x_1} \beta'(s) u'(s)\ ds - \eta(x_1) e_2.
$$
Here and in what follows we write $u'(x_1) = \int_I (\d_1 u)(x_1, x_2)\ dx_2$
for cylindrical deformations $u$.
\\
Define $\Phi : S\to\R^3$ by setting
\begin{equation}
\label{equi-2}
\Phi(x) = \varphi(x_1) - x_2\beta'(x_1) e_2\mbox{ for all }x\in S.
\end{equation}
Using \eqref{lecyl-1} we see that
\begin{equation}
\label{equi-1}
\begin{split}
\d_1\Phi &= b_{12}u' - b_{11}e_2
\\
\d_2\Phi &= - b_{12}e_2.
\end{split}
\end{equation}
Since $\d_2 u = e_2$, this implies that
$$
\d_{i}\Phi\times\d_{j} u = b_{ij}\ n
\mbox{ for }i, j = 1, 2.
$$
The isometry property of $u_t$ implies that $\partial_i \tau$ is orthogonal 
to $\partial_i u$.
More precisely, according to \cite[Lemma 3.8]{Ho17} (see also \cite{Ve63, LoMa98})
there exists $\t\Phi\in W^{1,1}(S, \R^3)$ (the so-called bending field of $\tau$)
such that
\begin{align}
\label{tildeOmega}
\d_i\tau = \t\Phi\times\d_i u
\end{align}
and $\d_{i}\t\Phi\times\d_{j} u = b_{ij}\ n$ for $i$, $j = 1, 2$ . Hence
$$
\d_{i}\t\Phi\times\d_{j} u = \d_{i}\Phi\times\d_{j} u\mbox{ for }i, j = 1, 2.
$$
Since $u$ is an immersion, this readily implies that $\D\t\Phi = \D\Phi$
almost everywhere on $S$.
On the other hand, $(u_t, \D u_t) = (u, \D u)$ on $\{x_1 = 0\}$ for all $t$,
due to the clamped boundary condition on the left boundary.
Hence $\D\tau = 0$ on $\{0\}\times (-1/2, 1/2)$
in the trace sense. 
Hence $\t\Phi = 0$ on $\{0\}\times (-1/2, 1/2)$, due to \eqref{tildeOmega}. 
\\
But by definition we observe that $\Phi = 0$ on $\{0\}\times (-1/2, 1/2)$.
Notice that Lemma~\ref{lecyl} implies $\beta'(0) = 0$.
Therefore, we have that $\t\Phi = \Phi$. 
In particular, \eqref{tildeOmega} implies that
\begin{equation}
\label{equi-5}
\d_1\tau = \Phi\times\d_1 u.
\end{equation}
Hence, by integration by parts,
\begin{equation}
\label{equi-4}
\int_S f\cdot\tau = \int_S F\cdot\d_1\tau
= \int_0^{\ell} \o F\cdot (\varphi \times u')
\end{equation}
because $\beta'$ and $u'$ are independent of $x_2$ and the first moment of $f$ (hence that of $F$) along
$x_2$ is zero.
\\
Since $u'\cdot e_2\equiv 0$, we see that $u'(x_1)\times\int_0^{x_1} \beta'(s) u'(s) \ ds$ is parallel to $e_2$
for all $x_1$. Since $F\cdot e_2 \equiv 0$, we deduce from \eqref{equi-4} that
\begin{equation}
\label{equi-7}
\int_S f\cdot\tau
= -\int_0^{\ell} \eta \o F\cdot e_2\times u' = \int_0^{\ell} \eta \o F\cdot n.
\end{equation}
On the other hand, testing \eqref{ELK-1} with $\eta$ we see that
\begin{equation}
\label{equi-6}
\int_0^{\ell} \eta \o F\cdot n = \int_0^{\ell} \o B K'\eta' = \int_0^{\ell} \o B K'\a''.
\end{equation}
But due to \eqref{moments}
\begin{align*}
\int_0^{\ell} \o B K'\a'' = \int_S B K' \left( \a'' + x_2 \b'' \right) = \int_S B K' b_{11}
=  \int_S B A : b.
\end{align*}
Inserting this into \eqref{equi-6}
and recalling \eqref{equi-7}, we conclude that indeed
$$
\int_S B A : b = \int_S f\cdot\tau.
$$
\end{proof}

Now we can assert the existence of a stationary point which is cylindrical.

\begin{theorem}[existence of cylindrical stationary points]\label{excyl}
Let $B\in L^{\infty}(S; [a, b])$ and $f\in L^2(S, \R^3)$ satisfy \eqref{moments}, let $f_0\in\R^3\setminus\{0\}$ with $f_0\cdot e_2 = 0$, 
and assume that  $\o f$ is almost everywhere parallel to $f_0$.
Then there exists a cylindrical deformation $u\in\SetDeformations_S$
which is a stationary point of $\Efree[B,\cdot]$ on $\SetDeformations_S$.
More precisely, every cylindrical deformation $u\in\SetDeformations_S$ whose phase $K$
satisfies \eqref{ELK-1} is a
stationary point of $\Efree[B,\cdot]$ on $\SetDeformations_S$.
\end{theorem}
\begin{proof}
Let $K\in W^{1,2}_l(0, \ell)$ be a solution of \eqref{ELK-1}.
The immersion $u$ defined by \eqref{uK} and the condition $u(0) = 0$
belongs to $\SetDeformations_S$.
Lemma \ref{K'neq0} implies that $K'$ has only finitely many zeros
(the case $K\equiv 0$ is trivial). Hence Proposition \ref{procyl} implies
that $u$ is a stationary point of $\Efree[B,\cdot]$.
\end{proof}

\section{Numerical discretization of nonlinearly elastic plates} \label{sec:numericStateEq}
To discretize bending isometries we follow \cite{Ba11} and make use of the discrete Kirchhoff triangle (DKT) as a suitable finite element space. 
In particular nodal wise degrees of freedom for derivative of the displacement enable to implement the isometry constraint as a simple constraint at nodal positions of a triangular mesh.
Here, we additionally take into account the material distribution $B$.
Different to \cite{Ba11}, where a discrete gradient flow approach 
with a linearized isometry constraint was proposed,
we take into account a Newton method for a associated Lagrangian with an exact isometry constraint at nodal positions.

For simplicity, we assume that $S \subset \R^2$ is polygonal, s.t. we can directly consider a triangulation $\mathcal{T}_{h}$ of $S$.
In particular, this is guaranteed for our case of interest, where $S$ is a rectangular domain.
Otherwise, $S$ could be approximated by a polygonal domain.
Then we denote by $\SetNodes$ the set of nodes in $\mathcal{T}_{h}$.
First, we consider discrete material distributions $B_h$ in the space of continuous and piece-wise affine functions
\begin{align}
   V_h^1(S) := \left\{ B_h \in W^{1,2}(S)  \; : \;  B_h \vert_{T} \in P_{3,\text{red}}(T) \; \forall T \in \mathcal {T}_{h} \right\} \, .
\end{align}
Now, we recall the DKT element \cite{BaBaHo80}.
For a triangle $T$ in $\mathcal{T}_{h}$, let $P_k(T)$ be the space of polynomials of order $k \in \N$.
In analogy, we consider for an edge $E$ the space $P_k(E)$.
Furthermore, we define for a triangle $T$ the space 
\begin{align}
  P_{3,\text{red}}(T) 
  := \left\{ w \in P_3(T) \; : \; w(p_T) = \frac{1}{3} \sum_{p \in \SetNodes \cap T} w(p) + \nabla w(p) (p_T - p) 
     \right\} \, ,
\end{align}
of polynomials of order three reduced by one degree of freedom by
where $p_T = \frac{1}{3} \sum_{p \in \SetNodes \cap T} p$ denotes the center of mass of the triangle $T$. 
This finally lead us to the following finite element spaces.
\begin{align}
 W_h(S) & :=
      \left\{ w_h \in W^{1,2}(S) \; : \;  w_h \vert_{T} \in P_{3,\text{red}}(T) \; \forall T \in \mathcal {T}_{h}, \, 
                                                   \nabla w_h(p) \text { continuous } \forall p \in \SetNodes
      \right\} \, , \\
  \Theta_{h}(S) &:=
      \left\{  \theta_{h} \in \left( W^{1,2}(S) \right)^{2}  \; : \;
                     \theta_{h} \vert_{T} \in \left( P_2(T) \right)^{2} \; \forall T \in \mathcal {T}_{h}, \,
                     \theta_{h} \cdot n \in P_1(E) \text{ for every edge } E
      \right\} \, .
\end{align}
Then we consider a discrete gradient operator
\begin{align}
    \nabla_h \colon W_h(S) \to \Theta_{h}(S) \, , \quad
    w_h \mapsto \nabla_h w_h = \theta_{h} (w_h) \, ,
\end{align}
where $\theta_{h} (w_h) \in \Theta_{h}(S)$ is the uniquely defined function that satisfies for each triangle $T \in \mathcal{T}_{h}$ with nodes $p_0,p_1,p_2$
the interpolation conditions
$\theta_{h} (w_h) (p_{i}) = \nabla w_h (p_{i})$ for $0 \leq i \leq 2$ and
$\theta_{h} (w_h) (p_{ij}) \cdot (p_{j} - p_{i})  = \nabla w_h (p_{ij}) \cdot (p_{j} - p_{i})$ for $0 \leq i, j \leq 2$ 
with $p_{ij} = p_{ji} = \frac {1}{2} (p_{i} + p_{j})$.
This allows defining an approximative second derivative of $w_h$ by $\nabla \theta_{h} (w_h)$.
Note that $w_h$ can be determined by the values $w_h(p)$ and the derivatives $\nabla w_h(p)$ at nodes $p$, and thus, has three degrees of freedom per node.
As in \cite{Ba11}, we use the space $W_{h,\Gamma}(S)^3$ to discretize elastic displacements $w$ with $w(x) = u(x)-x$
as component-wise DKT-functions with clamped boundary conditions on $\Gamma \subset \partial \omega$.
Remember that in the case of the rectangular domain, $\Gamma$ is typically given by the left side.

\par\bigskip

To implement the corresponding discrete energy functionals,
we apply a Gaussian quadrature of degree $6$ with $Q = 12$ quadrature points for each triangle element with weights $\omega$, 
and obtain a discrete bending energy
\begin{align}\label{eq:Eh}
 \Estored_h[B_h,\displacement_h] 
 & = \sum_{T \in \mathcal{T}_h} |T| 
       \sum_{\quadPoint = 1, \ldots, Q} 
          \omega_\quadPoint B_h(x_\quadPoint) \| \nabla \nabla_h \displacement_h(x_\quadPoint) \|^2  
\end{align}
and a discrete potential energy
\begin{align}\label{eq:Fh}
 \Epot_h[\displacement_h] 
 & = \sum_{T \in \mathcal{T}_h} |T| \sum_{\quadPoint = 1, \ldots, Q} \omega_q f_h(\quadPoint) \cdot \displacement_h(\quadPoint)  \, ,
\end{align}
where $|T|$ denotes the area of the triangle $T$.
Consequently, a discrete free energy is given by
$\Efree_h[B_h,\displacement_h] = \Estored_h[B_h,\displacement_h] - \Epot_h[\displacement_h]$.
Note that the isometry constraint in terms of a displacement $\displacement$ is given by
\begin{align} \label{eq:IsometryConstraintPlatePointwise}
 0 =  \left(
  \begin{array}{cc}
    G_h^{11} & G_h^{12} \\
    G_h^{12} & G_h^{22} 
  \end{array}
  \right)[w_h](p)
\end{align}
on all non Dirichlet nodes $p\in \SetInteriorNodes = \SetNodes \setminus \Gamma$, where
$\IsoFct_h^{11}[\displacement_h ]  = 2 \partial_1 \displacement_1 + \sum_{i=1}^3 (\partial_1 \displacement_i)^2$, 
$\IsoFct_h^{12}[\displacement_h ] =  \partial_2 \displacement_1 + \partial_1 \displacement_2 + \sum_{i=1}^3  \partial_1 \displacement_i \partial_2 \displacement_i$, and  
$\IsoFct_h^{22}[\displacement_h ] = 2 \partial_2 \displacement_2 + \sum_{i=1}^3 (\partial_2 \displacement_i)^2$.
To encode the nodalwise isometry constraint we consider 
a vector  $\LMIso_h(p) = \left( \LMIso_h^1(p),\LMIso_h^2(p),\LMIso_h^{12}(p) \right)$ 
of Lagrange multipliers 
and define 
\begin{align}
 \IsoFct_h[\displacement_h, \LMIso_h ] & = \sum_{p \in \SetInteriorNodes} 
      \LMIso_h^1(p) \IsoFct_h^{11}[\displacement_h ] (p)    
 +    \LMIso_h^2(p) \IsoFct_h^{22}[\displacement_h ] (p)
 +    \LMIso_h^{12}(p) \IsoFct_h^{12}[\displacement_h] (p)  \, .
\end{align}
Note that all the values $(\partial_j \displacement_h(p))_{i}$ for $j=1,2$ and $i=1,2,3$ are explicit degrees of freedom for all $p\in \SetInteriorNodes$.
Finally, the discrete Lagrangian functional is given by
\begin{align} \label{eq:discreteLagrangian}
 \mathcal{L}_h[ B_h, \displacement_h, \LMIso_h ]
 = \Estored_h[B_h,\displacement_h] - \Epot_h[\displacement_h]+ \IsoFct_h[\displacement_h,\LMIso_h]  \, . 
\end{align}
Then, to compute for a fixed material distribution given by $B_h$ solutions to the state equation 
we apply Newton's method to solve the saddle point problem
\begin{align} \label{eq:shapeDesignShellsStateEquationLagrangian}
 \partial_{(\displacement_h, \LMIso_h )} \mathcal{L}_h[ B_h, \displacement_h, \LMIso_h ] = 0 \, .
\end{align}

It was shown in \cite{HoRuSi16} that the optimal design for planar beams is simply
$b\chi_{(0, t_0)} + a\chi_{(t_0, 1)}$ for a suitable $t_0\in (0, 1)$.
Naively, one might expect these to apply as well
to the case of cylindrically deformed plates addressed here. However, this is not the case.

In this section we provide an example of a very simple design $B$ 
which is not constant along $x_2$
and which beats any design depending only on $x_1$. It is obtained by putting a horizontal 
strip of hard material across the plate.
In fact, we compare three different material distributions, 
where, depending on the area $\AreaHardPhase$, the subdomain covered with hard material is given by
\begin{enumerate} \item[I.] a layer $[0,\AreaHardPhase] \times [0,1]$ at the clamped boundary, i.e.,
       the solution to the $1D$ problem, 
 \item[II.] a layer $[0,1] \times [0.5 - 0.5 \AreaHardPhase, 0.5 + 0.5 \AreaHardPhase]$ 
       orthogonal to the clamped boundary, and
 \item[III.] a square 
       $[0, \sqrt{\AreaHardPhase} ] 
        \times 
        [0.5 - 0.5 \sqrt{\AreaHardPhase}, 0.5 + 0.5 \sqrt{\AreaHardPhase} ]$ 
       centered in the middle of the clamped boundary.
\end{enumerate}
Here, we consider three area fractions $\AreaHardPhase = 0.25,0.5,0.75$ for the amount of hard material.
In Figure~\ref{fig:CompareDesignsIsometry},
we compare the potential energy of these three designs in dependence of $|f|$.
For all computations, we use a mesh of $|\SetNodes| = 16641$ nodes.
We observe that for a large area fraction $\AreaHardPhase = 0.75$ of the hard material,
the $1D$ optimizer (I) is optimal w.r.t. the potential energy independent of $|f|$.
In any cases, it seems that for large forces design (I) is optimal.
For an area fraction $\AreaHardPhase = 0.5$, and small forces, design (III) is optimal.
For an area fraction $\AreaHardPhase = 0.25$, we even obtain that design (II) is optimal for small forces and design (III) is better on an intermediate range. 
\begin{figure}[!htbp]
\resizebox{1.0\textwidth}{!}{
    {\includegraphics[width=\textwidth]{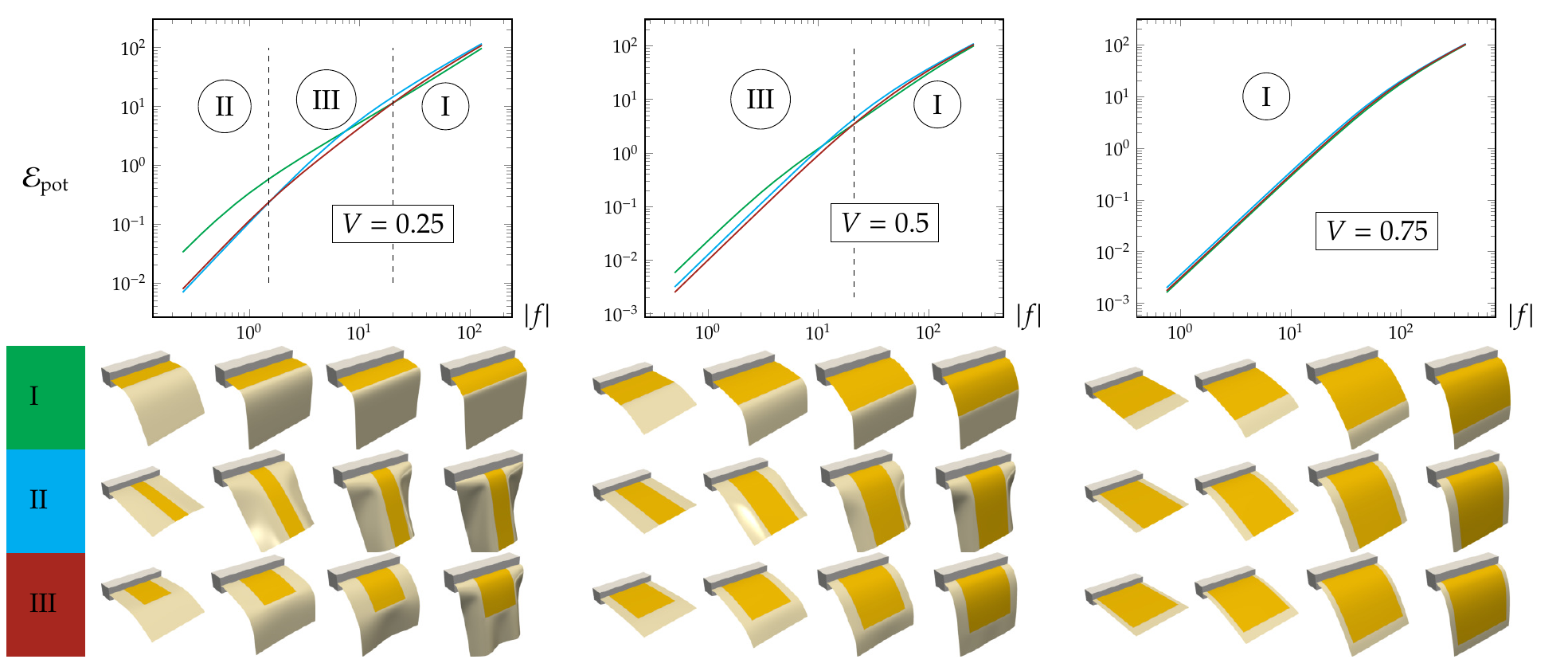}}
}
\caption
[Comparison of three designs for bending isometries.]
{Comparison of the potential energy functional in dependence of $|f|$ in a logarithmic scale 
for three design types and different area fractions $\AreaHardPhase = 0.25,0.5,0.75$ of the hard material.
By dotted lines we separate the ranges, where a specific design is optimal w.r.t. the potential energy.
}
\label{fig:CompareDesignsIsometry}
\end{figure}
Intuitively, it is not surprising that putting in (II) such a thin but extremely
hard `mid rib' ensures that a cylindrically deformed plate will bend down arbitrarily little.
Indeed, even if the
material is very soft elsewhere, the stiff mid rib determines the overall
behaviour. The effect of such a strip can clearly not be emulated by mixing phases along the $x_1$ axis either.
\\
Notice, however, that it is not clear a priori whether it is correct to restrict ourselves to
cylindrical deformations. On one hand, the following numerical simulations suggest
that indeed a hard mid rib leads to a better compliance than a one-dimensional design. On the other hand, they
also suggest that without additional constraints the energy minimising deformation $u : S\to\R^3$
for such a material distribution may not be cylindrical after all. This is
addressed in Section \ref{sec:Developability}.

\section{Kirchhoff plates with nontrivial designs}
\label{sec:nontrivial}

Now, we ask for an optimal design of a plate, where by a  
design we mean an optimal choice of the hardness function $B$. 
We consider this as a specific shape optimisation problem and 
describe optimality of a design $B$ via the 
minimization of the compliance as the most commonly used cost functional. 
In Section \ref{sec:plate} we have shown that there exists under suitable assumptions on the 
boundary conditions a unique minimum in the class of isometric,
cylindrical deformations.
Under suitable symmetry assumptions \eqref{moments}
on the hardness function $B$ and the force $f$ 
this deformation is a stationary point of the plate energy in the 
class of general isometric deformations.
\\
Thus, in what follows we will as well assume \eqref{moments} with the force $f \equiv -e_3$ 
and restrict ourselves to cylindrical deformations.
In this section we do not yet seek the optimal design; instead,
we discuss in the light or our previously discussed numerical findings in Section \ref{sec:numericStateEq} explicit designs and compare them to each other.

For a given design $B\in L^{\infty}(S)$ we denote by $K[B] \in W^{1,2}_l(0, \ell)$
the unique global minimiser (see \cite{HoRuSi16} for existence and uniqueness)
of the functional $\Efree_{\text{phase}}(B,\cdot)$ defined in \eqref{en-4}; as usual 
$\o B(x_1) = \int_{-1/2}^{1/2} B(x_1, x_2)\ dx_2$.
In order to simplify the notation we assume that $\ell = 1$. Then,
for $f\equiv - e_3$, the corresponding state equation reads
$$
(\o B K[B]')' = (1 - t)\cos K[B].
$$
The compliance to be minimised by the optimal design is  defined as  
functional
\begin{equation}
\label{compli-01}
\mathcal{F}(B) =  \int_0^1  u_B(t) \cdot e_3 \ dt =  \int_0^1 (1-t)|\sin K[B](t)|\ dt\,.
\end{equation}
For the reformulation we used integration by parts and the fact that
$K[B] \in (-\pi,0)$ on $(0,1)$.
In what follows, we are at first only interested in the question which of two given
designs $B_I$, $B_\II : S\to\{a, b\}$ leads to a smaller value for
the right-hand side of \eqref{compli-01}, where 
\begin{align*}
B_I(x_1, x_2) &= \chi_{(0, \AreaHardPhase)}(x_1)b + \chi_{(\AreaHardPhase, 1)}(x_1)a,\\
B_\II(x_1, x_2) &= \chi_{(-\AreaHardPhase/2, \AreaHardPhase/2)}(x_2)b + 
\left(1 - \chi_{(-\AreaHardPhase/2, \AreaHardPhase/2)}(x_2)\right)a
\end{align*}
are the two designs already depicted in Figure \ref{fig:CompareDesignsIsometry} for the cases I and II, respectively. It was shown in \cite{HoRuSi16} that
design $B_I$ is the best among all designs which are independent of $x_2$. 
But, the numerical observations in Figure \ref{fig:CompareDesignsIsometry}
suggest that the design $B_\II$ might be better in the general case, which we will now 
verify analytically. To this end, we will first fix some relative area of 
hard phase $\AreaHardPhase > 0$. Then we choose the hardness $b \gg 2/\AreaHardPhase$.
Observe that both $B_I$ and $B_\II$ have the same area $\AreaHardPhase$ of hard phase $b$.

The averaged design $\o B_\II : (0, 1) \to [a, b]$ is given by
\begin{equation}
\label{arithmeticmean}
\o B_\II(x_1) = \int_{-1/2}^{1/2} B_\II(x_1, x_2)\intd x_2
= \AreaHardPhase b + (1 - \AreaHardPhase)a.
\end{equation}
In what follows we will apply the maximum principle to the state equation.
In order to transform the state equation into a boundary value problem, we extend 
$K_I = K[B_I]$ from the interval
$[0, 1]$ to the interval $[0, 2]$ by reflection about $t = 1$,
i.e., we define
$$
K_I(t) = K_I(2 - t) \mbox{ if }t\in [1, 2].
$$
Other functions, such as $B_I$, are extended to $[0, 2]$ in the same way.
With this notation, the original state equation $(B_I K_I')' = (1-t)\cos K_I$
on $(0, 1)$ with mixed Dirichlet-Neumann boundary
conditions $K_I(0) = 0$ and $K_I'(1) = 0$ is equivalent to
the Dirichlet boundary value problem
$(B_I K_I')' = |1-t|\cos K_I\mbox{ on }(0, 2)$ and 
$K_I(0) = K_I(2) = 0$.
We can apply the maximum principle to this semilinear problem, because $K_I$ takes
values in $(-\pi/2, 0]$. 
In what follows we will apply this maximum principle directly
to the equation on $(0, 1)$, without extending it explicitly to $(0, 2)$ every time.
\\
We construct a barrier for the solution $K_I = K[B_I]$ of the state equation
$$
(B_I K_I')' = (1-t)\cos K_I \mbox{ on }(0, 1)
$$
corresponding to the design $B_I$.
\\
Let $\o K_I : (\AreaHardPhase, 1) \to (-\pi/2, 0]$ be the solution of
$
-a\o K_I'' + (1-t)\cos\o K_I = 0 \mbox{ on }(\AreaHardPhase, 1)
$
with boundary conditions $\o K_I(\AreaHardPhase) = 0$ and $\o K_I'(1) = 0$.
We define the barrier function $\h K_I : (0, 1)\to (-\pi/2, 0]$ by setting
$$
\h K_I(t) = \chi_{(\AreaHardPhase, 1)}\o K_I.
$$
Then $\h K_I(0) = K_I(0) = 0$ and $\h K_I'(1) = K_I'(1) = 0$. And
$$
-(B_I \h K_I')' + (1-t)\cos\h K_I \geq 0.
$$
Thus, the maximum principle implies that
$
\h K_I \geq K_I \mbox{ on }(0, 1).
$
Observe that this is true for any choice of $b > a$.
It implies (using that $K_I$, $\h K_I$ take values in $[-\pi/2, 0]$)
$|\sin\h K_I| \leq |\sin K_I| \mbox{ on }(0, 1)$.
Hence the compliances satisfy, as claimed,
$$
\int_0^1 |1-t| |\sin\h K_I(t)|\ dt \leq \int_0^1 |1-t| |\sin K_I(t)|\ dt.
$$
The left-hand side is positive and independent of $b$. It will be denoted by $\varepsilon$.
\\
Next, we consider the design $B_\II$.
Since $\o B_\II$ is constant, $K_\II$ satisfies the state equation
$$
-\o B_\II K_\II'' + |1-t|\cos K_\II = 0 \mbox{ on }(0, 1).
$$
Dividing by $\o B_\II$ and taking absolute values we deduce
$$
|K_\II''| \leq \frac{1}{\o B_\II}.
$$
Since $K_\II'(1) = 0$, this implies
$
|K_\II'(t)| \leq \frac{1}{\o B_\II}.
$
And since $K_\II(0) = 0$, this implies $|K_\II|\leq\frac{1}{\o B_\II}$.
Hence
$$
|\sin K_\II|\leq \sin\frac{1}{\o B_\II}\leq\frac{\varepsilon}{2},
$$
provided $b$ (and thus $\o B_\II$) is large enough.
For such $b$ the compliance of $B_\II$ satisfies
$$
\int_0^1 |1-t| |\sin K_\II(t)|\ dt
\leq \frac{\varepsilon}{2}.
$$
It is therefore strictly better than the compliance of $B_I$.

Summarising, the examples depicted in Figure \ref{fig:CompareDesignsIsometry} show that the optimal design $B : S\to\R$ must depend on $x_2$,
because the design with a stiff horizontal strip has a better compliance than even the best one-dimensional design.
\\
After considering these examples we may ask firstly whether it was legitimate to restrict ourselves
to cylindrical deformations in the analysis.
In fact, the numerical simulations suggest otherwise.
In Section \ref{sec:Developability} we will address this question.
\\
Secondly, we may ask whether the horizontal strip is the best two dimensional design.
In Section \ref{sec:properties} we will derive the Euler-Lagrange equation for the optimal design among cylindrical deformations. 
It turns out that it never consists of a mid rib of hard material of uniform width. Instead, it consists of
a `mid rib' that becomes narrower as the distance from the clamped left boundary increases.

\section{Structure of solutions under symmetry requirement} \label{sec:Developability}
From now on we assume that the symmetry assumption \eqref{moments} holds.
The numerical simulations in Figure \ref{fig:CompareDesignsIsometry} suggest that
stationary points $u$ of $\Efree[B,\cdot]$ enjoy the same mirror symmetry about the $\{x_2 = 0\}$ plane
as $B$ and $f$. But they also suggest that $u$ may not be cylindrical.
Thus, we will provide a rigorous analysis leading to a qualitative
global description of deformations which are symmetric about the $\{x_2 = 0\}$ plane but not necessarily cylindrical.
\\
This description allows us to identify a mild and reasonably natural
additional hypothesis on the `right' part of the boundary
of the plate which ensures that any symmetric deformation satisfying it is in fact cylindrical.
However,  we will see that the numerical simulations display
deformations which are symmetric, which satisfy the additional condition on the right boundary
and which nevertheless at a rather high spatial resolution of the mesh still fail to be cylindrical. 
We will discuss this discrepancy between analysis and numerics.
\medskip

Let us consider the immersion $\t u$ obtained by reflecting an immersion $u$ along the $x_2 = 0$ plane with 
\begin{equation}
\label{sym-0}
u(x) = \t u(\t x)
\end{equation}
where $\t x = (x_1, -x_2)$ and
$$
\t u
=
\begin{pmatrix}
u_1
\\
- u_2
\\
u_3
\end{pmatrix}.
$$
Deformations $u$ satisfying \eqref{sym-0} will be called {\em symmetric} in what follows.
Observe that the deformations depicted in Figure \ref{fig:CompareDesignsIsometry} appear to be symmetric.
Let $u\in W^{2,2}(S, \R^3)$ be a symmetric immersion. Then the following are true
for almost every $x\in S$:
\begin{enumerate}[(i)]
\item We have
$\d_2 u(x) = -\d_2\t u(\t x)$ and $\d_1 u(x) = \d_1\t u(\t x)$.
In particular,
$
\d_2 u(x_1, 0) \ \parallel e_2 \mbox{ for a.e. }x_1.
$
\item For $\a = 1, 2$ we have 
$\d_{\a}\d_{\a} u(x) = \d_{\a}\d_{\a}\t u(\t x)$. Moreover,
$\d_1\d_2 u(x) = -\d_1\d_2\t u(\t x)$.
\item The normal to $u$ satisfies
$
n(x) = \t n(\t x).
$
\item The second fundamental form of $u$
satisfies
$$
A_{\a\a}(x) = A_{\a\a}(\t x)\mbox{ for }\a = 1, 2,\quad  A_{12}(x) = -A_{12}(\t x).
$$
\end{enumerate}
The statements are straightforward for smooth maps and they follow by approximation for $W^{2,2}$ maps.

The following result shows that the level set structure of $\D u$ of symmetric maps
$u\in\SetDeformations_S$ is heavily restricted: it must begin with segments parallel to $e_2$ and
then there follows a (possibly truncated) triangle on which $u$ is affine.
We refer to Figure~\ref{fig:EOCHomMaterialForceLR} for numerical simulation, where such a triangle appears.

\begin{proposition}\label{levelset}
Let $u\in\SetDeformations_S$ be symmetric. Then there exists $m\in [0, \ell]$
such that $u$ is cylindrical on $[0, m]\times (-\frac{1}{2}, \frac{1}{2})$ and
there exists $r\in [0, \frac{1}{2})$ such that $u$ is affine on the convex hull of
$$
\left\{
\left(m, \frac{1}{2}\right),
\left(m, - \frac{1}{2}\right),
\left(\ell, r\right),
\left(\ell, -r\right)
\right\}.
$$
\end{proposition}
\begin{proof}
{\it Claim 1.} If $k\in (0, \ell)$ and $(k, 0)\in S\setminus C_{\D u}$ 
then $u$ is affine on $\{k\}\times I$.
\\

In fact, set $z = (k, 0)$. Since $u$ is affine along $\{0\}\times I$, the segment $[z]$
does not intersect $\{0\}\times I$. Otherwise $z\in C_{\D u}$ because
neighbouring segments would intersect $\{0\}\times I$ as well.
\\
The normal $n$ to $u$ is constant
on $[z]$. Since $n(x) = \t n(\t x)$, it follows that
$n$ is also constant on 
$$
\t{[z]} = \{x\in S : \t x\in [z]\}.
$$
Since $z\in [z]\cap \t{[z]}$, by uniqueness we have $[z] = \t{[z]}$.
This is only possible if $[z]$ is parallel to $e_2$. (The possibility
that $[z]$ is parallel to $e_1$ has been ruled out earlier by observing
that $[z]$ cannot intersect $\{0\}\times I$.)
\\

{\it Claim 2.}
Let $k_1\in [0, \ell)$ and $k_2\in (k_1, \ell]$ and
assume that $\D u$ is constant on each segment $\{k_i\}\times I$, for $i = 1, 2$. Then
$u$ is cylindrical on $[k_1, k_2]\times (-\frac{1}{2}, \frac{1}{2})$.
\\

In fact, otherwise there would exist a $k\in (k_1, k_2)$ such that
$(\{k\}\times I)\setminus C_{\D u}$ contains a point $z$.
Since $u$ is affine on the segments $\{k_i\}\times I$, we see that $[z]$ cannot intersect
$\{k_1, k_2\}\times I$. Hence there exists $k'\in (k_1, k_2)$ such that
$(k', 0)\in [z]$. But then Claim 1 implies that $\D u$ is constant on $\{k'\}\times I$.
So by uniqueness we must have $[z] = \{k'\}\times I$ and therefore $k = k'$.
This completes the proof of Claim 2.
\\

Since $u$ is affine on $\{0\}\times I$ due to the boundary conditions,
combining Claim 2 with Claim 1 we see that $u$ is cylindrical
on $[0, k]\times I$ whenever $(k, 0)\in S\setminus C_{\D u}$.
\\
Therefore, if 
$$
m = \sup\{m'\in (0, \ell) : (m', 0)\in S\setminus C_{\D u}\},
$$
then $u$ is cylindrical on $[0, m]\times I$
(so let us assume that $m < \ell$ since otherwise there is nothing left to prove),
and $(m, \ell)\times\{0\}$ is contained in $C_{\D u}$.
By connectedness, it is contained in a single connected component $U$ of $C_{\D u}$.
Clearly $\{m\}\times I\subset \d U$.
\\
Now let $z\in S\cap\d U$ be such that $z_1 > m$ and $z_2 > 0$. Then $[z]\subset\d U$ 
otherwise [z] would intersect $C_{\D u}$ (cf. \cite{Ho11a}) and $[z]$ can neither intersect
$\{m\}\times I$ (by uniqueness) nor $(m, \ell)\times\{0\}$ (because this set is contained in $C_{\D u}$).
Hence there exists $m' \in [m, \ell)$ and $r\in [0, \frac{1}{2}]$ such that $[z]$ is the segment
with endpoints $(m', \frac{1}{2})$ and $(\ell, r)$. By symmetry also the segment
with endpoints $(m' , -\frac{1}{2})$ and $(\ell, -r)$ is contained in $\d U$.
\\
Furthermore,  $C_{\D u}$ is convex because for $\t z\in S\setminus C_{\D u}$  line through $\t z$ 
can not intersect $C_{\D u}$ twice.
(see e.g. \cite[Section 3.2]{Ho11a}), we conclude that $U$ is the convex hull
of the points $(m, \pm 1/2)$, $(m', \pm 1/2)$ and $(\pm r, \ell)$. The claim follows from this,
because on $(m, m')\times I$ the map $\D u$ is constant, hence $u$ is cylindrical on all of
$[0, m']\times I$.
\end{proof}

The following result asserts that every symmetric deformation $u$ satisfying a
mild additional condition is in fact cylindrical. 
This additional conditions is that $u$ must not bend the `right' boundary.

\begin{corollary}\label{cor1}
Let $u\in \SetDeformations_S$
be symmetric and assume that
\begin{equation}
\label{derecha}
\Big|u\left(\ell, \frac{1}{2}\right) - u\left(\ell, -\frac{1}{2}\right)\Big| = 1.
\end{equation}
Then $u$ is cylindrical.
\end{corollary}
\begin{proof}
Together with the isometry of $u$, the hypotheses imply that $u$ is affine on 
$\{0\}\times I$ and on $\{\ell\}\times I$.
Hence Claim 2 in the proof of Proposition \ref{levelset} shows that $u$ is cylindrical
on $S$.
\end{proof}
In Figure~\ref{fig:EOCHomMaterialForceLR} we compute for a homogeneous material distribution $B = 1$ and $l=1$ the solution to the state equation w.r.t. the force 
\begin{align*}
 f = 
   \chi_{[0.9,1] \times [0,0.1]} \begin{pmatrix}   0 \\ 50 \\ 1 \end{pmatrix}
 + \chi_{[0.9,1] \times [0.9,1]} \begin{pmatrix}   0 \\ -50 \\ 1 \end{pmatrix} \, .
\end{align*}
\begin{figure}[!htbp]
\resizebox{\textwidth}{!}{
\begin{tabular}{  c  c  c  c  }
  deformed config. & approx. affine region & image of Gauss map & discrete Gauss curvature $\kappa_h$ \\ 
   \begin{minipage}{0.3\textwidth} \centering 
     \includegraphics[max width=0.9\linewidth, max height=0.9\linewidth]{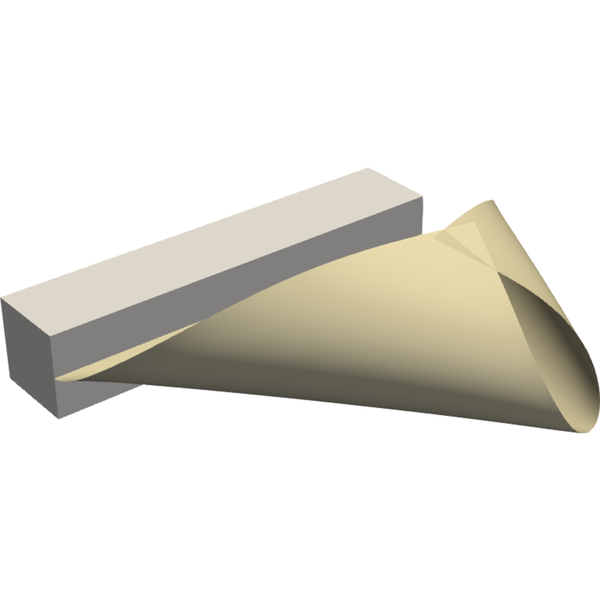}
    \end{minipage}
  & \begin{minipage}{0.3\textwidth} \centering 
     \includegraphics[max width=0.9\linewidth, max height=0.9\linewidth]{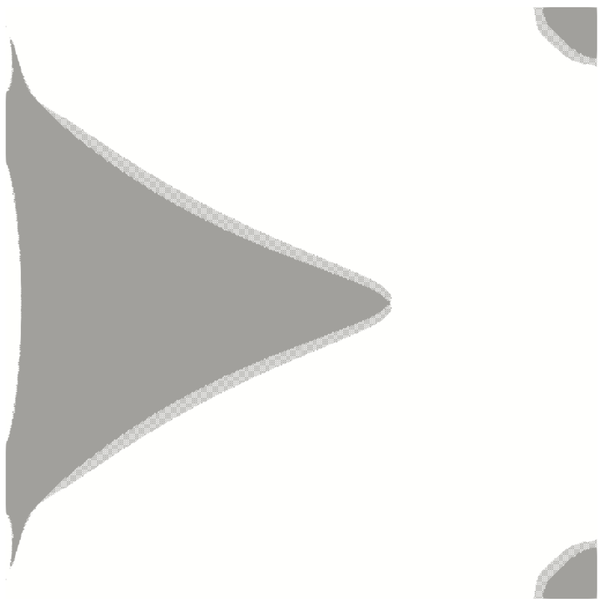}
    \end{minipage}
   & \begin{minipage}{0.3\textwidth} \centering 
     \includegraphics[max width=0.9\linewidth, max height=0.9\linewidth]{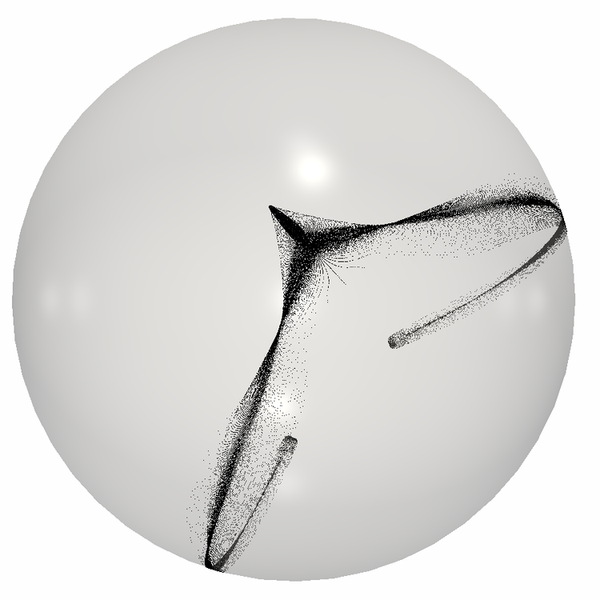}
    \end{minipage}
     & \begin{minipage}{0.3\textwidth} \centering 
     \includegraphics[max width=0.9\linewidth, max height=0.9\linewidth]{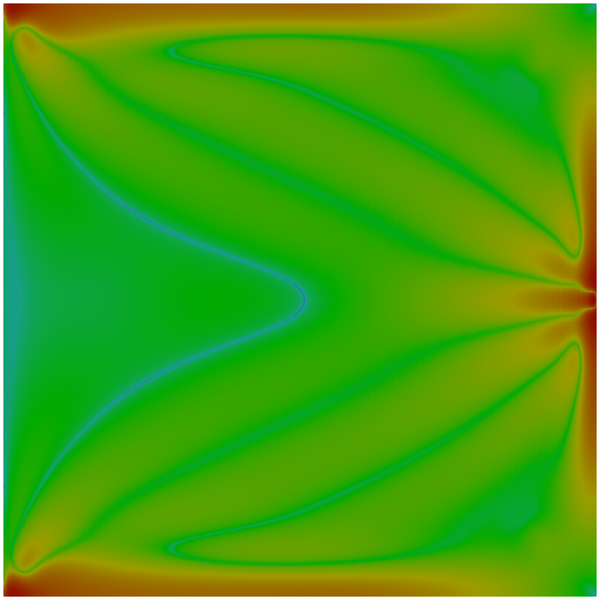}
    \end{minipage}
    \begin{minipage}{0.1\textwidth}
     \begin{tikzpicture}
        \node at (0, 0){\includegraphics[width=0.1\textwidth,height=2.5\textwidth]{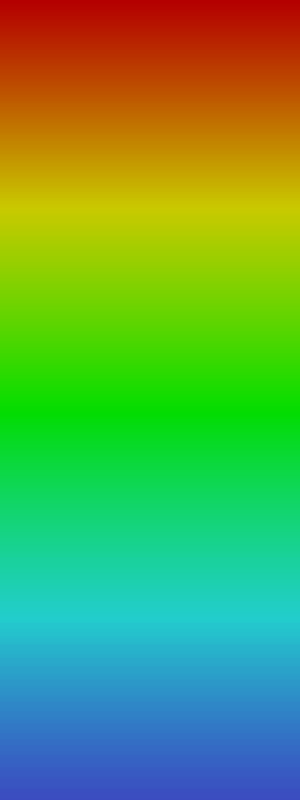}};
        \node at (0.8,2.0) {\tiny $7.2 \cdot 10^1$};
        \node at (0.8,0.) {\tiny $5.8 \cdot 10^{-2}$};
        \node at (0.8,-2.0) {\tiny $4.4 \cdot 10^{-5}$};
     \end{tikzpicture}
    \end{minipage}

  \\ 
\end{tabular}
}

\vspace*{0.3cm}

\resizebox{\textwidth}{!}{
\begin{tabular}{ | c | c | c | c | c | c | c | c | c | c |  c |   }
\hline
$h$        & $|\D u^T \D u - \Id |_{L^1}$ & EOC                     & $|\kappa_h |_{L^1}$ & EOC & $| \nabla \nabla_h (u_h - u_{h_{\text{min}}})|_{L^2}$ & EOC \\ \hline
0.0441942  & 0.0638301           &                         &  3.99201 &          & 36.6783 &  \\ \hline
0.0220971  & 0.028775            & 1.14942                 &  2.93039 & 0.446022 & 18.6961 & 0.972193 \\ \hline
0.0110485  & 0.0124149           & 1.21274                 &  2.08379 & 0.491881 & 9.66444 & 0.951975 \\ \hline
0.00552427 & 0.00506241          & 1.29418                 &  1.48302 & 0.490671 & 4.86258 & 0.99064 \\ \hline
0.00276214 & 0.00215124          & 1.23465                 &  1.06876 & 0.472599 & 2.23702 & 1.12014 \\ \hline
0.00138107 & 0.000962264         & 1.16067                 &  0.76414 & 0.48396  &         & \\ \hline
\end{tabular}
}
\caption{Expected order of convergence for the isometry error in $L^1$, the discrete Gauss curvature $\kappa_h$ in $L^1$ and the second derivative of $u_h$ in $L^2$.}
\label{fig:EOCHomMaterialForceLR}
\end{figure}

In Fig. \ref{fig:EOCHomMaterialForceLR} we take into account a sequence of successively refined measures with uniform grid sizes $h \in \{ 0.0441942, 0.0220971, 0.0110485, 0.00552427, 0.00276214, 0.00138107 =: h_{\text{min}} \}$ 
to uniformly discretize the reference configuration $S = [0,1]^2$.
Then we compute the expected order of convergence (EOC) for the isometry error in $L^1$ and the second derivative of the deformation in $L^2$ consider the finest mesh as ground truth.
From \cite{Ba11} it is proven that both converge linearly in $h$, which is verified by our result.

To identify numerically the affine region, we compute for each triangle element the variance of the normal vector and apply a threshold with $10^{-9}$.
Furthermore, we plot the image of the Gauss map and observe as expected 
two nearly one dimensional arms meeting at the normal of the approximately affine region on the sphere. 
We refer to \cite{Po73} for the analytical result that the deformed configuration is developable if and only if the image of the Gauss map is singular.
Here for each vertex we plot a small dot on the sphere indicating the discrete normal.

Finally, we observe an EOC for the discrete Gauss curvature 
$\kappa_h = \det( \nabla \nabla_h u_h \cdot n_h)$ 
of approximately $\frac{1}{2}$  in $L^1$. This low order of convergence in a weak norm
appears to be too weak to prevent curvature singularities to develop on already fairly fine meshes.
In fact on a logarithmic scale we see singularities at the corners of the clamped boundary and in the middle of the opposite boundary. This singular behaviour might deteriorate the analytically expected
cylinderical solution structure. 
For a affine lateral boundary, in our numerical simulations, we obtain deformations
which are not cylindrical as shown in Fig.~\ref{fig:CompareStiffRightBoundary}.
Here, we use a layer of hard material orthogonal to the clamped boundary, which we have already considered in Figure~\ref{fig:CompareDesignsIsometry} (Design II).
We compute solutions of the state equation for different forces and plot the Gauss curvature.
As in Figure~\ref{fig:EOCHomMaterialForceLR}, we observe an EOC significantly less than $1$.
\begin{figure}[!htbp]
\resizebox{\textwidth}{!}{
\begin{tabular}{  m{1.6cm} c c c c c }
$u_h(S)$
  & \begin{minipage}{0.2\textwidth} \centering 
     \includegraphics[max width=0.9\linewidth, max height=0.9\linewidth]{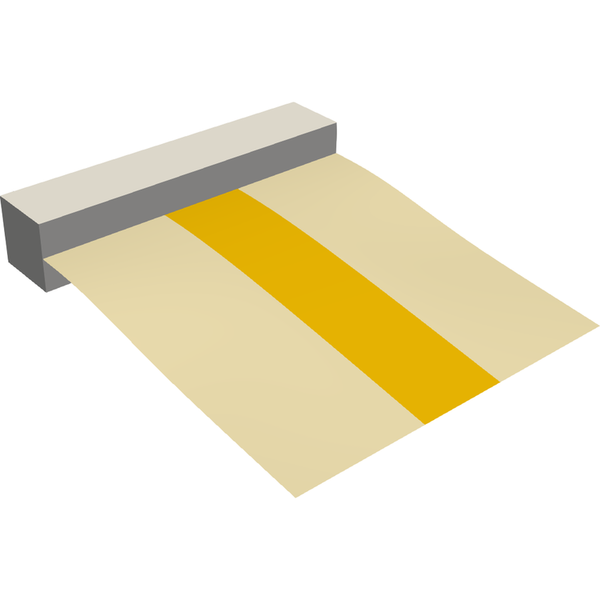}
    \end{minipage}
  & \begin{minipage}{0.2\textwidth} \centering 
     \includegraphics[max width=0.9\linewidth, max height=0.9\linewidth]{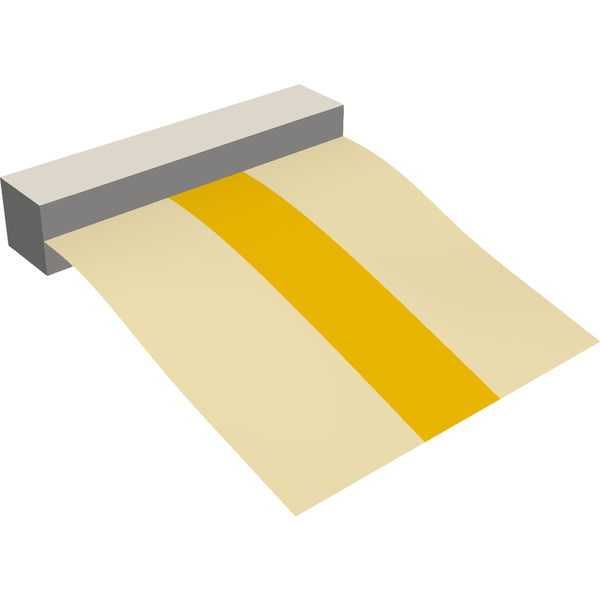}
    \end{minipage}
  & \begin{minipage}{0.2\textwidth} \centering 
     \includegraphics[max width=0.9\linewidth, max height=0.9\linewidth]{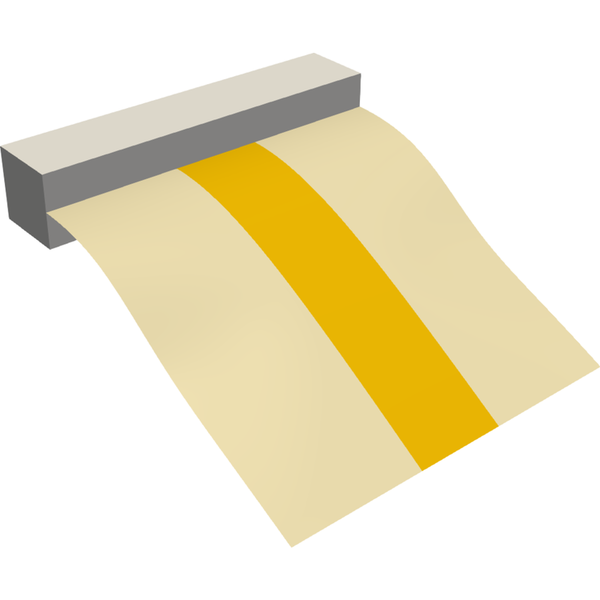}
    \end{minipage}
  & \begin{minipage}{0.2\textwidth} \centering 
     \includegraphics[max width=0.9\linewidth, max height=0.9\linewidth]{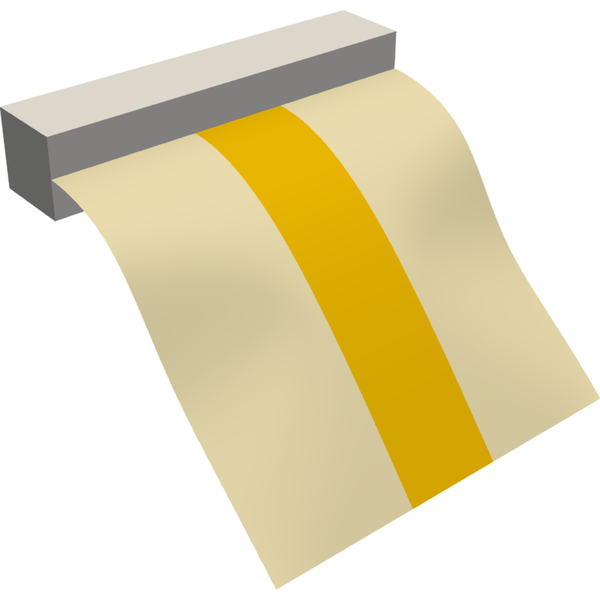}
    \end{minipage}
  & \\
    $\kappa_h$
  & \begin{minipage}{0.2\textwidth} \centering 
     \includegraphics[max width=0.9\linewidth, max height=0.9\linewidth]{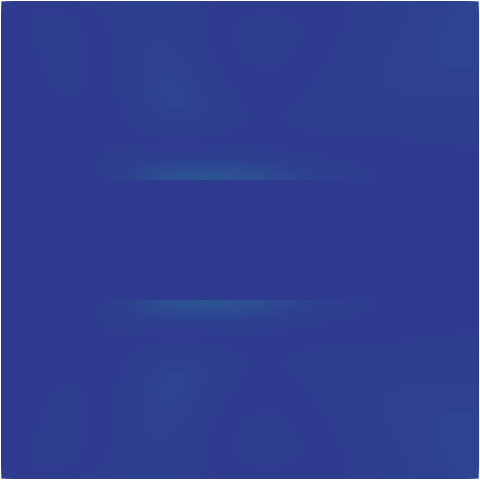}
    \end{minipage}
  & \begin{minipage}{0.2\textwidth} \centering 
     \includegraphics[max width=0.9\linewidth, max height=0.9\linewidth]{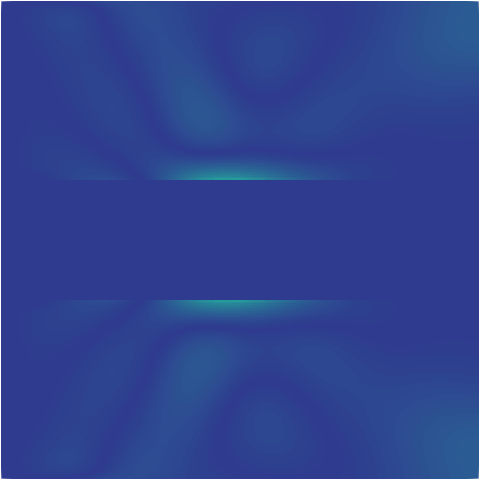}
    \end{minipage}
  & \begin{minipage}{0.2\textwidth} \centering 
     \includegraphics[max width=0.9\linewidth, max height=0.9\linewidth]{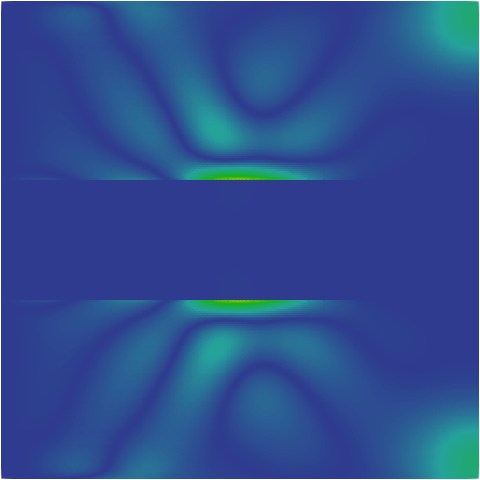}
    \end{minipage}
  & \begin{minipage}{0.2\textwidth} \centering 
     \includegraphics[max width=0.9\linewidth, max height=0.9\linewidth]{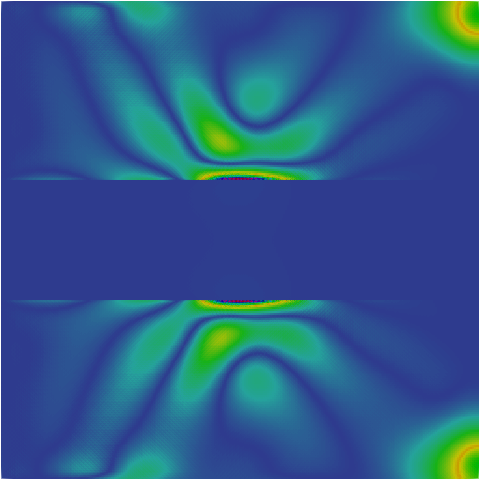}
    \end{minipage}
  & \begin{minipage}{0.1\textwidth}
       \begin{tikzpicture}
        \node at (0, 0){\includegraphics[width=0.1\textwidth,height=1.5\textwidth]{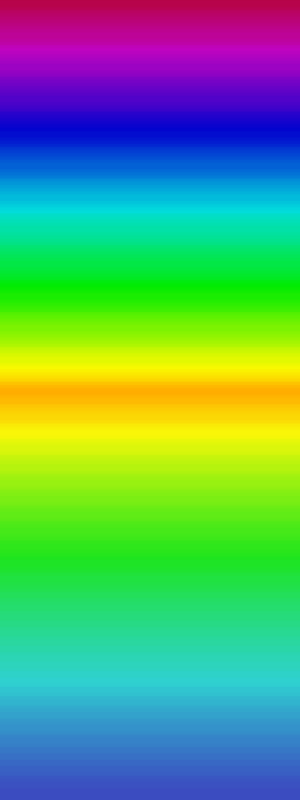}};
        \node at (0.5,1.15) {\tiny $2.3$};
        \node at (0.5,-0.15) {\tiny $1$};
        \node at (0.5,-1.15) {\tiny $0$};
     \end{tikzpicture}
    \end{minipage}
  \\
\end{tabular}
}

\vspace*{0.3cm}

\resizebox{\textwidth}{!}{
\begin{tabular}{ | c | c | c | c | c | c | c | }
\hline
$h$        & $|\D u^T \D u - \Id |_{L^1}$ & EOC                     & $|\kappa_h |_{L^1}$ & EOC & $| \nabla \nabla_h (u_h - u_{h_{\text{min}}})|_{L^2}$ & EOC \\ \hline
0.0441942  & 0.0273584           &          & 0.792172                                        &          & 8.06378                                              & \\ \hline
0.0220971  & 0.0145798           & 0.908017 & 0.553654                                        & 0.516828 & 3.99447                                              & 1.01345 \\ \hline
0.0110485  & 0.00560266          & 1.37978  & 0.328346                                        & 0.753768 & 1.70073                                              & 1.23185 \\ \hline
0.00552427 & 0.00209298          & 1.42055  & 0.207958                                        & 0.658927 & 0.727473                                             & 1.22519 \\ \hline
0.00276214 & 0.000783029         & 1.41843  & 0.131085                                        & 0.665787 & 0.291338                                             & 1.3202 \\ \hline
0.00138107 & 0.000280999         & 1.4785   & 0.079705                                        & 0.717759 &                                                      & \\ \hline
\end{tabular}
}
 \caption{Top: Deformations for Design II with stiff right boundary for different forces ($|f| = 2,4,8,16$) and Gauss curvature.
 Bottom: Expected order of convergence of $|f| = 16$.}
 \label{fig:CompareStiffRightBoundary}
\end{figure}

\section{Optimal design problem}\label{sec:Disenyo}

In this section we seek optimal designs, i.e., material distributions $B$
which optimise the compliance functional. We are only interested in gravitational
force, i.e., the force $f = -e_3$.
\\
Moreover, in the analytical results, we will seek optimal designs $B$ among those the first moment
$\int x_2B(x_1, x_2)\ dx_2$ is identically zero. This restriction
seems justified in view of the symmetry of the force and the boundary conditions.
In particular, it is satisfied by the optimal designs found numerically,
cf. Figure \ref{fig:OptDesignsIsometryF}.
\\
Let $c_l > 0$. For given $B\in L^{\infty}(S; [a, b])$ 
we define $\theta_B : S\to [0, 1]$ via
\begin{equation}
\theta_B = \frac{B - a}{b-a}
\end{equation}
and we denote by $K[B]$ the unique global minimiser
of the functional $\Efree_{\text{phase}}(B,\cdot)$ within the space $W^{1,2}_l(0, \ell)$.
As already mentioned earlier (cf. also \cite{HoRuSi16}) this function $K[B]$
is the unique solution of \eqref{ELK-1}, which for the present force reads
\begin{equation}\label{state}
(\o BK[B]')' = (1 - t)\cos K[B],
\end{equation}
and which satisfies \eqref{statecon}.
\\
The cylindrical deformation
$u_B\in\SetDeformations_S$ with phase $K[B]$ is the unique global 
minimiser of $\Efree[B,\cdot]$ within the class of cylindrical deformations in 
$\SetDeformations_S$. By Corollary \ref{excyl} the deformation $u_B$ is
a stationary point of $\Efree[B,\cdot]$ in $\SetDeformations_S$.
\\
We seek to minimise the compliance
\begin{equation}
\label{compli-1}
B\mapsto -\int_S u_B\cdot e_3 + c_l\int_S\theta_B
\end{equation}
among all $B\in L^{\infty}(S; \{a, b\})$. 
\\
As a side remark, note that, in view of the results in Section \ref{sec:Developability},
instead of imposing the a priori condition \eqref{moments} on
the admissible designs $B$, we could restrict ourselves to symmetric deformations $u$ which satisfy the additional
condition \eqref{derecha}. Then Corollary \ref{cor1} implies that in fact we are restricting 
ourselves to cylindrical deformations -- and therefore the $x_2$-dependence of $B$ becomes irrelevant (since $u$
is cylindrical only $\o B$ plays a role) and therefore we could as well have assumed 
from the outset that \eqref{moments} is satisfied, by symmetrising $B$.

In order to simplify the notation, from now on we normalise the length instead of the width $w > 0$ of $S$, i.e.,
we assume that $S = (0, 1) \times I$, where $I = (-w/2, w/2)$. 
In view of the previous considerations, from now on only cylindrical deformations $u : S\to\R^3$ will be considered.
and in what follows we will frequently write $t$ instead of $x_1$.
\\
In terms of the phase $K[B]$ the right-hand side of \eqref{compli-1} equals
$$
B\mapsto -\int_0^1 (1-t)\sin K[B](t)\ dt + c_l\int_0^1\theta_B
=
-\int_0^1 (1-t)\sin K[\o B](t)\ dt + c_l\int_0^1\theta_{\o B}.
$$
Observe that the `physical' design $B : S \to\R$ can only take the values $a$ or $b$.
However, the corresponding one-dimensional design $\o B : (0, 1)\to\R$ can take all
values in $[a, b]$.
\\
Clearly the map $B\mapsto \o B$ is highly non-injective and therefore the optimal two-dimensional
design $B$ will be non-unique: all designs $B$ with the same average $\o B$ have the
same compliance if we define $u_B$ as above.
\\
In view of the above, for $K\in W^{1,2}_l(0, 1)$ and $\theta\in L^{\infty}(0, 1)$
let us define the cost functional
$$
\costFct[K, \theta] = -\int_0^1 (1-t)\sin K + c_l \int_0^1 \theta
$$
consisting of the compliance cost and a cost for hard material.
\medskip 

There is fundamental difference between Beams studied in \cite{HoRuSi16} 
and cylindrical plates. In fact, the admissible asymptotic designs for one dimensional beams
are those arising from designs $B : (0, 1)\to\R$ of the form
$$
B(t) = (1 - \chi(t))a + \chi(t)b,
$$
where $\chi$ only takes values in $\{0, 1\}$.
The possible asymptotic designs $B^*$ for beams arise
as harmonic weak-$*$ limits of sequences of such $B_n$, i.e.,
$$
\frac{1}{B_n}\weaks\frac{1}{B^*}\mbox{ in }L^{\infty}(0,1).
$$
Since $\chi_n\in\{0,1\}$ pointwise, for $B_n = (1-\chi_n)a + \chi_n b$ we have
$$
\frac{1}{B_n} = \frac{1-\chi_n}{a} + \frac{\chi_n}{b}.
$$
Therefore, denoting by $\theta$ the weak-$*$ limit of $\chi_n$
one obtains
$$
\frac{1}{B^*} = \frac{1-\theta}{a} + \frac{\theta}{b}.
$$
Observe that clearly $\int_I \chi_n\to\int_I\theta$.
\\
The two-dimensional situation addressed here
differs from this one-dimensional in that the
designs $B_n$ must be replaced by their averages $\o B_n$.
And each $\o B_n$ can take values in the whole interval $[a, b]$.
Therefore, the Young measure generated by $\o B_n$ is supported on $[a, b]$,
whereas in \cite{HoRuSi16} the Young measure generated by $B_n$ was supported on $\{a, b\}$.
\\
Since apart from this difference the problem studied here is one dimensional
as well, we will henceforth write $B$ instead of $\o B$ to simplify the notation.
So $B : (0, 1)\to [a, b]$ can indeed take values between $a$ and $b$ as well and study this relaxed problem.

For $\theta\in L^{\infty}((0,1); [0, 1])$ we define
$B_{\theta} = (1 - \theta)a + \theta b$
and consider
$$
\h \costFct[\theta] =  \costFct[ K[B_{\theta}], \theta].
$$
A function $\theta\in L^{\infty}((0, 1), [0, 1])$ will be called an optimal
design if it is a minimiser of $\h \costFct$. 
The following existence result for optimal designs hold.
\begin{theorem}[existence of optimal designs]\label{existence}
The functional $\h \costFct : L^{\infty}((0,1); [0, 1])\to\R$ attains its minimum.
\end{theorem}
\begin{proof}
Let $(\theta_n)$ be a minimising sequence and set $B_n = B_{\theta_n}$
and $K_n = K[B_n]$. As in \cite{HoRuSi16} we see that $K_n$ converges to $K[B^*]$
weakly in $W^{1,2}$, where
$B^*$ is defined by
$$
\frac{1}{B_n}\weaks\frac{1}{B^*}\mbox{ in }L^{\infty}(0, 1).
$$
Here we have passed to a subsequence, which
we do not relabel. Notice that
\begin{equation}
\label{relax-0}
\int_0^1\theta_n\to\int_0^1\theta\mbox{ as }n\to\infty.
\end{equation}
Hence
\begin{equation}
\label{relax-2}
\inf\h \costFct = \costFct[K[B^*], \theta].
\end{equation}
There exists a unique $\theta^*\in L^{\infty}((0, 1); [0, 1])$ with $B_{\theta^*} = B^*$.
We claim that 
\begin{equation}
\label{relax-1}
\int_0^1\theta^*\leq\int_0^1\theta.
\end{equation}
To prove \eqref{relax-1} we note that $\Phi(z) = z^{-1}$ is convex on $(0, \infty)$. Hence
by weak lower semicontinuity and by \eqref{relax-0} and recalling the
definition of $\theta^*$, we see that
\begin{align*}
\int_0^1(a + \theta^*(b-a)) &=
\int_0^1 B^* = \int_0^1 \Phi\left( \frac{1}{B^*} \right)
\\
&\leq\liminf_{n\to\infty}\int_0^1\Phi\left( \frac{1}{B_n} \right) = \liminf_{n\to\infty}\int_0^1 B_n
= \int_0^1(a + \theta (b - a)).
\end{align*}
Hence \eqref{relax-1} follows.
Now we deduce from \eqref{relax-2} that
$$
\inf\h \costFct \geq \costFct[K[B^*], \theta^*] = \h \costFct[\theta^*].
$$
So $\theta^*$ is the sought-for minimiser.
\end{proof}
\section{Properties of the optimal design} \label{sec:properties}
To further study the optimal design, we will compute the derivative of the cost functional $\h\costFct$ with respect to design $\theta$.
To this end, we first define as usual the unique solution $P\in W^{1,2}_l(0,1)$ of the adjoint equation
\begin{equation}
\label{adjoint}
(B_{\theta}P')' = (1-t)\cos K - (1-t)P\sin K.
\end{equation}
 in the dual space of $W^{1,2}_l(0, 1)$.
As before, this includes the weak formulation of the boundary
conditions $P(0) = P'(1) = 0$.

As in \cite{HoRuSi16} we introduce $p = B_{\theta}P'$ and write \eqref{adjoint} as
\begin{equation}
\label{p'}
p' = -(1-t)(\sin K)(P - \cot K).
\end{equation}
The right hand side is well-defined because we know that $K(t) \in (-\tfrac{\pi}{2},0]$ and $K(t)=0$ if and only if $t=0$.
The following lemma extends \cite[Proposition 5.5]{HoRuSi16}.
\begin{lemma}\label{lemon2}
The adjoint variable $P$ satisfies $P < 0$ on $(0, 1)$ and there exists $\tau\in (0, 1]$
such that $p' > 0$ on $(0, \tau)$ and $p'< 0$ on $(\tau, 1)$.
We have $\tau = 1$ if and only if $P(1)\geq\cot K(1)$.
In this case $(\tau, 1) = \emptyset$ and so $p > 0$ on all of $[0, 1)$.
\\
If $\tau < 1$ then there exists $\tau_0\in (0, \tau)$
such that $p < 0$ on $[0, \tau_0)$ and $p > 0$ on $(\tau_0, 1)$.
\end{lemma}
\begin{proof}
Since $K\in (-\frac{\pi}{2}, 0]$ on $[0, 1)$
we have $\sin K \leq 0$.
Hence the adjoint equation \eqref{adjoint} for $P$ is
$$
-(B_{\theta}P')' + (1 - t) |\sin K| P = -(1-t)\cos K.
$$
The right-hand side is negative and the zeroth order coefficient is positive. Hence the 
strong maximum principle (using once again the extension approach onto the interval $[0,2]$) 
implies that $P$ does not attain a
nonnegative local maximum in $(0, 2)$. Since $P(0) =P(2) = 0$,
we conclude that indeed $P < 0$ on $(0, 2)$.
For convenience we include the following argument from \cite{HoRuSi16}.
Set 
$$
\tau = \inf\{t\in (0, 1) : P(t) \leq \cot K(t)\},
$$
with $\tau = 1$ if the set on the right-hand side is empty.
\\
Let us assume that it is nonempty. Then we have $\tau\in (0, 1)$;
indeed $\tau > 0$ because
$P(0) = 0$ while $\cot K(0) = -\infty$.
Hence by continuity $P(\tau) = \cot K(\tau)$.
\\
By definition of $\tau$ we have $P > \cot K$ on $(0, \tau)$.
Hence by \eqref{p'} we have $p' > 0$ on $(0, \tau)$.
\\
Now, $(\cot K)' = - \tfrac{K'}{\sin ^2 K}$ implies 
\begin{align*}
- \left( B_{\theta}\left( P -\cot K \right)' \right)'  
& =  (1-t) \sin K (P-\cot K) - \left( \frac{B_{\theta}K'}{\sin^2 K}\right)' \\
& =   (1-t) \sin K (P-\cot K)  - (1-t) \frac{\cos K}{\sin^2 K} + \frac{2 B_{\theta} (K')^2 \cos K}{\sin^3 K}
\end{align*}
Taking into account that $K\in (-\tfrac{\pi}{2}, 0)$ on $(0,1)$ we observe that 
$P - \cot K$ satisfies the differential inequality
$$
-\left( B_{\theta}\left( P -\cot K \right)' \right)' -  (1-t) \sin K (P-\cot K) < 0
$$
on $(0,1)$. And $(1-t)\sin K < 0$ on $(0,1)$. 
An application of the strong maximum principle,
again via extension to $(0,2)$, implies that $P < \cot K$ on $(\tau, 2-\tau)$,
hence on $(\tau, 1)$. Thus $p' < 0$ on $(\tau, 1)$ by \eqref{p'}. 
In particular, since $p(1) = 0$, we have $p > 0$ on $[\tau, 1)$.
\\
On the other hand, by definition of $\tau$ we have $p' > 0$ on $(0, \tau)$.
If $p(0) \geq 0$ then $p > 0$ on $(0, \tau)$, which in turn would imply
that $P > 0$ on $(0, \tau)$, contradicting our earlier observation that $P$ is negative.
Therefore, $p(0) < 0$.
\\
The existence of $\tau_0\in (0, \tau)$ with
the claimed properties now follows from the intermediate value theorem
and from the strict monotonicity of $p$ on $(0, \tau)$ and on $(\tau, 1)$.
\end{proof}

Next, we compute the derivative of the cost functional $\h \costFct[\theta]$.
For given $\beta\in L^{\infty}(0, 1)$ denote by $\FDK = \partial K[B_\theta] (\beta)$
the Fr\'echet derivative of $B \mapsto K[B]$ at the point $B_{\theta}$ in direction $\beta$.
Considering the variation of the state equation with respect to $\theta$ we find
\begin{equation}
\label{Kdot}
(B_{\theta}\FDK')' + (1-t)\sin K \, \FDK = - (\beta K')'.
\end{equation}
Testing \eqref{Kdot} with the adjoint variable $P$ we find
\begin{equation}
\label{deriva-1}
\int_0^1 B_{\theta}P'\FDK' - (1-t)P \FDK\sin K = - \int_0^1 \beta P'K'.
\end{equation} 
On the other hand, testing the adjoint equation \eqref{adjoint} with $\FDK$ we get
\begin{equation}
\label{deriva-2}
\int_0^1 B_{\theta}P'\FDK' - (1-t)(\sin K)\FDK P = - \int_0^1 (1-t)\FDK\cos K.
\end{equation}
Comparing \eqref{deriva-1} and \eqref{deriva-2} we see that
$$
\int_0^1 (1-t)\FDK\cos K =  \int_0^1 \beta P'K'.
$$
We can use these computations to compute the derivative of $\h \costFct$ with respect to $\theta$.
We observe that the derivative of $\theta\mapsto B_{\theta}$
at the point $\theta$ is clearly $\partial B_{\theta} = b - a$ (continuing to denote Fr\'echet derivatives by a $\partial$) 
and conclude 
\begin{align*}
\partial \h \costFct[\theta](\beta) &=
- \int_0^1 (1-t)\FDK \cos K  + c_l\int_0^1\beta = - \int_0^1 \beta P'K'  + c_l\int_0^1\beta.
\end{align*}
For more details in the above computations we refer to \cite{HoRuSi16}.
Summarising,
\begin{equation}
\label{deriva-3}
\partial \h \costFct[\theta] = - K'P' + c_l.
\end{equation}

\begin{proposition}
If $\theta$ is an optimal design, then
\begin{equation}
\label{el-1}
B_{\theta}^2 K'P'
\begin{cases}
\leq c_l a^2&\mbox{ on }\{\theta = 0\}
\\
\geq  c_l b^2 &\mbox{ on }\{\theta = 1\}
\\
= c_l B^2_{\theta}&\mbox{ on }\{\theta\in (0, 1)\}.
\end{cases}
\end{equation}
\end{proposition}
\begin{proof}
The Euler-Lagrange equation is
$\partial \h \costFct[\theta](\beta)\geq 0$
for all $\beta\in L^{\infty}(I)$ satisfying
$\beta \geq 0$ on $\{\theta = 0\}$ and $\beta \leq 0$ on $\{\theta = 1\}$.
In view of \eqref{deriva-3} this leads to the following pointwise conditions:
$$
K'P' 
\begin{cases}
\leq c_l&\mbox{ on }\{\theta = 0\}
\\
\geq c_l&\mbox{ on }\{\theta = 1\}
\\
= c_l&\mbox{ on }\{\theta\in (0, 1)\}.
\end{cases}
$$
We multiply both sides by $B^2_{\theta}$ to deduce \eqref{el-1},
because $B_{0} \equiv a$ and $B_1 \equiv b$. 
\end{proof}

Finally, the following theorem identifies some features of the optimal designs.

\begin{theorem}[characterization of optimal designs]\label{optimal}
Every optimal design $\theta$ is continuous on $[0, 1]$,
and either $\theta \equiv 0$ or
there are $0 \leq t_0 < t_1 < 1$ such that
\begin{itemize}
\item if $t_0 > 0$ then $\theta = 1$ on $(0, t_0)$,
\item $\theta$ is nonzero and strictly decreasing
on $(t_0, t_1)$ and 
\item $\theta = 0$ on $[t_1, 1]$.
\end{itemize}
Moreover, $\theta\in C^{\infty}(t_0, t_1)$.
\end{theorem}
\begin{proof}
According to Lemma \ref{lemon2} there is $\tau\in (0, 1]$ such that
$p = B_{\theta}P' \geq 0$ on $[\tau, 1]$ (when $\tau = 1$ then we have $p(\tau) = 0$
due to the boundary condition).
\\
Hence, setting $k = B_{\theta}K'$, we have
$B_{\theta}^2K'P' = kp \leq 0$ on $[\tau, 1]$.
On $[0, \tau)$ the continuously differentiable function $kp$
is positive and strictly decreasing,
because here both $k$ and $p$ are negative and strictly increasing.
Define 
$$
t_1 = \inf\left\{t\in [0, \ell] : (kp)(t) \leq c_la^2\right\}.
$$
Then $t_1 < \tau$ because $kp\leq 0 < c_l a^2$ on $[\tau, 1]$
and $kp$ is continuous.
On the other hand, unless $\theta = 0$ almost everywhere
(in which case we are done, so we exclude this from now on),
we must have $t_1 > 0$. In particular,
$$
(kp)(t_1) = c_l a^2.
$$
Since $t_1\in (0, \tau)$ and since
$kp$ is strictly decreasing on $(0, \tau)$ we conclude that
$\{\theta = 0\} = [t_1, 1]$. 
\\
On $[0, t_1]$ the function $kp$ decreases strictly from $(kp)(0)$
to $c_l a^2$. Set
$$
t_0 = \sup\left\{t\in [0, \ell] : (kp)(t) > c_l b^2\right\},
$$
and set $t_0 := 0$ if the set on the right-hand side is empty.
By monotonicity
$$
kp > c_l b^2 \mbox{ on }(0, t_0).
$$
Hence $\theta = 1$ on this set, by \eqref{el-1}.
\\
Since $c_lb^2$ is strictly greater than $c_la^2$ and since 
$kp$ is continuous
and strictly monotone on $[0, t_1]$, we have
$$
kp\in \left(c_l a^2, c_l b^2 \right)
\mbox{ on }(t_0, t_1).
$$
Hence \eqref{el-1} implies that $\theta \in (0,1)$  on $(t_0, t_1)$ and that
\begin{equation}
\label{1407-1}
B_{\theta}^2 = \frac{1}{c_l}\cdot kp\mbox{ on }(t_0, t_1).
\end{equation}
In view of this,
the monotonicity and continuity of $\theta$ follow from the same properties for $kp$,
because $\theta \mapsto B_{\theta}$ is linear and strictly increasing.
\\
It remains to show that $\theta\in C^{\infty}(t_0, t_1)$. 
This follows from a bootstrap argument, since $kp$
is always more regular than $\theta$.
More precisely, we have
$$
K'(t) = \frac{1}{B_{\theta}(t)}\int_t^1(1-s)\cos K(s)\ ds.
$$
We have already shown that $B_{\theta}\in C^0([0,1])$.
Hence $K'\in C^0([0,1])$, too. Since $(B_{\theta}K')' = (1-t)\cos K$
we have $(B_{\theta}K')''\in C^0([0,1])$.
\\
Now we use 
$$
P'(t) = - \frac{1}{B_{\theta}(t)}\int_t^1(1-s)(\cos K(s) - P(s)\sin K(s))\ ds
$$
to see that, similarly, $P'\in C^0([0, 1])$, hence from
$p' = (1-t)(\cos K - P\sin K)$ we see $p''\in C^0([0,1])$.
\\
But then $kp\in C^2([0,1])$ and by \eqref{1407-1}
$$
B_{\theta} = \left(\frac{kp}{c_l}\right)^{1/2}
\mbox{ on }(t_0, t_1).
$$
Together with the fact that $kp\neq 0$ on $(t_0, t_1)$
this allows us to conclude that $B_{\theta}\in C^2(t_0, t_1)$.
We can bootstrap the above argument to conclude that $B_{\theta}\in C^{\infty}(t_0, t_1)$.
\end{proof}

{\bf Remarks.}
\begin{enumerate}[(i)]
\item If
\begin{equation}
\label{nonop-1}
c_la^2 > \|B_{\theta}^2K'P'\|_{L^{\infty}(0, 1)}
\end{equation}
then $\theta\equiv 0$ is the unique optimal design.
Indeed, if \eqref{nonop-1} is satisfied then \eqref{el-1}
implies that $\theta = 0$ everywhere.
\\
Observe that the right-hand side of \eqref{nonop-1} can be bounded in terms of $a$ and $b$,
while the left-hand side can be made arbitrarily large for fixed $a$ and $b$ by choosing $c_l$ large enough.
\item When $\tau = 1$ (with $\tau$ as in Lemma \ref{lemon2}) then we have more precisely:
$\theta \equiv 0$ is the unique optimal design if and only if
\begin{equation}
\label{nonop-2}
c_l a^2 \geq \|B_{\theta}^2K'P'\|_{L^{\infty}(0, 1)}.
\end{equation}
Indeed,
observe that $B_{\theta}^2K'P' = kp$ is positive and strictly decreasing on $(0, 1)$.
Hence if \eqref{nonop-2} is satisfied, then
$kp < a^2 c_l$ everywhere on $(0, 1)$. Hence
\eqref{el-1} implies $\theta \equiv 0$.
Conversely, if \eqref{nonop-2} is not satisfied, then
$kp > a^2 c_l$
on a set of positive length, which must be an interval $(0, t_0)$, due to
the monotonicity of $kp$.
And according to \eqref{el-1} we have $\theta > 0$ on $(0, t_0)$.
\end{enumerate}

\section{Computation of optimal designs} \label{sec:numericOptDesign}
As in Section~\ref{sec:numericStateEq}, we take into account the DKT element for displacements. 
To describe the material distribution $B$ we use a phase-field function $\pf \in W^{1,2}(S,[-1,1])$.
More precisely, we define 
\begin{align*}
 B_h[\pf] := a (1 - \chi(\pf)) + b \chi(\pf) 
\end{align*}
with $\chi(\pf) = \frac{\pf + 1}{2}$.
Furthermore, we regularize the interface via a Modica--Mortola functional
\begin{align} \label{eq:ModicaMortola}
  \ModicaMortola{\ModicaMortolaInterfaceWidth}[\pf] 
  & = \frac{1}{2} \int_{S} \ModicaMortolaInterfaceWidth | \nabla \pf |^2 + \frac{1}{\ModicaMortolaInterfaceWidth} \DoubleWell(\pf)  \intd x
\end{align}
with $\DoubleWell(\pf) = \frac{9}{16} (\pf^2-1)^2$, which approximates the perimeter functional \cite{MoMo77}.
We discretize the phase-field variable with piecewise affine and continuous finite elements.
Then, we take into account the discrete elastic energy \eqref{eq:Eh}
and introduce $\Estored_h^\pf[\pf_h,\displacement_h] \coloneqq \Estored_h[B_h[\pf_h],\displacement_h]$. 
According to \eqref{eq:discreteLagrangian}, 
the state equation is defined by stationary points of the Lagrangian 
\begin{align}
 \mathcal{L}_{h} [\pf_h, \displacement_h, \LMIso_h ]
 \coloneqq  \Estored_h^\pf[\pf_h,\displacement_h] - \Epot_h[\displacement_h]+ \IsoFct_h[\displacement_h,\LMIso_h] \, .
\end{align}
For a fixed phase-field $\pf_h$, we denote by $(\displacement_h, \LMIso_h)[\pf_h]$ an associated pair of discrete elastic displacement and Lagrange multiplier,
such that $(\pf_h, (\displacement_h, \LMIso_h)[\pf_h])$ is a saddle point of the Lagrangian $\mathcal{L}_h$. Then, for a parameter $\penaltyPerimeter > 0$, we define a discrete, regularized cost functional
\begin{align}
                     & \hat \costFct_h^{\penaltyPerimeter}[\pf_h] = \costFct_h^{\penaltyPerimeter}[\pf_h,\displacement_h[\pf_h]], \\
 \text{where } \quad & \costFct_h^{\penaltyPerimeter}[\pf_h,\displacement_h] := \Epot_h[\displacement_h]+ \penaltyPerimeter \ModicaMortola{\ModicaMortolaInterfaceWidth}[\pf_h]. 
\end{align}
We numerically compute a minimizer of $\costFct_h^{\penaltyPerimeter}$ with a first-order method.
This requires to evaluate the derivative
\begin{align}
\partial \hat \costFct_h^{\penaltyPerimeter}[\pf_h] ( z_h) 
 = \partial_{\pf_h} \costFctExpl_h^{\penaltyPerimeter}[ \pf_h, \displacement_h[\pf_h]] ( z_h)
 + \partial_{\displacement_h} \costFctExpl_h^{\penaltyPerimeter}[ \pf_h, \displacement_h[\pf_h]] 
   (\partial_{\pf_h} \displacement_h[\pf_h] (z_h)) 
\end{align}
for discrete test functions $z_h \in W_{h,\Gamma}(S)^3$. 
To evaluate the second term on the right we consider a suitable adjoint problem
involving the Lagrange multiplier $\LMIso_h$.
By the optimality condition $\partial_{(\displacement_h,\LMIso_h)} \mathcal{L}_h \left[ \pf_h, (\displacement_h,\LMIso_h)[\pf_h] \right] = 0$,
the variation of the mapping $\pf_h \mapsto \partial_{(\displacement_h,\LMIso_h)} \mathcal{L}_h \left[ \pf_h, (\displacement_h,\LMIso_h)[\pf_h] \right]$ vanishes,
hence
\begin{align} \label{eq:VariationLagrangian}
 \partial_{(\displacement_h,\LMIso_h)}  \partial_{(\displacement_h,\LMIso_h)} \mathcal{L}_{h} \left[ \pf_h, (\displacement_h,\LMIso_h)[\pf_h] \right]
 \partial_{\pf_h} (\displacement_h,\LMIso_h)[\pf_h]
 = - \partial_{\pf_h} \partial_{(\displacement_h,\LMIso_h)} \mathcal{L}_{h} \left[ \pf_h, (\displacement_h,\LMIso_h)[\pf_h] \right] \, .
\end{align}
Now, we define adjoint variables 
$(p_h, \mu_h) 
\in W_{h,\Gamma}(S)^3 \times \R^{3|\SetInteriorNodes|}$
as solutions of the linear system
\begin{align*} 
 & \partial_{(\displacement_h, \LMIso_h)} \partial_{(\displacement_h, \LMIso_h)} 
  \mathcal{L}_{h}[\pf_h, (\displacement_h, \LMIso_h)[\pf_h]] (z_h, \beta_h)
 (p_h, \mu_h) 
 = - \partial_{(\displacement_h, \LMIso_h)}
     \costFctExpl_h^{\penaltyPerimeter}[ \pf_h, (\displacement_h, \LMIso_h)[\pf_h]]
     (z_h, \beta_h)   
\end{align*}
for all $(z_h, \beta_h) 
\in W_{h,\Gamma}(S)^3 \times \R^{3|\SetInteriorNodes|}$.
With the help of $(p_h,\mu_h)$ and \eqref{eq:VariationLagrangian}, the derivative of the cost functional reads as 
\begin{align} \label{eq:shapeDerivativeForLagrangian}
\begin{split}
  \partial \hat \costFct_h^{\penaltyPerimeter}[\pf_h](z_h) 
 & = \partial_{\pf_h} \costFctExpl_h^{\penaltyPerimeter}[ \pf_h, (\displacement_h, \LMIso_h)[\pf_h]] (z_h)
    + \partial_{(\displacement_h,\LMIso_h)} \costFctExpl_h^{\penaltyPerimeter}[ \pf_h , (\displacement_h, \LMIso_h) [\pf_h]] 
     (\partial_{\pf_h} (\displacement_h, \LMIso_h) [\pf_h] (z_h) ) \\
 & = \partial_{\pf_h} \costFctExpl_h^{\penaltyPerimeter}[ \pf_h, (\displacement_h, \LMIso_h)[\pf_h]] (z_h) 
   + \partial_{\pf_h} \partial_{(\displacement_h, \LMIso_h)} 
     \mathcal{L}_{h}[ \pf_h, (\displacement_h, \LMIso_h)[\pf_h]](p_h, \mu_h) (z_h).
\end{split}
\end{align}
Since in our case we have
\begin{align*}
 \partial_{\pf_h} \partial_{(\displacement_h, \LMIso_h)} 
 \mathcal{L}_{h}[ \pf_h, \displacement_h, \LMIso_h ]
 =   
 \begin{pmatrix}
  \partial_{\pf_h}\partial_{\displacement_h} \Estored_h^\pf[\pf_h, \displacement_h ] & 0   \\
  0 & 0 \\
  \end{pmatrix} \, ,
\end{align*}
the expression \eqref{eq:shapeDerivativeForLagrangian} for the shape derivative simplifies to
\begin{align*}
\partial \hat \costFct_h^{\penaltyPerimeter}[\pf_h](z_h)
 = \partial_{\pf_h} \costFctExpl_h^{\penaltyPerimeter}[ \pf_h, (\displacement_h, \LMIso_h)[\pf_h]] (z_h) 
 +  \partial_{\pf_h}\partial_{\displacement_h} \Estored_h^\pf[ \pf_h, \displacement_h[\pf_h]] (p_h) (z_h) \, .
\end{align*}

Then, we apply the \verb!IPOPT! solver \cite{WaBi06} to compute minimizer of the fully discrete cost functional $\costFct_h^{\penaltyPerimeter}$
over all $\pf_h \in V_h^1(S,[-1,1])$ with the additional area constraint
\begin{align}
 \mathcal{V}_h[\pf_h] := \int_S \chi(\pf_h) \intd x = \AreaHardPhase \, .
\end{align}
We apply an adaptive refinement scheme via longest edge bisection.
More precisely, to refine the interface we mark those elements $T \in \mathcal{T}_h$ with $\fint_T |\nabla \pf_h|^2 \intd x > \frac{1}{2}$.
Additionally, we mark those elements $T \in \mathcal{T}_h$, where the isometry error $\int_T |\D u^T \D u - I|^2 \intd x$ is large,
i.e., we compute this error for all elements and select the largest $25 \%$ for a longest edge bisection refinement.

\bigskip

In what follows, we discuss selected numerically computed optimal designs.
We always choose $a = 1$ and $b=100$ for the material hardness.
At first, we study the configuration as in Figure~\ref{fig:CompareDesignsIsometry} and take into account the same area constraints $\AreaHardPhase = 0.25,0.5,0.75$.  
We always start with a coarse mesh of $|\SetNodes| = 289$ nodes and use $8$ adaptive refinement steps. 
For the Modica--Mortola functional $\ModicaMortola{\ModicaMortolaInterfaceWidth}$, we set $\penaltyPerimeter=10^{-2}$ and depending on the mesh size $h$ we choose $\ModicaMortolaInterfaceWidth=2h$.
In Figure~\ref{fig:OptDesignsIsometryF} we first consider large forces with $|f| = 100 \AreaHardPhase$.
We observe for a large amount of hard material ($\AreaHardPhase = 0.75$) that the design (I) is optimal.
However, for $\AreaHardPhase = 0.25,0.5$ we obtain optimal designs with significantly better compliance compared to the above considered designs.
Furthermore, we consider in Figure~\ref{fig:OptDesignsIsometryF} small forces with $|f| = 10 \AreaHardPhase$.
Here, for all investigated constraints $\AreaHardPhase$, the optimal solutions are different to the designs (I),(II), and (III), 
even for an area $\AreaHardPhase = 0.75$, where design (I) performs better than (II) and (III).

\begin{figure}[!htbp]
\resizebox{\textwidth}{!}{
\begin{tabular}{  m{2cm}  c  c  c  c  c }
 deformed config.
   & \begin{minipage}{0.15\textwidth} \centering
      \includegraphics[max width=\linewidth]{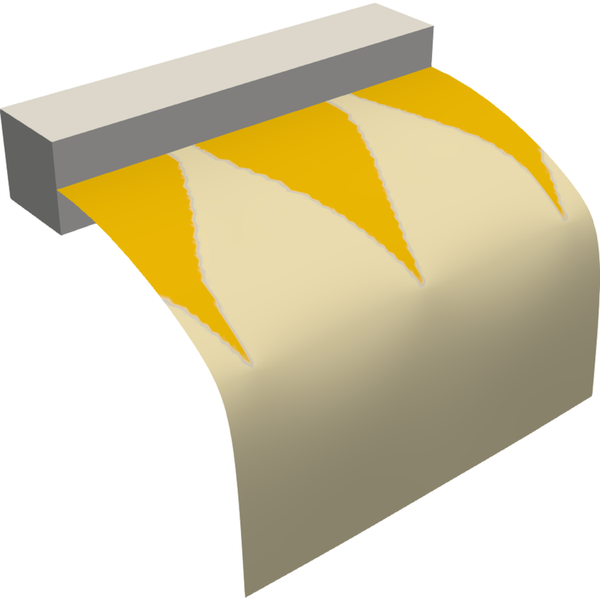}
  \end{minipage}
  & \begin{minipage}{0.15\textwidth} \centering
      \includegraphics[max width=\linewidth]{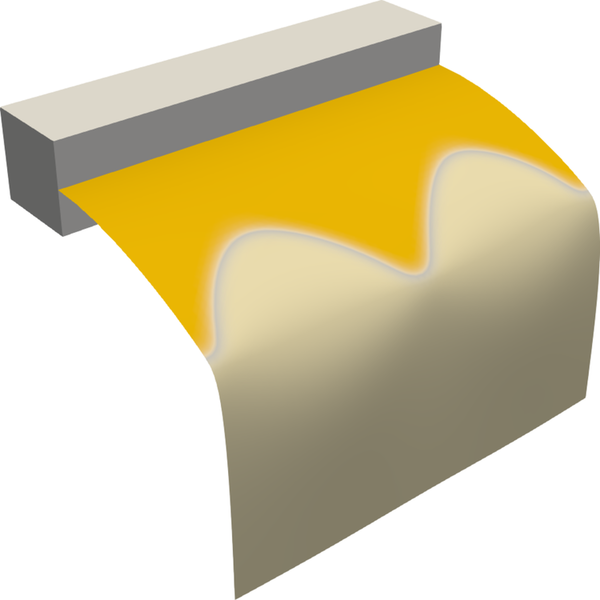}
  \end{minipage}
  & \begin{minipage}{0.15\textwidth} \centering
      \includegraphics[max width=\linewidth]{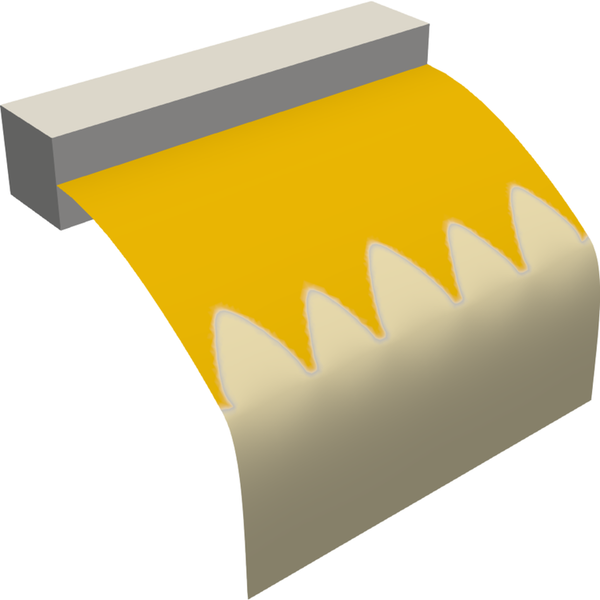}
  \end{minipage}
  & \begin{minipage}{0.15\textwidth} \centering
      \includegraphics[max width=\linewidth]{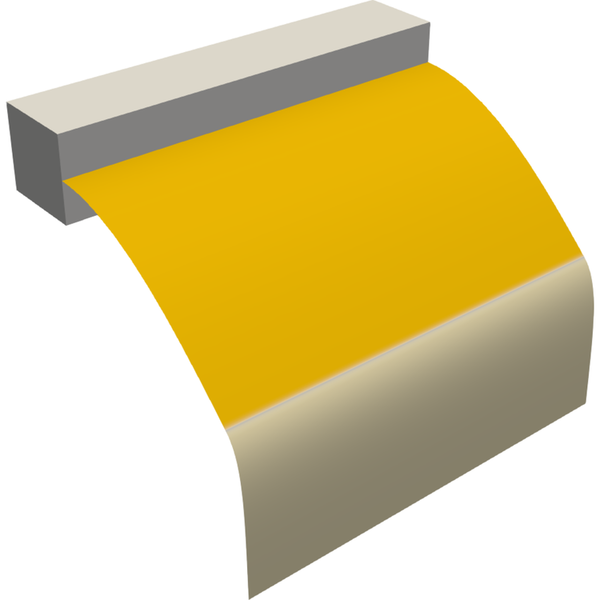}
  \end{minipage}
  & \begin{minipage}{0.15\textwidth} \centering
      \includegraphics[max width=\linewidth]{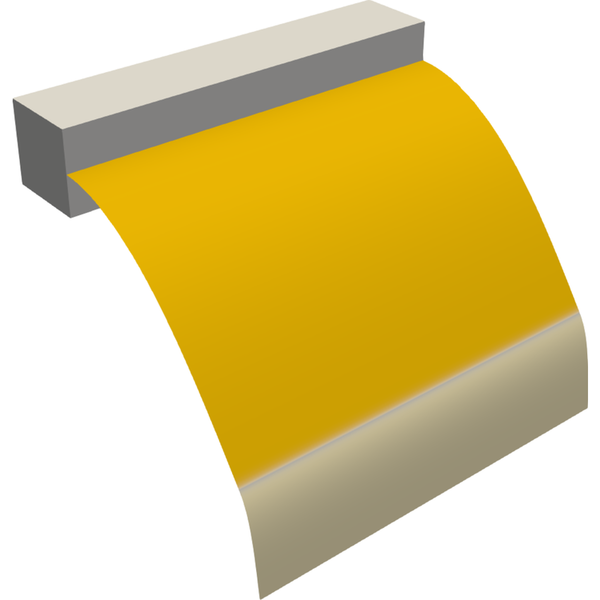}
  \end{minipage}
  \\[0.075\textwidth]
 reference config.
  & \begin{minipage}{0.15\textwidth} \centering 
     \includegraphics[max width=0.9\linewidth, max height=0.9\linewidth]{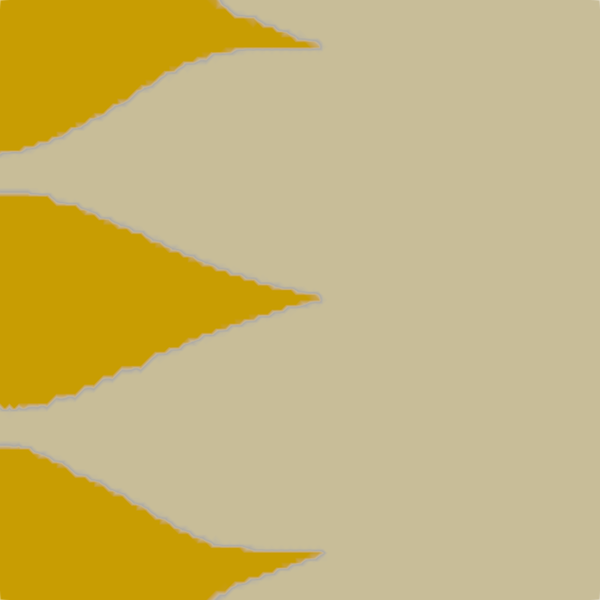}
    \end{minipage}
  & \begin{minipage}{0.15\textwidth} \centering 
     \includegraphics[max width=0.9\linewidth, max height=0.9\linewidth]{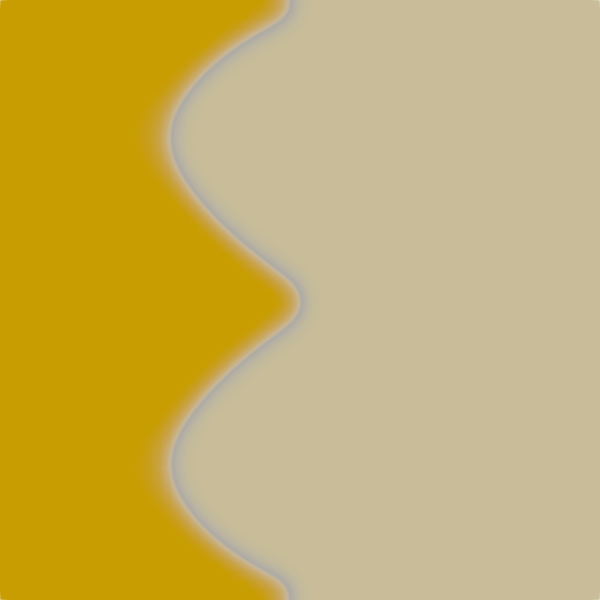}
    \end{minipage}
  & \begin{minipage}{0.15\textwidth} \centering
      \includegraphics[max width=0.9\linewidth, max height=0.9\linewidth]{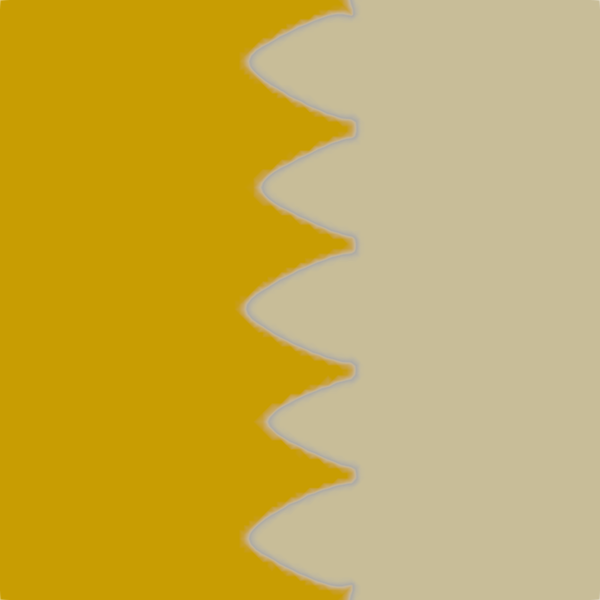}
     \end{minipage}
  & \begin{minipage}{0.15\textwidth} \centering
      \includegraphics[max width=0.9\linewidth, max height=0.9\linewidth]{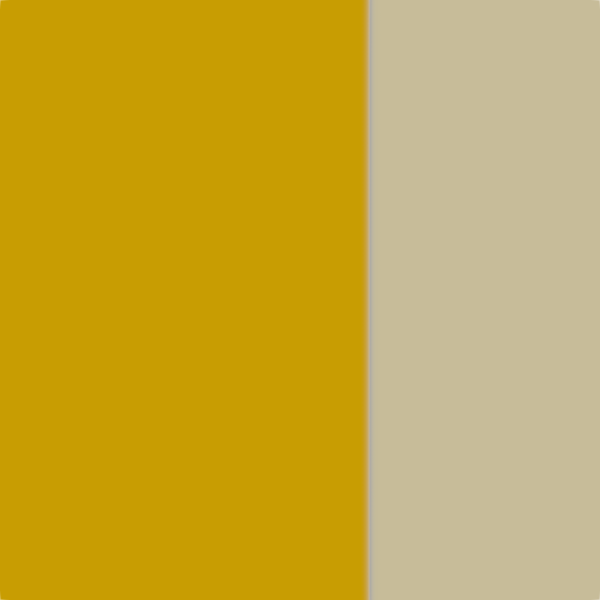}
     \end{minipage}
  & \begin{minipage}{0.15\textwidth} \centering
      \includegraphics[max width=0.9\linewidth, max height=0.9\linewidth]{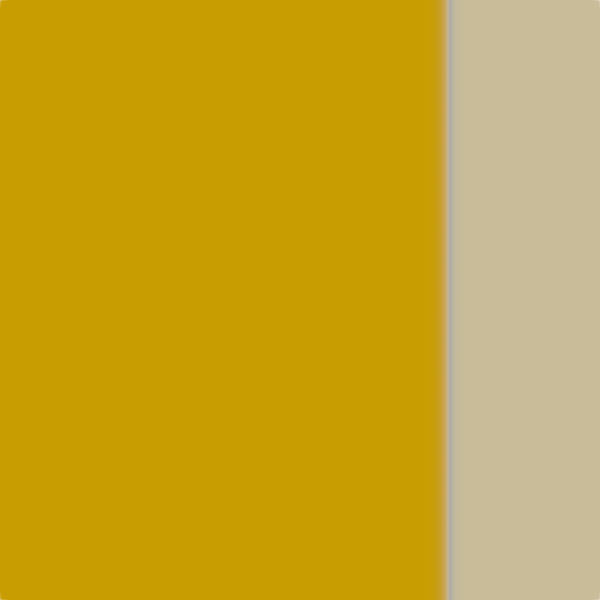}
    \end{minipage}
    \\[0.075\textwidth] \hline        deformed config.
  & \begin{minipage}{0.15\textwidth} \centering
      \includegraphics[max width=\linewidth]{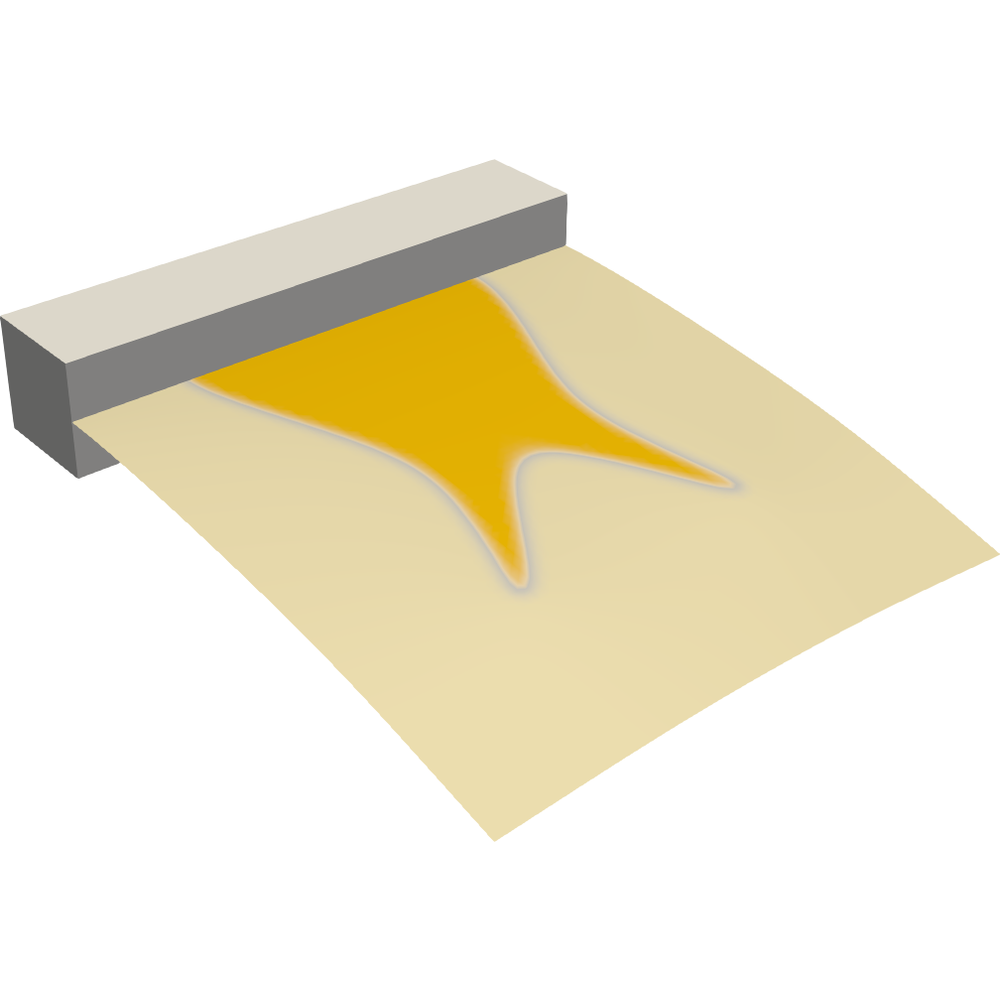}
  \end{minipage}
  & \begin{minipage}{0.15\textwidth} \centering
      \includegraphics[max width=\linewidth]{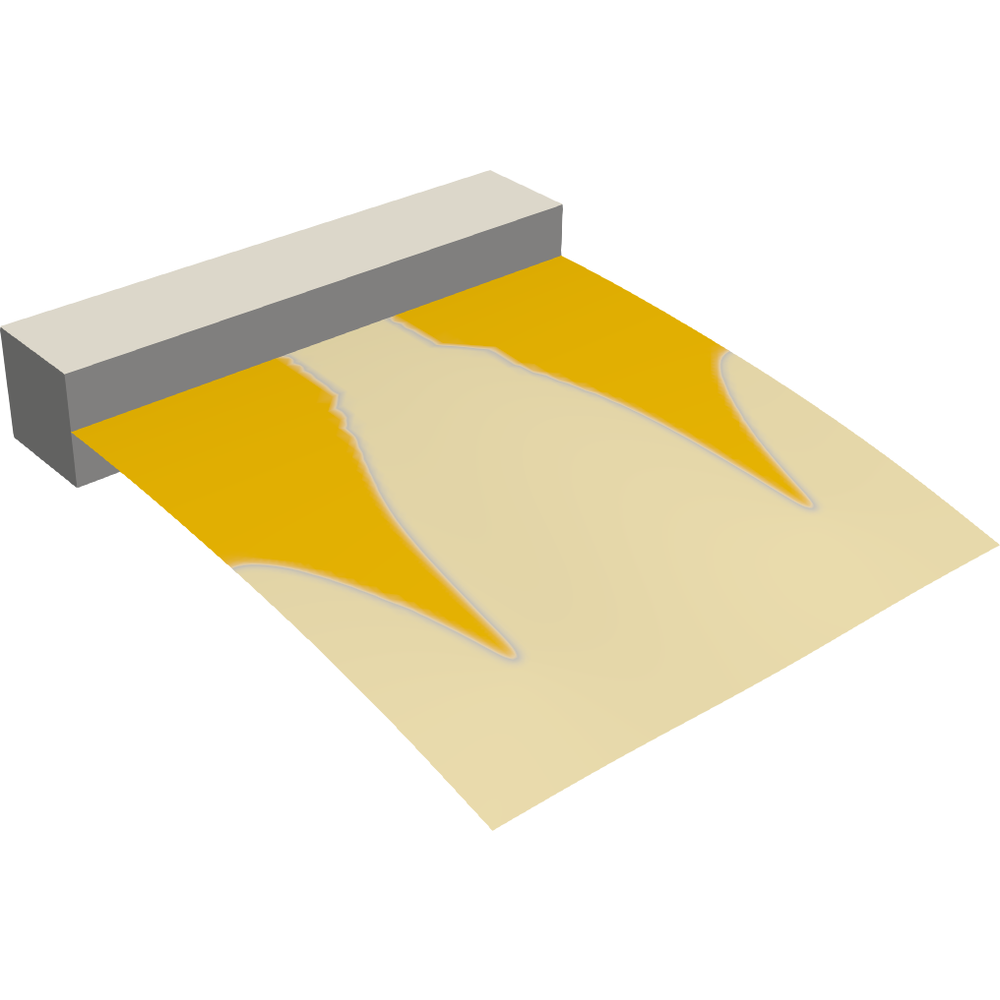}
  \end{minipage}
  & \begin{minipage}{0.15\textwidth} \centering
      \includegraphics[max width=\linewidth]{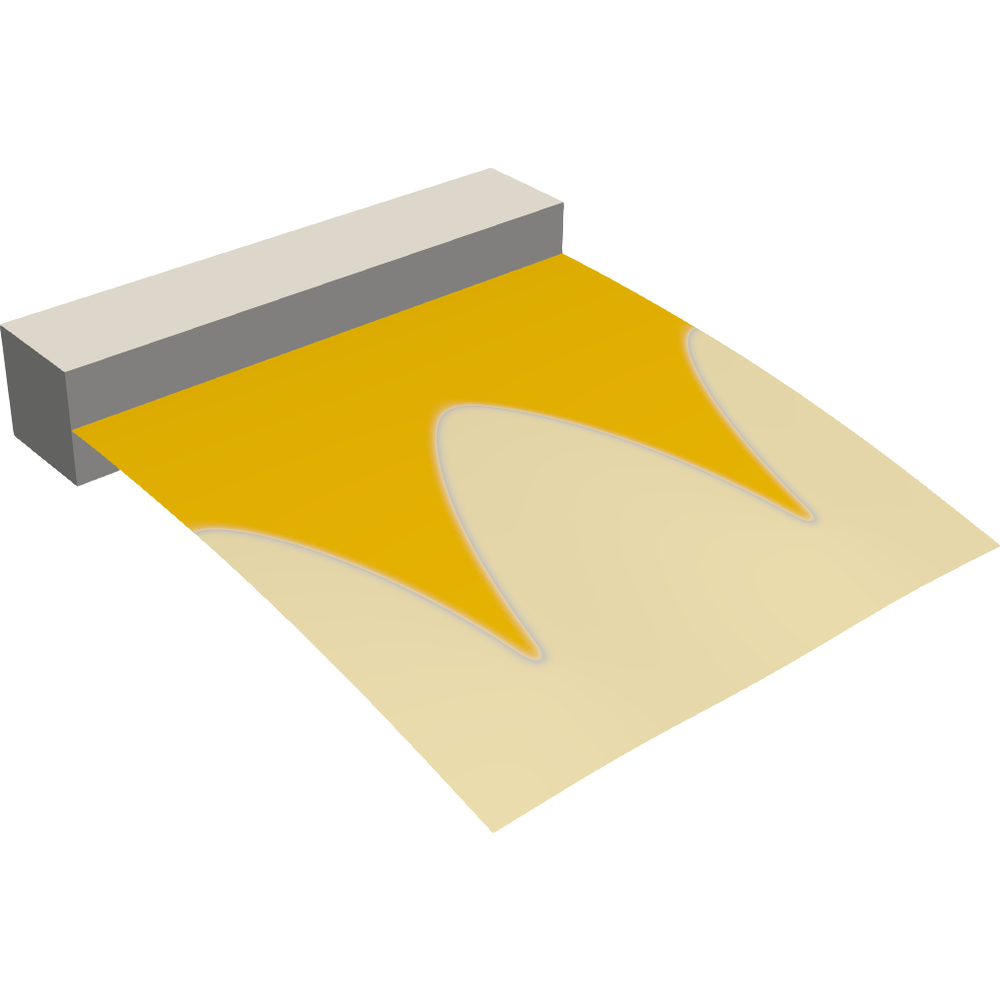}
  \end{minipage}
  & \begin{minipage}{0.15\textwidth} \centering
      \includegraphics[max width=\linewidth]{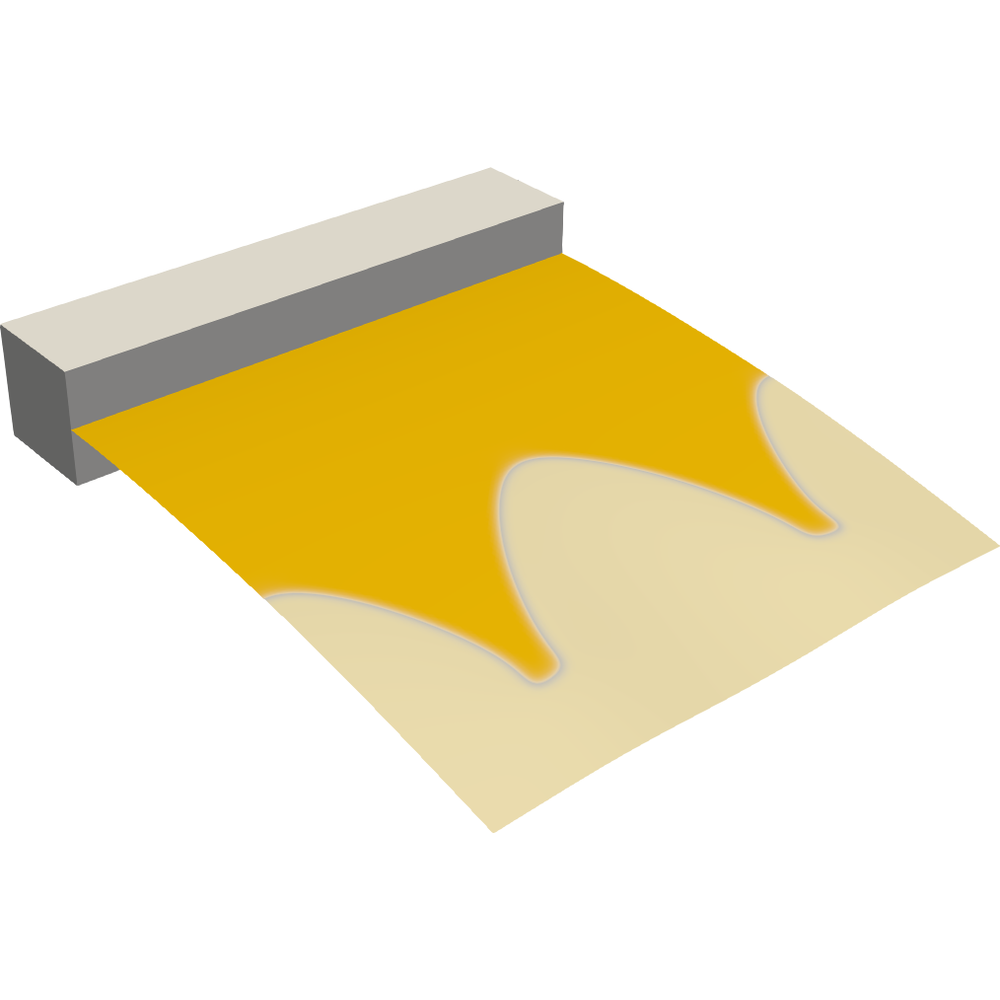}
  \end{minipage}
  & \begin{minipage}{0.15\textwidth} \centering
      \includegraphics[max width=\linewidth]{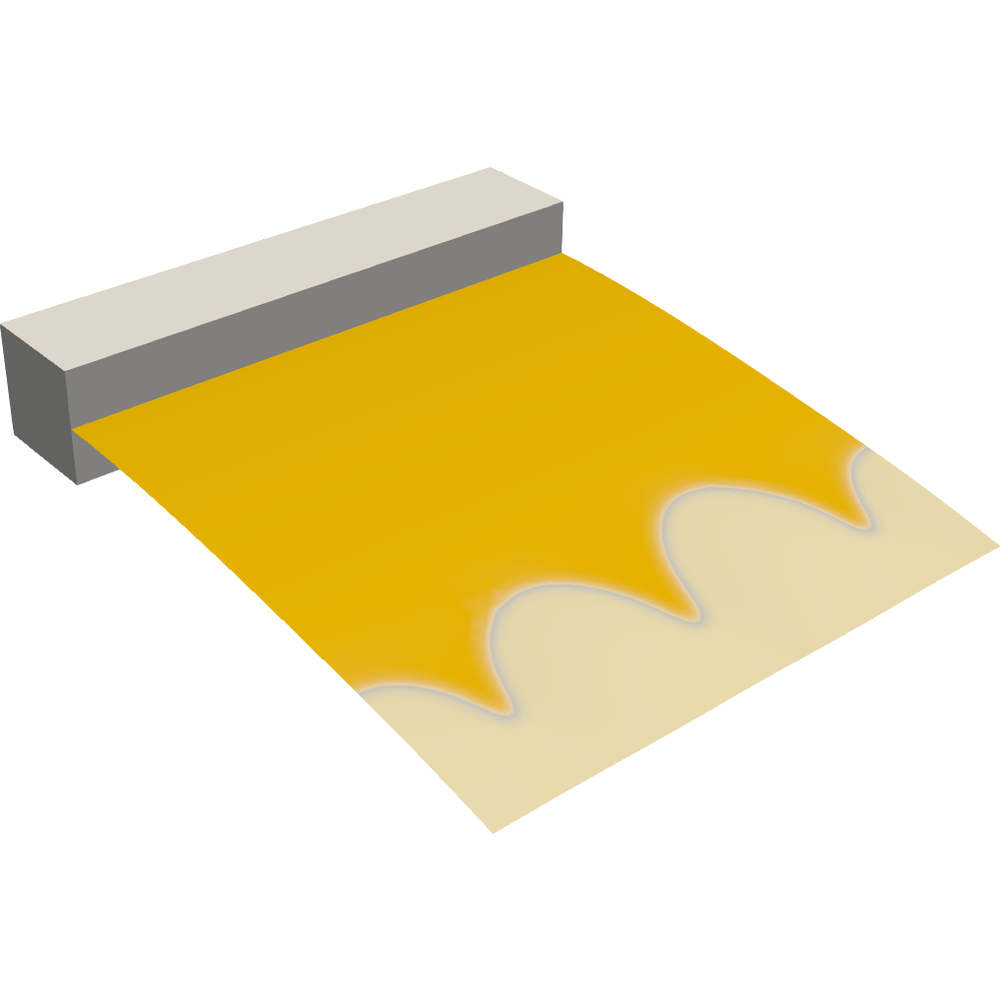}
  \end{minipage}
  \\[0.075\textwidth]
 reference config.
  & \begin{minipage}{0.15\textwidth} \centering 
     \includegraphics[max width=0.9\linewidth, max height=0.9\linewidth]{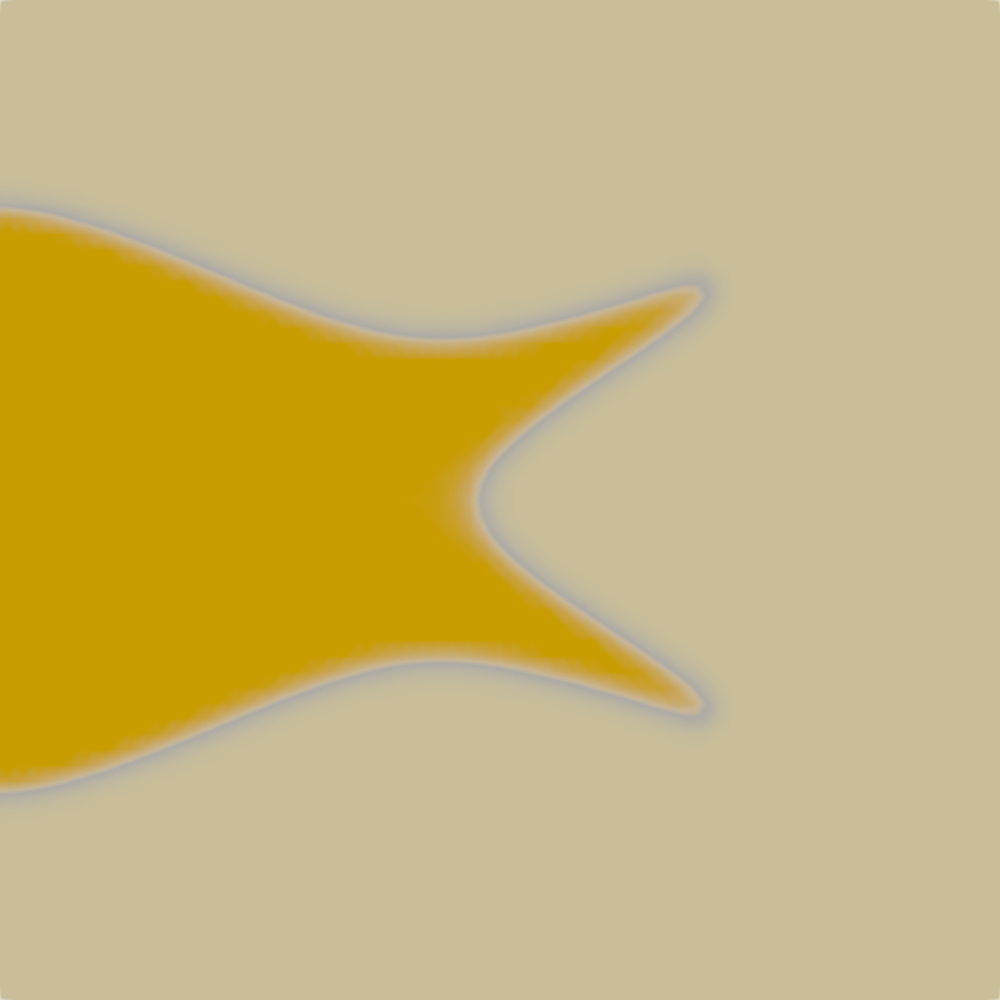}
    \end{minipage}
  & \begin{minipage}{0.15\textwidth} \centering 
     \includegraphics[max width=0.9\linewidth, max height=0.9\linewidth]{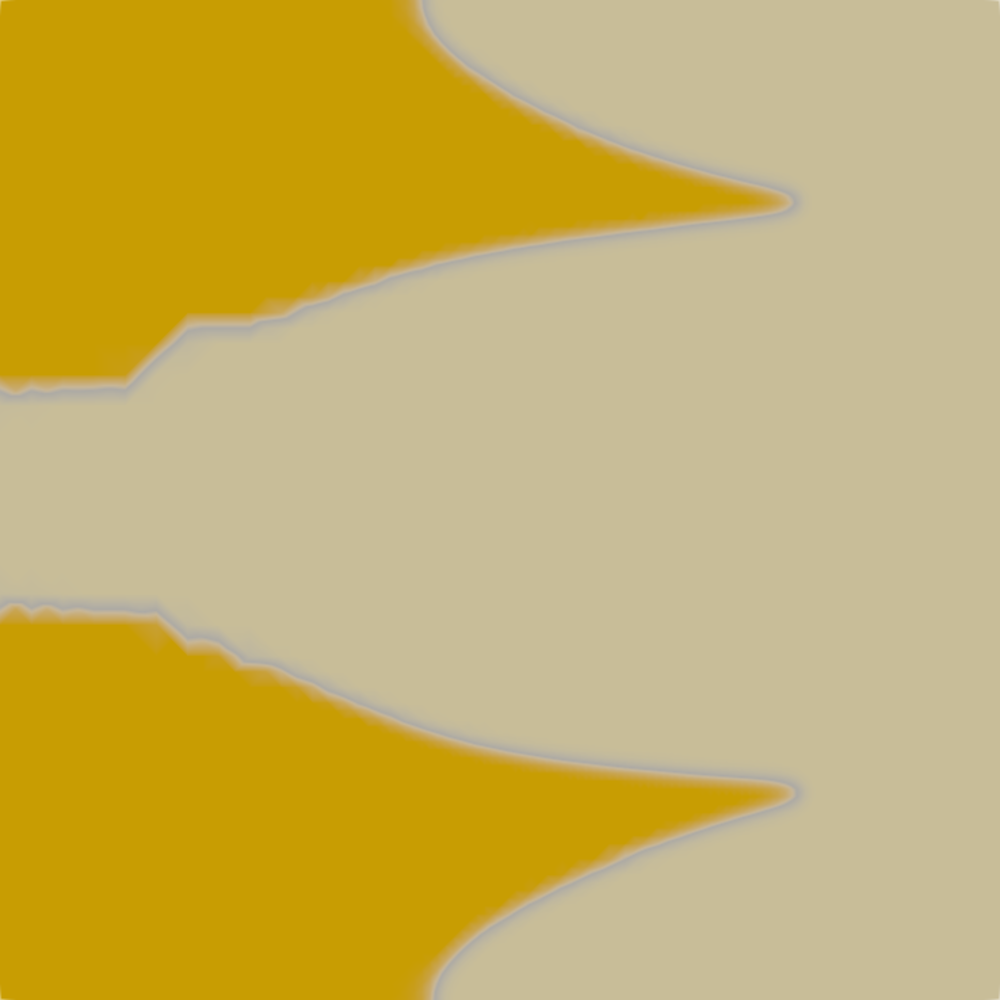}
    \end{minipage}
  & \begin{minipage}{0.15\textwidth} \centering
      \includegraphics[max width=0.9\linewidth, max height=0.9\linewidth]{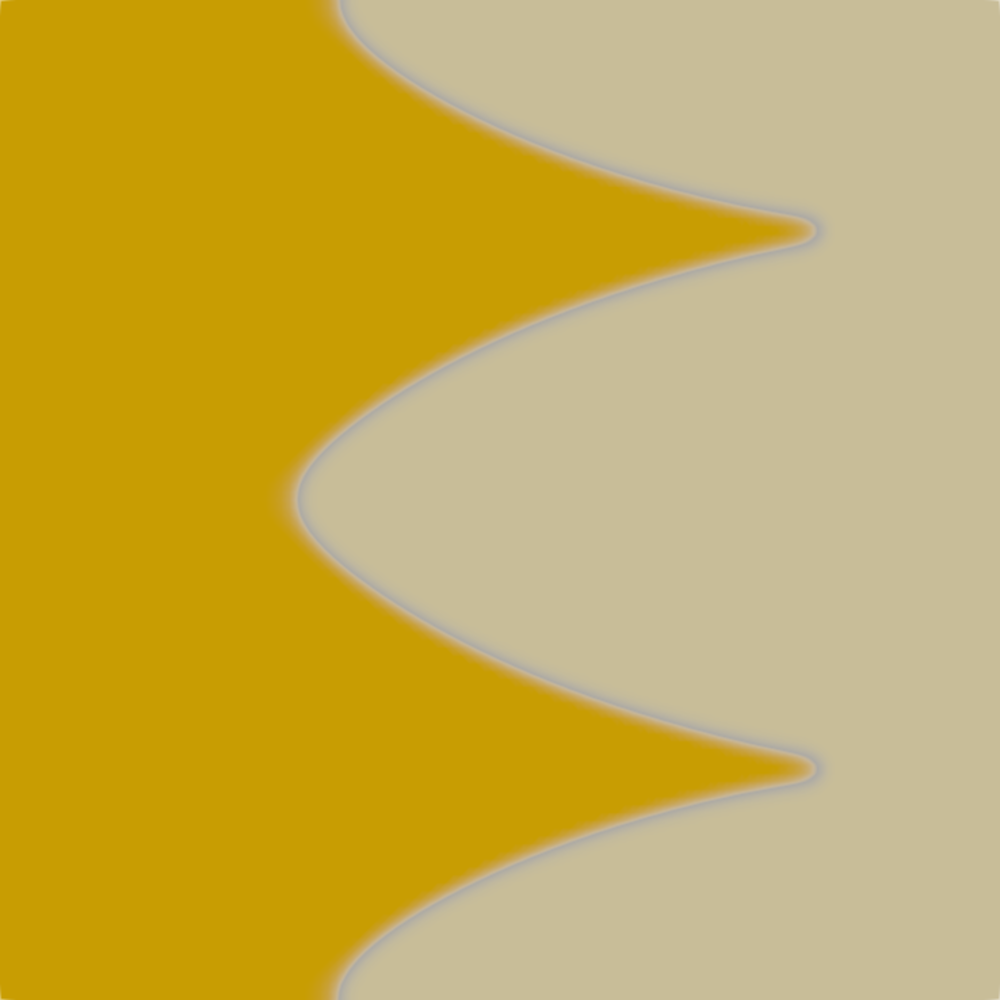}
     \end{minipage}
  & \begin{minipage}{0.15\textwidth} \centering 
     \includegraphics[max width=0.9\linewidth, max height=0.9\linewidth]{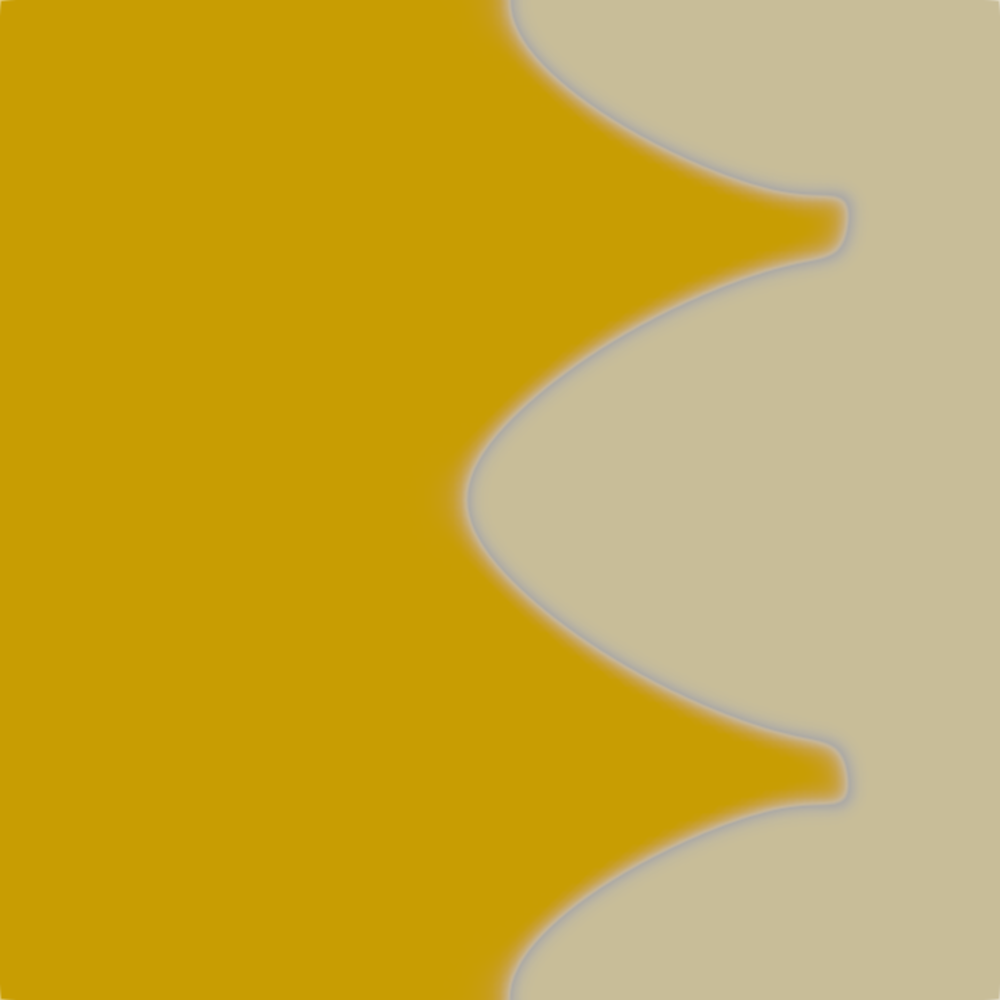}
    \end{minipage}
  & \begin{minipage}{0.15\textwidth} \centering
      \includegraphics[max width=0.9\linewidth, max height=0.9\linewidth]{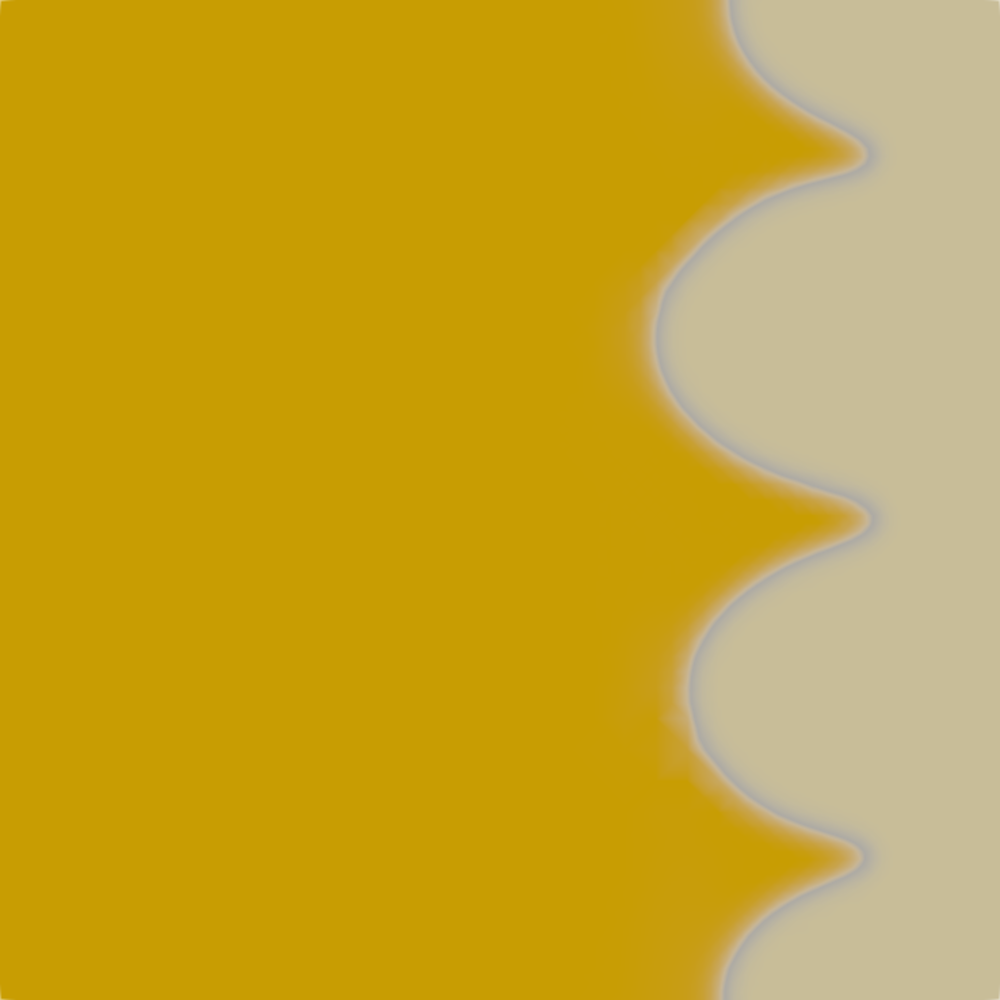}
    \end{minipage}
  \\[0.075\textwidth]
 adaptive mesh
  & \begin{minipage}{0.15\textwidth} \centering 
     \includegraphics[max width=0.9\linewidth, max height=0.9\linewidth]{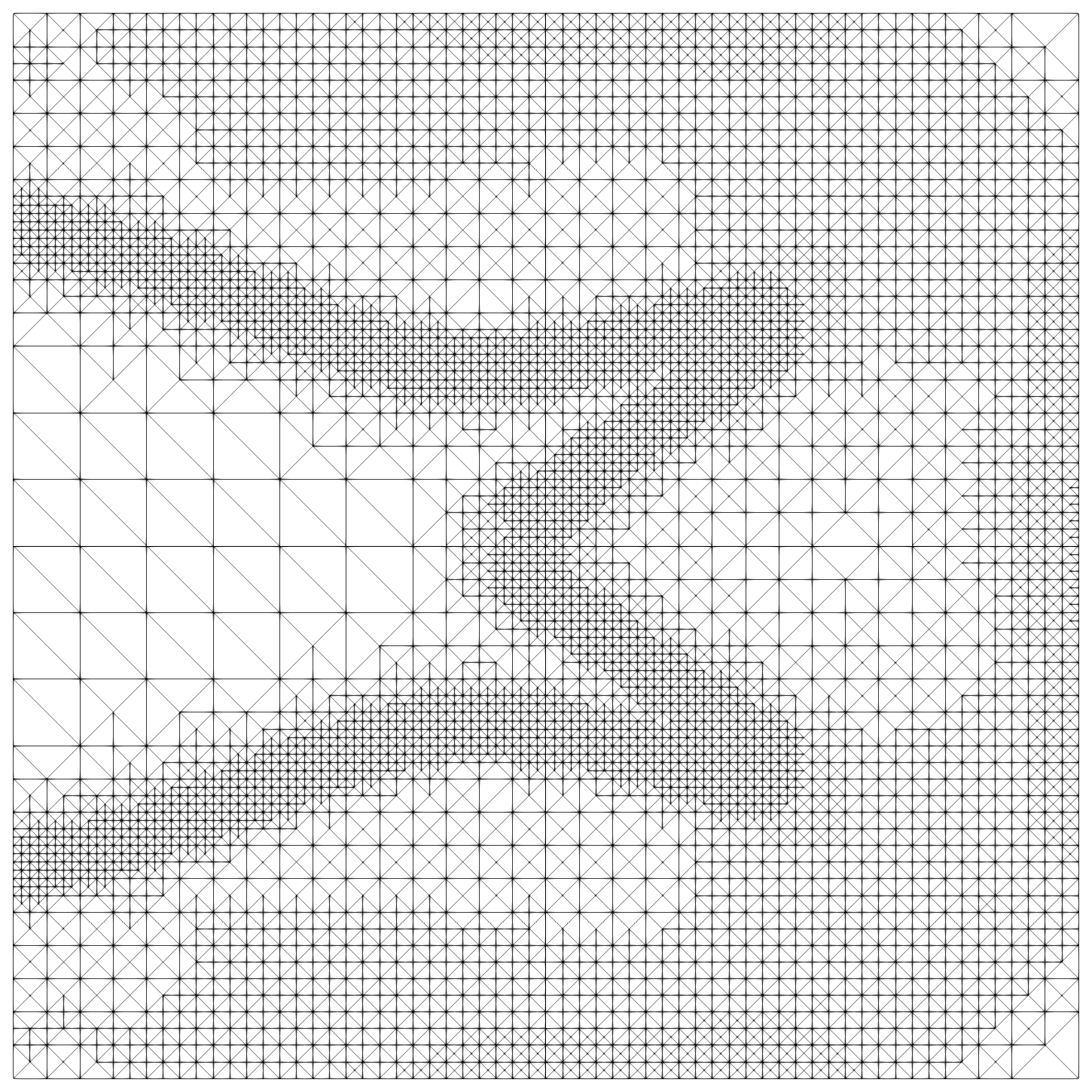}
    \end{minipage}
  & \begin{minipage}{0.15\textwidth} \centering 
     \includegraphics[max width=0.9\linewidth, max height=0.9\linewidth]{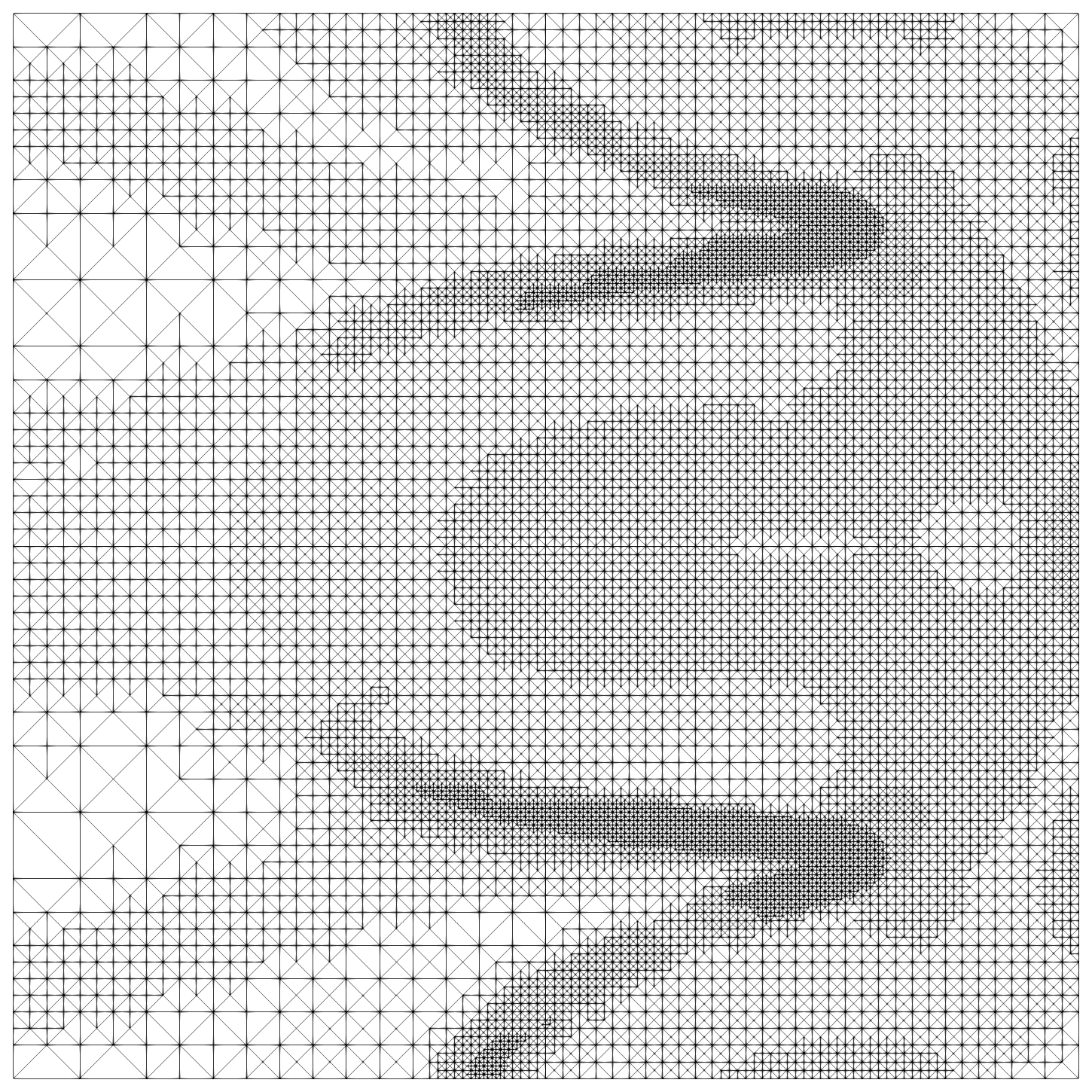}
    \end{minipage}
  & \begin{minipage}{0.15\textwidth} \centering
      \includegraphics[max width=0.9\linewidth, max height=0.9\linewidth]{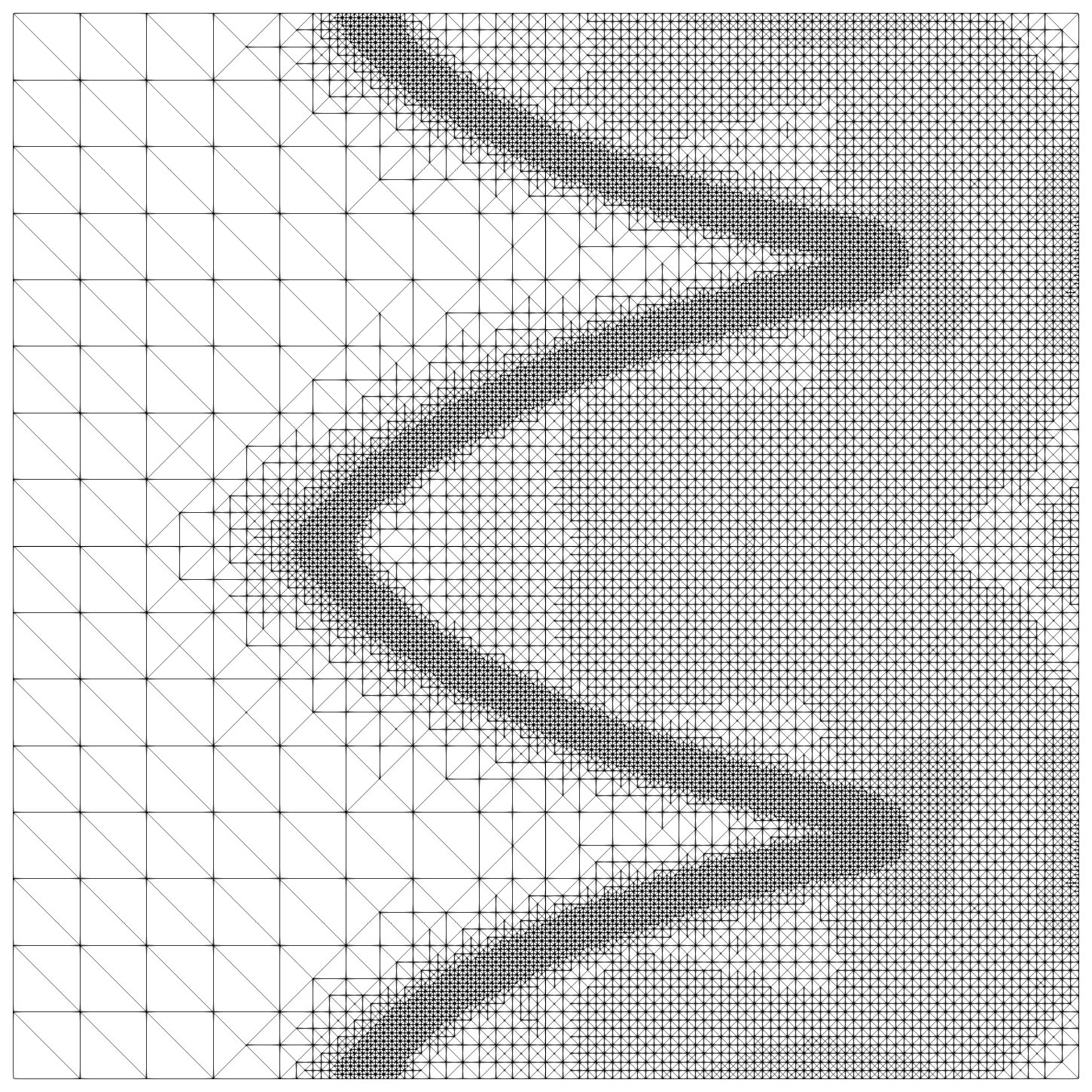}
     \end{minipage}
  & \begin{minipage}{0.15\textwidth} \centering 
     \includegraphics[max width=0.9\linewidth, max height=0.9\linewidth]{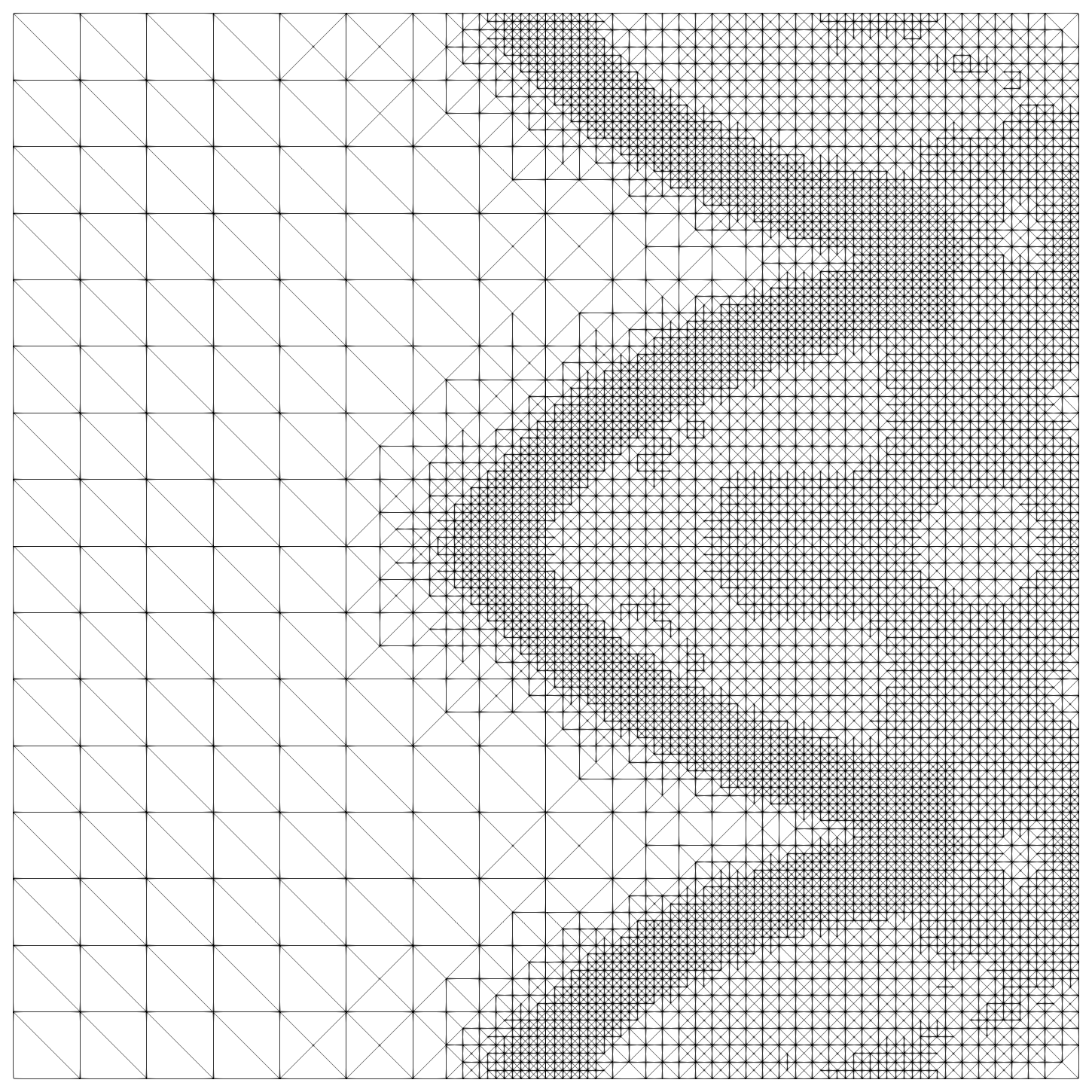}
    \end{minipage}
  & \begin{minipage}{0.15\textwidth} \centering
      \includegraphics[max width=0.9\linewidth, max height=0.9\linewidth]{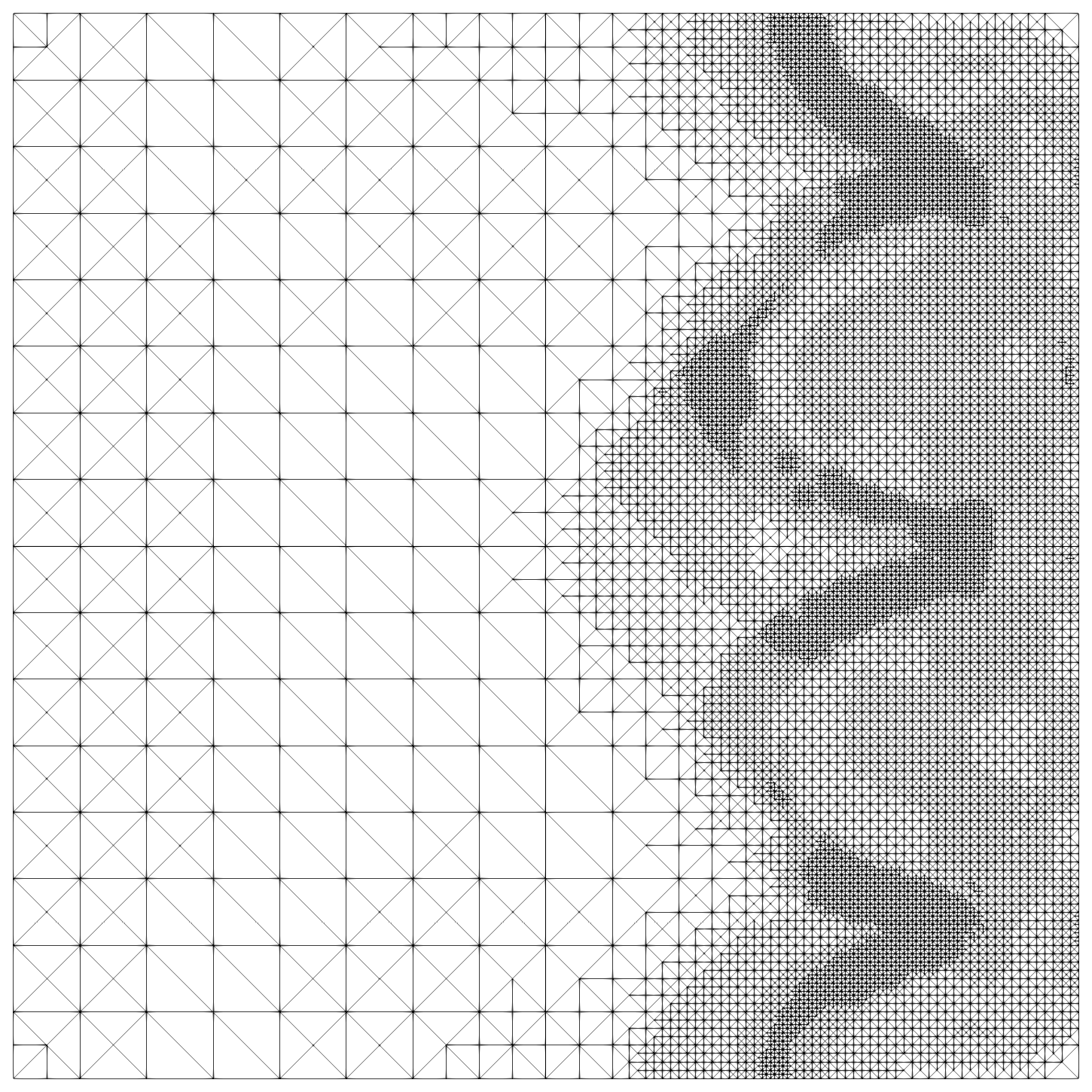}
    \end{minipage} 
  \\
\end{tabular}
}  
\caption
{Optimal material distributions for bending isometries of a plate.
The amounts of hard material are chosen by (from left to right) $\AreaHardPhase=0.25,0.375,0.5,0.625,0.75$
and the corresponding forces are given by $f = (0,0,100 \AreaHardPhase)$ (first and second row) and  $f = (0,0,10 \AreaHardPhase)$ (third and fourth row). On the bottom the adaptive finite element meshes are displayed.}
\label{fig:OptDesignsIsometryF}
\end{figure}

Note that in all computations in Figure~\ref{fig:OptDesignsIsometryF} the deformation on the boundary $ \{l \} \times (-1/2, 1/2)$ is not constraint. 
Nevertheless, as predicted in Theorem \ref{optimal} under the additional constraint that the deformation is affine on $\{l\}\times (-1/2, 1/2)$
the optimal designs are characterized by 
\begin{itemize}
\item[-] a strip of hard material for $x_1\in (0,t_0]$, where $t_0$ vanishes for small values of $V$,
\item[-] a transmission zone from fully hard to fully soft material phase for $x_1\in (t_0;t_1)$ with $t_1$ strictly larger than $t_0$ (In fact, 
in our two phase field model the achieved zick--zack profile seems to reflect 
a local minimum of the total cost functional $\costFct_h^{\penaltyPerimeter}$ including the additional approximate perimeter functional
$\penaltyPerimeter \ModicaMortola{\ModicaMortolaInterfaceWidth}$. 
In particular, different choices of the initial phase field $\pf_h$ lead to different 
zick--zack pattern),
\item[-] a  strip of soft material for $x_1\in [t_1,1)$ with $t_1 \leq 1$ depending on $V$.
\end{itemize}
Adding the constraint that the deformation is affine on $\{l\}\times (-1/2, 1/2)$,
we observe numerically almost no difference concerning the optimal shapes of the hard and soft material phase 
for the forces and area constraints as in Figure~\ref{fig:OptDesignsIsometryF}.
Only for an area constraint $\AreaHardPhase=0.25$ and a force $|F| = 10 \AreaHardPhase$ we obtain a different optimizer, which we depict in Figure~\ref{fig:optDesignStiffRightBdry}. 
To avoid too much increase of compliance density due to a bending in $-e_3$ direction in the corners $(l,-\tfrac12)$ and $(l,\tfrac12)$,
hard spikes occur when optimizing the material distribution in the absence of the constraints. 
This bending is prohibited in case of the constraint, leading to centered spike in this concrete configuration.
For both optimizer we compare in Figure~\ref{fig:optDesignStiffRightBdry} the integral over the hard material phase $\chi(\pf_h)$ 
$\int_{ x_1 \times [0,1]} \chi(\pf_h) \intd x_2$ along the $x_2$-axis as a function of $x_1$. 
In the constraint case, the resulting function $\bar B$ is indeed strictly monotone decreasing in $x_2$ as predicted in Theorem \ref{optimal}. 
However, in the non constraint case at the branch point of the two spikes in $x_2$ direction the strict monotonicity appears to be violated. 
\begin{figure}
\begin{tikzpicture}
   \node at (-4, 3){no additional constraint};
   \node at (0, 3){\includegraphics[max width=0.2\linewidth, max height=0.2\linewidth]{results/Isometry/F10/Area0.25/SolMaterial_Deformed.png}};
   \node at (-4, 0){{\color{BrickRed} affine deformation}};  \node at (-4, -0.5){ {\color{BrickRed} on $\{l\}\times (-1/2, 1/2)$}};
   \node at (0, 0){\includegraphics[max width=0.2\linewidth, max height=0.2\linewidth]{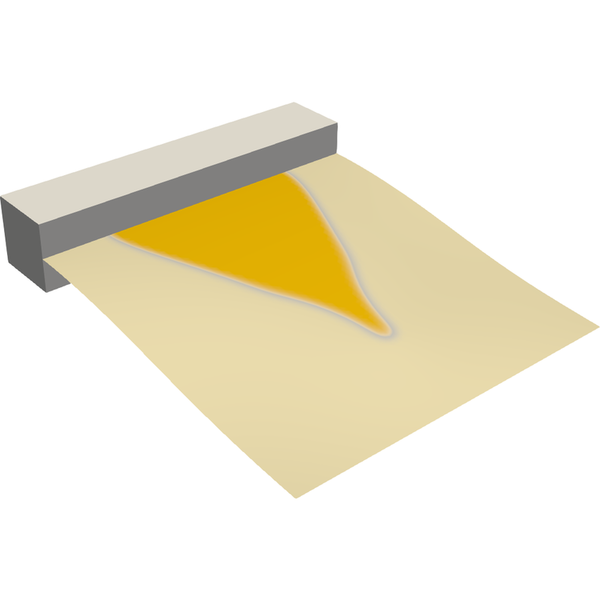}};
   \node at (6, 1.5){\includegraphics[max width=0.4\linewidth, max height=0.4\linewidth]{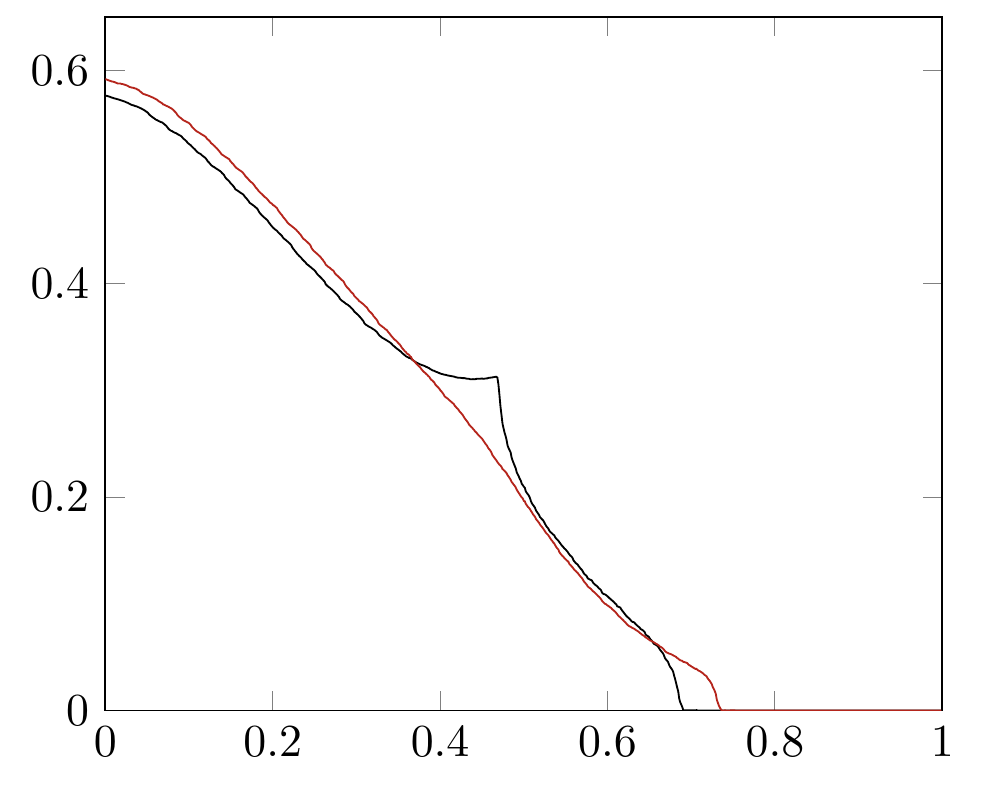}};
\end{tikzpicture}
 \caption{Optimal material distribution with and without stiff right boundary. 
 We compare the integral of the hard material phase along the $x_2$-axis (red curve in the constraint case).}
 \label{fig:optDesignStiffRightBdry}
\end{figure}

\section{Optimal design in case of membrane and bending energy}\label{sec:Mixed}
So far, we have investigated the case of pure bending isometries.
In this section, we relax the isometry constraint by using a membrane energy in addition to the bending energy $\mathcal{E}$.
Furthermore, we consider numerically elastic shells with curved undeformed configuration.
In explicit, the reference configuration $S_{ \text{ref}}$ is assumed to be a two-dimensional, compact, orientable manifold with Lipschitz boundary,
which is parametrized by a fixed, single chart $\psi_{ \text{ref}} \colon \omega \to S_{ \text{ref}}$, where $\omega \subset \R^2$.
We consider deformed configurations $S_{ \text{def}} = u(S)$ which can be parameterized over $\omega$ by $\psi_{ \text{def}} := u \circ \psi_{ \text{ref}}$.
This allows to formulate the energies in terms of $\psi_{ \text{def}}$.
We denote by $g_{ \text{ref}} = D\psi_{ \text{ref}}^T D \psi_{\text{ref}}$, $g_{ \text{def}} = D\psi_{ \text{def}}^T D \psi_{ \text{def}} \in \R^{2 \times 2}$ the first fundamental forms of $S_{ \text{ref}}$ and $S_{ \text{def}}$ 
on the chart domain as functions of $\psi_{\text{ref}}$.
As in Section~\ref{sec:numericOptDesign}, we take into account a phase field function $v \in H^1(\omega,[-1,1])$ to the describe the material hardness $B: \omega \to [a,b]$.
Assuming a Poisson ratio $\nu = 0.25$, the Lam\'{e}--Navier parameters can be expressed by $\mu(v) = \lambda(v) = \frac{2}{5} B(v)$.
Now, we define a membrane energy 
\begin{align*}
 \mathcal{E}^v_{\text{mem}}[\pf, \psi_{ \text{def}}]  = \int_{\omega} \sqrt{\det g_{ \text{ref}}} \; W_{\text{mem}} \left( \pf, g_{ \text{ref}}^{-1} g_{ \text{def}} \right) \intd \xi ,
\end{align*}
where the density function is given by
\begin{align*}
 W_{\text{mem}}(\pf, F)  
 =   \frac{\mu(\pf)}{2} \Tr(F) 
   + \frac{\lambda(\pf)}{4} \det( F )
   - \left( \frac{\mu(\pf)}{2} + \frac{\lambda(\pf)}{4} \right) \log( \det(F) ) - \mu(\pf) - \frac{\lambda(\pf)}{4} 
\end{align*}
for $F \in \R^{2 \times 2}_{\text{sym}}$ (cf. \cite{HeRuWa12}).
For the bending energy, we simply use the squared Frobenius-norm of the relative shape operator $g_{ \text{ref}}^{-1} (A_{ \text{def}} - A_{ \text{ref}})$
and choose
\begin{align*}
  \mathcal{E}^v_{\text{bend}}[\pf, \psi_{ \text{def}}]  = \int_{\omega} \sqrt{\det g_{ \text{ref}}} \; B(\pf) \| g_{ \text{ref}}^{-1} (A_{ \text{def}} - A_{ \text{ref}}) \| \intd \xi  \, .
\end{align*}
Then, the stored elastic energy is defined via properly scaling both energy components with respect to the shell 
thickness parameter $\delta$ as follows:
\begin{align*}
\mathcal{E}^v[\pf,\psi_{ \text{def}}] := \delta\; \mathcal{E}^v_{\text{mem}}[\pf,\psi_{ \text{def}}] + \delta^3\; 
\mathcal{E}^v_{\text{bend}}[\pf,\psi_{ \text{def}}].
\end{align*}
In what follows we consider different undeformed configurations and loads.
Similar to bending isometries, we numerically compute solutions of the state equation via a Newton method.
Then, for the material optimization, we apply the \verb!IPOPT! solver to compute minimizer of a fully discrete cost functional.
Here, we also use an adaptive meshing strategy with $7$ refinement steps via longest edge refinement of those elements $T$ with $\fint_T |\nabla \pf_h|^2 \intd x > \frac{1}{2}$.

\bigskip

\emph{Centered Load on a Plate.}
First, we investigate the flat case $S_{ \text{ref}} = [0,1]^2$.
We consider a force $f = \left( 0, 0, c \, \chi_{[0.45,0.55]^2} \right)$ in normal direction and is supported on a square in the center of $S$.
The deformation is supposed to be clamped at the boundary $\partial S$.
As penalty parameter for the Modica--Mortola functional, we choose $\penaltyPerimeter = 10^{-3}$.
Moreover, we choose different area constraints $\AreaHardPhase = \frac{k}{8}$ for $k=2,3,4,5,6$.
Then, depending on this area constraint, we set $c = - 250 \AreaHardPhase$ for the force 
to ensure that the corresponding deformations are comparable.
Furthermore, we consider $\delta = 10^{-2}$.
In Figure~\ref{fig:shapeOptKoiterPlateCenteredLoad}, 
we depict the cross type structure for the hard phase which characterizes the minimizer of  the compliance functional.
\begin{figure}[!htbp]
\resizebox{\textwidth}{!}{
\begin{tabular}{ m{2cm} c c c c c }
  deformed config.
  & \begin{minipage}{0.15\textwidth} \centering
      \includegraphics[max width=1.0\linewidth]{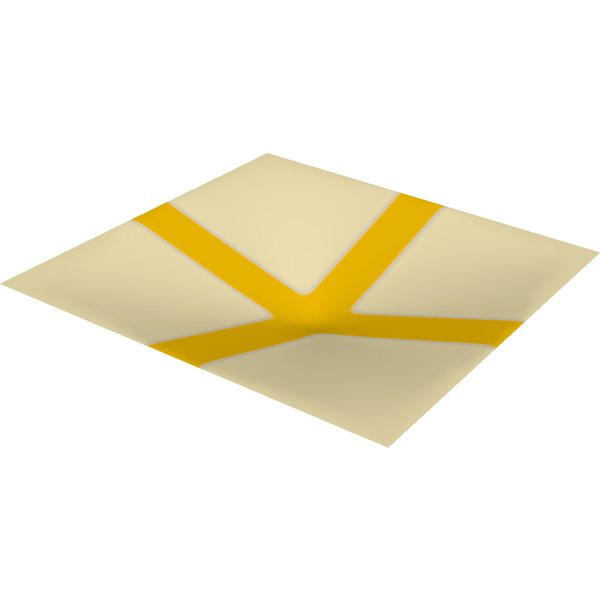}
  \end{minipage}
  & \begin{minipage}{0.15\textwidth} \centering
      \includegraphics[max width=1.0\linewidth]{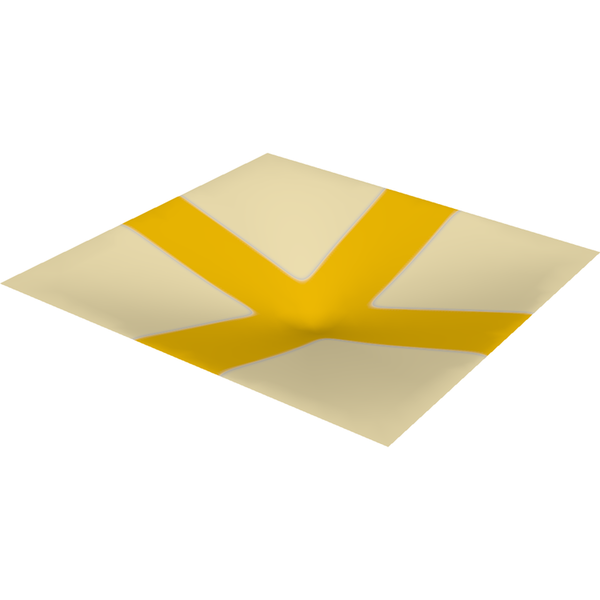}
  \end{minipage}
  & \begin{minipage}{0.15\textwidth} \centering
      \includegraphics[max width=1.0\linewidth]{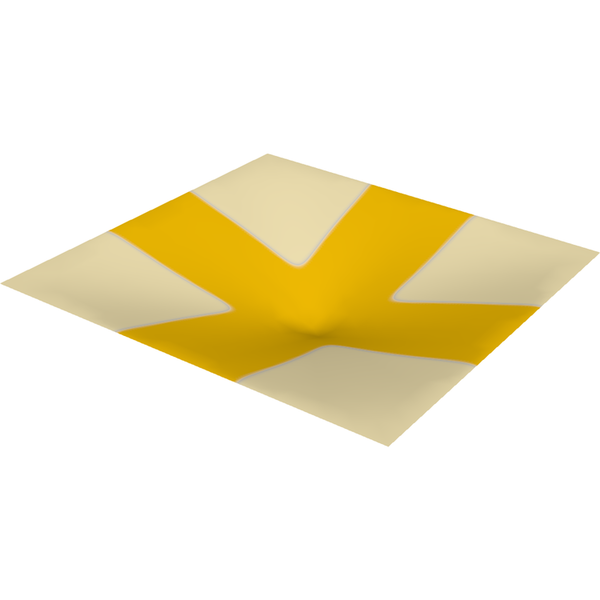}
  \end{minipage}
  & \begin{minipage}{0.15\textwidth} \centering
      \includegraphics[max width=1.0\linewidth]{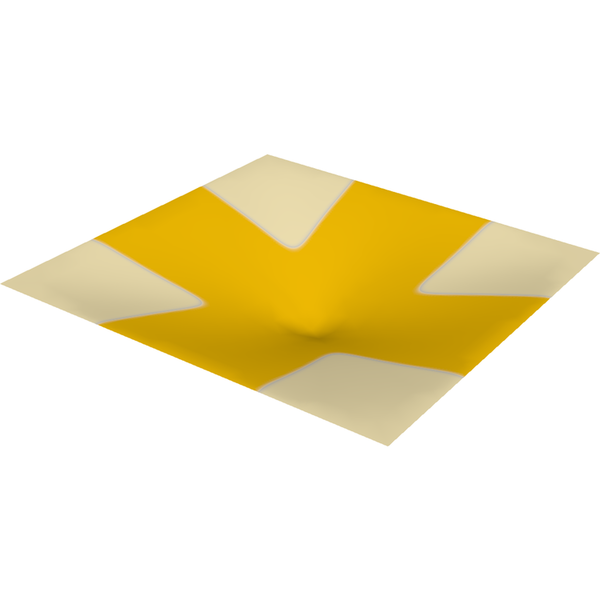}
  \end{minipage}
  & \begin{minipage}{0.15\textwidth} \centering
      \includegraphics[max width=1.0\linewidth]{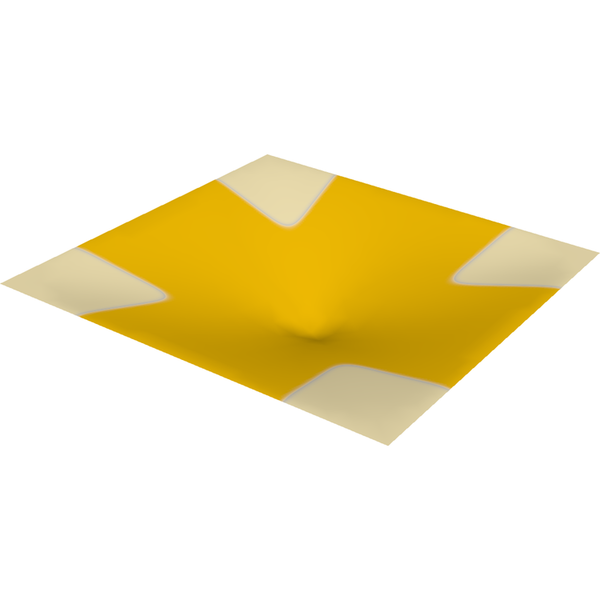}
  \end{minipage}
  \\
   undeformed config.
  & \begin{minipage}{0.15\textwidth} \centering 
     \includegraphics[max width=0.9\linewidth, max height=0.9\linewidth]{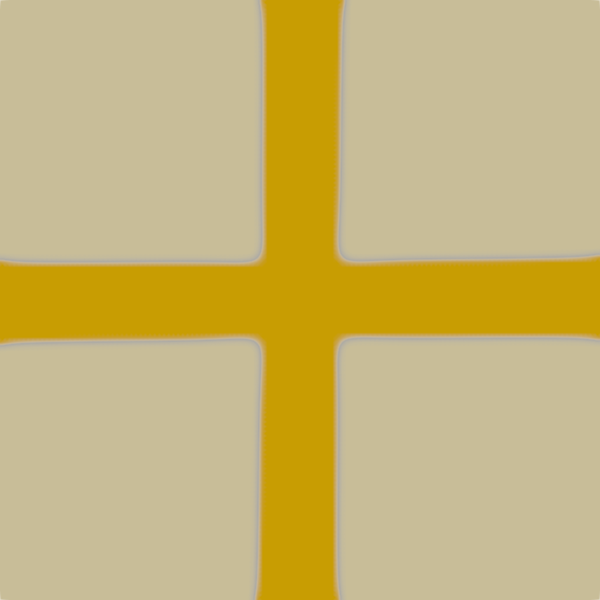}
    \end{minipage}
  & \begin{minipage}{0.15\textwidth} \centering
      \includegraphics[max width=0.9\linewidth, max height=0.9\linewidth]{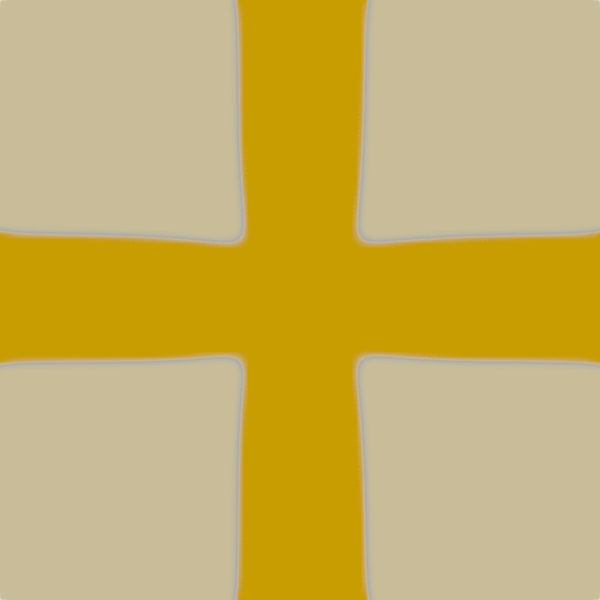}
     \end{minipage}
  & \begin{minipage}{0.15\textwidth} \centering
      \includegraphics[max width=0.9\linewidth, max height=0.9\linewidth]{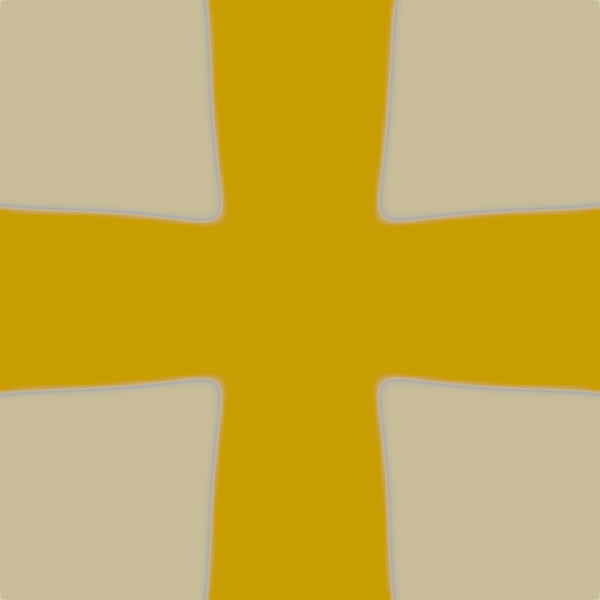}
     \end{minipage}
  & \begin{minipage}{0.15\textwidth} \centering
      \includegraphics[max width=0.9\linewidth, max height=0.9\linewidth]{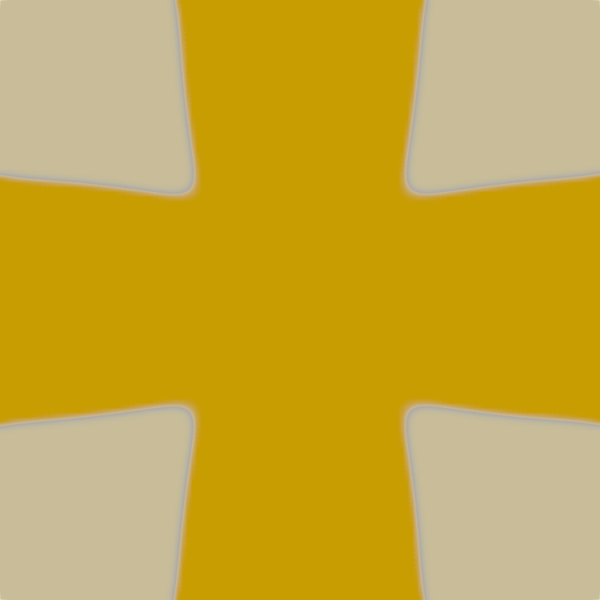}
    \end{minipage}
  & \begin{minipage}{0.15\textwidth} \centering
      \includegraphics[max width=0.9\linewidth, max height=0.9\linewidth]{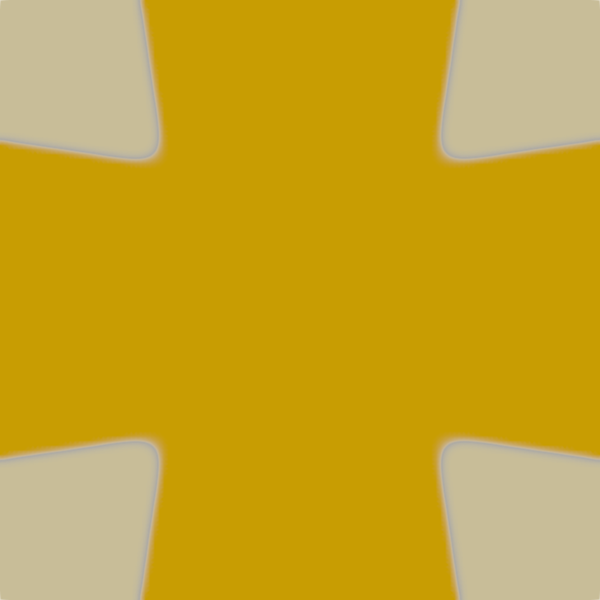}
    \end{minipage}
 \\
\end{tabular}
}
\caption
[Optimal material distributions on a plate for a centered load.]
{\label{fig:shapeOptKoiterPlateCenteredLoad} 
Optimal material distributions on a plate $S = [0,1]^2$ for a centered load in normal direction supported on $(0.45,0.55)^2$.
We compare the resulting hard phase for different area constraints $\AreaHardPhase=0.25,0.375,0.5,0.625,0.75$.

}
\end{figure}

\bigskip

\emph{Constant Load on a Plate.}
Next, in Figure~\ref{fig:shapeOptKoiterPlateConstLoad}, still for the flat case $S_{ \text{ref}} = [0,1]^2$,
we consider a force $f = \left( 0, 0, c \right)$
acting everywhere on the plate in normal direction for some constant $c$.
Again, we assume clamped boundary conditions of the displacement on $\partial S$.
As above, we choose $\penaltyPerimeter = 10^{-3}$ for the Modica--Mortola functional and $\delta = 10^{-2}$ for the thickness.
Furthermore, we compare different area constraints $\AreaHardPhase = \frac{k}{8}$ for $k=2,3,4,5,6$
and set $c = - 20 \AreaHardPhase$ for the force.
While for the centered load it has been sufficient to stabilize the area in the region, where the force is concentrated, by trusses connected to the boundary,
for a constant load there is a need of microstructures to keep the deformation as small as possible in terms of the potential energy.
\begin{figure}[!htbp]
\resizebox{\textwidth}{!}{
\begin{tabular}{  m{2cm} c c c c c }
  deformed config.
  & \begin{minipage}{0.15\textwidth} \centering
      \includegraphics[max width=\linewidth]{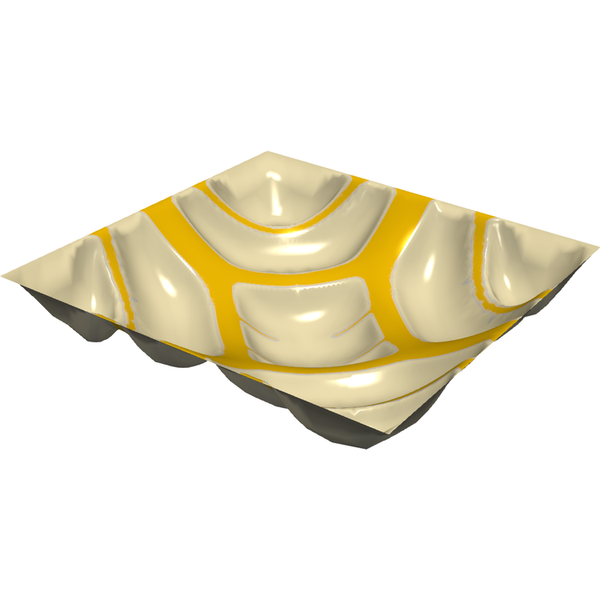}
  \end{minipage}
  & \begin{minipage}{0.15\textwidth} \centering
      \includegraphics[max width=\linewidth]{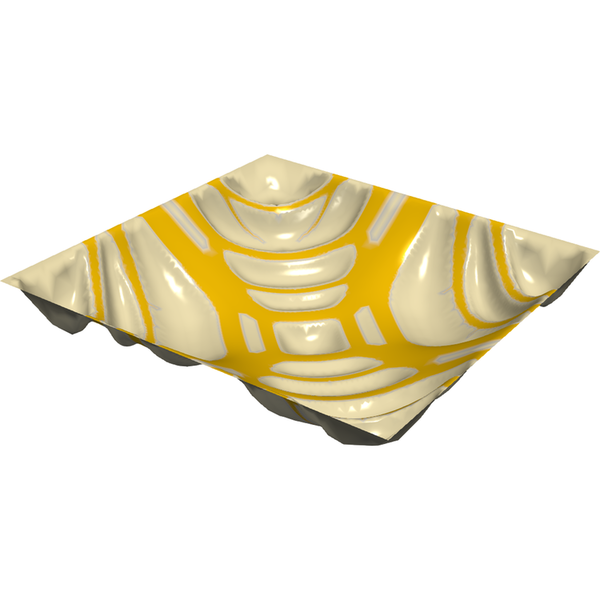}
  \end{minipage}
  & \begin{minipage}{0.15\textwidth} \centering
      \includegraphics[max width=\linewidth]{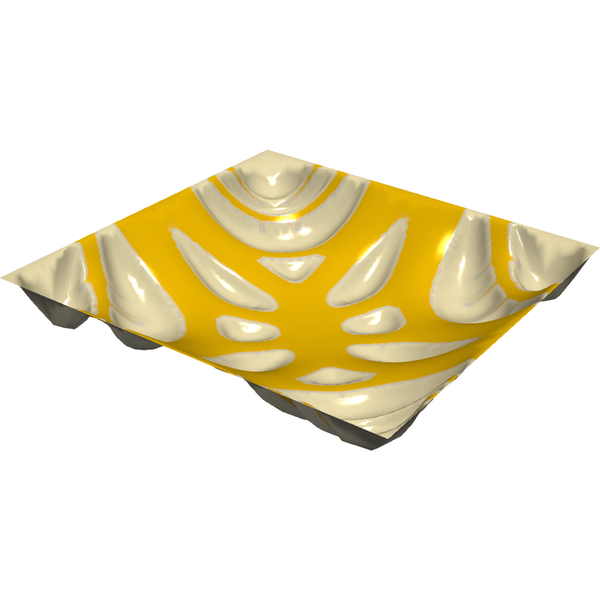}
  \end{minipage}
  & \begin{minipage}{0.15\textwidth} \centering
      \includegraphics[max width=\linewidth]{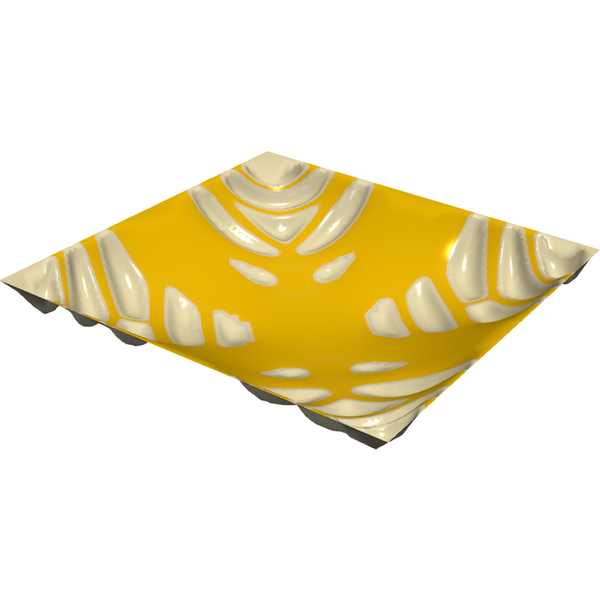}
  \end{minipage}
  & \begin{minipage}{0.15\textwidth} \centering
      \includegraphics[max width=\linewidth]{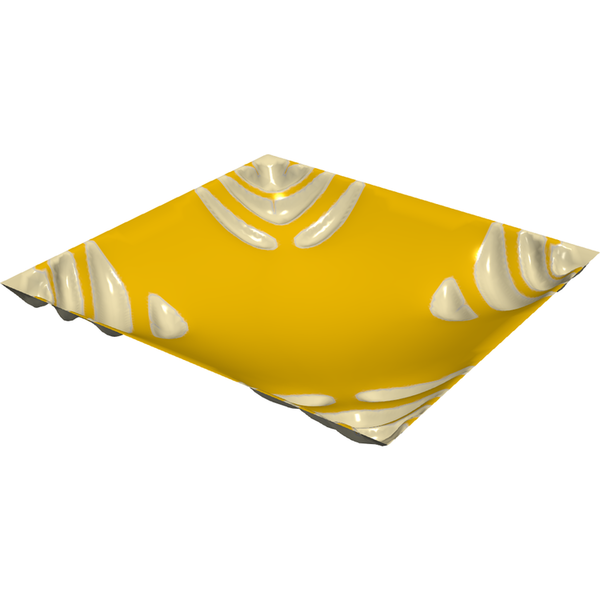}
  \end{minipage} 
  \\ 
  undeformed config.
  & \begin{minipage}{0.15\textwidth} \centering 
     \includegraphics[max width=0.9\linewidth, max height=0.9\linewidth]{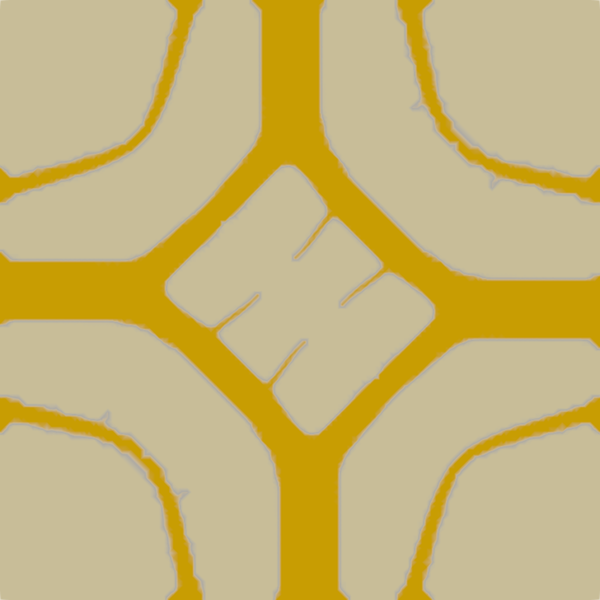}
    \end{minipage}
  & \begin{minipage}{0.15\textwidth} \centering
      \includegraphics[max width=0.9\linewidth, max height=0.9\linewidth]{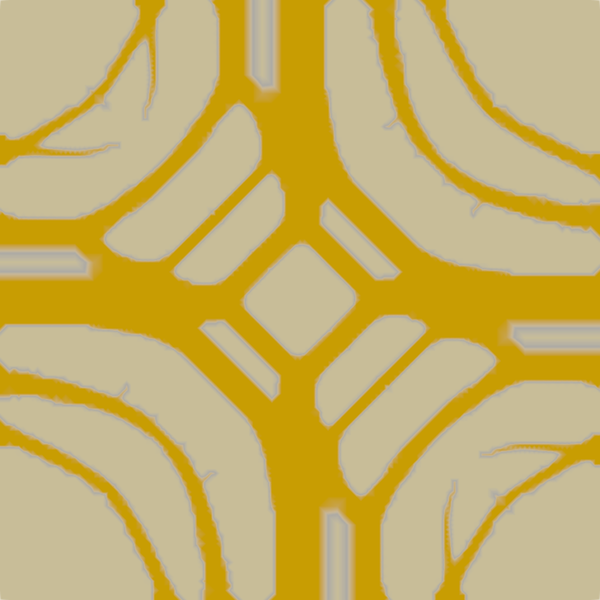}
     \end{minipage}
  & \begin{minipage}{0.15\textwidth} \centering
      \includegraphics[max width=0.9\linewidth, max height=0.9\linewidth]{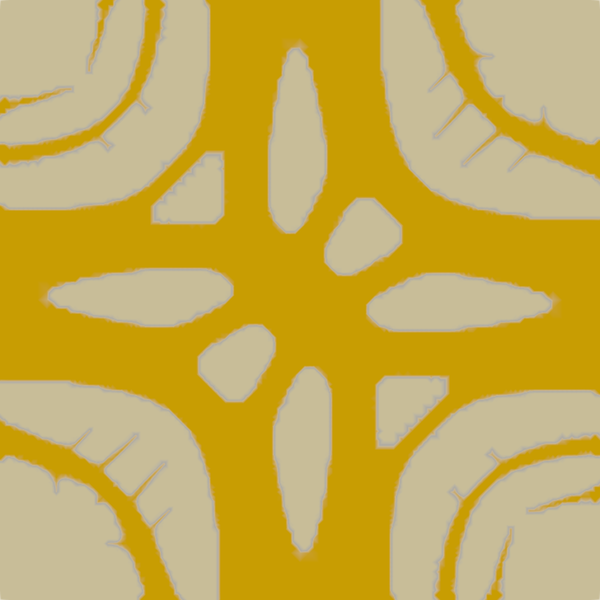}
     \end{minipage}
  & \begin{minipage}{0.15\textwidth} \centering
      \includegraphics[max width=0.9\linewidth, max height=0.9\linewidth]{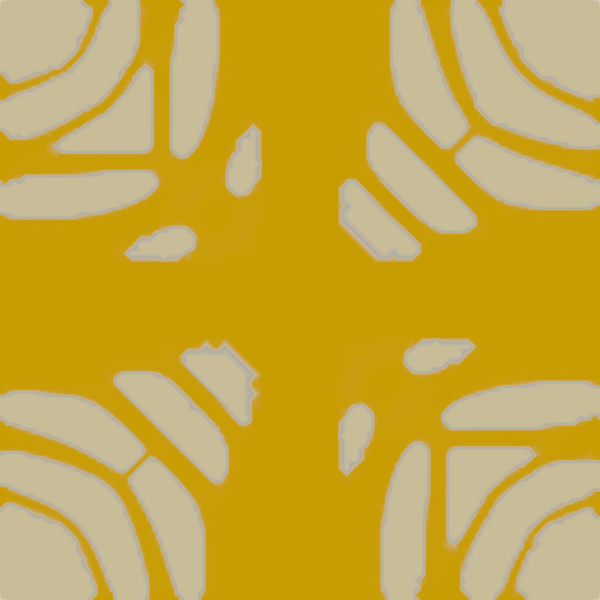}
    \end{minipage}
  & \begin{minipage}{0.15\textwidth} \centering
      \includegraphics[max width=0.9\linewidth, max height=0.9\linewidth]{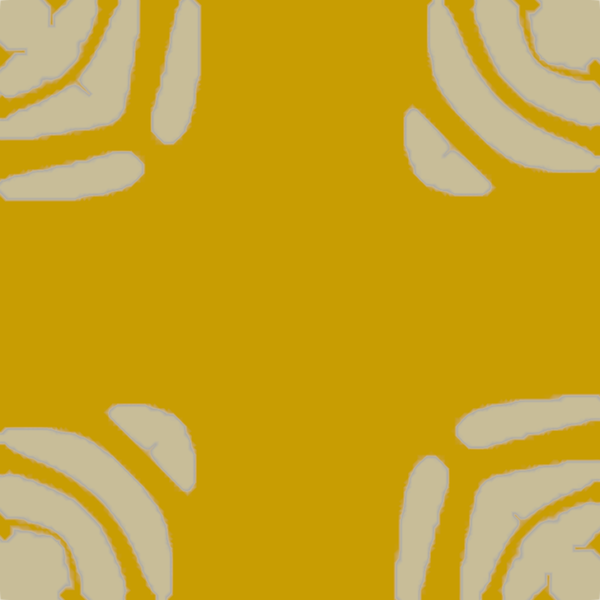}
    \end{minipage}
  \\
 \end{tabular}
 }
\caption
[Optimal material distributions on a plate for a constant load.]
{Optimal material distributions on a plate $S = [0,1]^2$ for a constant load acting in normal direction
and clamped boundary conditions on $\partial S$.
We compare the results for different area constraints $\AreaHardPhase=0.25,0.375,0.5,0.625,0.75$.
}
\label{fig:shapeOptKoiterPlateConstLoad}
\end{figure}

\bigskip

\emph{Constant Load on a Hemisphere.}
Now, we investigate optimal material distributions on the upper hemisphere
with the inverse of the stereographic projection as parametrization $\psi_{ \text{ref}}(\xi) 
        = \left( \frac{2 \xi_1}{\xi_1^2 + \xi_2^2 + 1}, \;
                 \frac{2 \xi_2}{\xi_1^2 + \xi_2^2 + 1}, \;
                 \frac{1 - \xi_1^2 - \xi_2^2}{\xi_1^2 + \xi_2^2 + 1}
\right) $ over the unit disc.
We assume clamped boundary conditions on the left and right side, i.e., we set 
$\Gamma_D =  \{ p \in S_{ \text{ref}}  \; : \;  p_3 = 0 \, , \; |p_1| \geq 0.9\}$.
Moreover, we consider a single area constraint $\AreaHardPhase = \tfrac12 \mathcal{H}^2(S_{ \text{ref}})$.
A force $f = \left( 0, 0, c \right)$ with $c = 0.001$ is acting on the reference domain and the thickness is $\delta = 10^{-2}$.
Then, we apply $6$ adaptive refinement steps.

In Figure~\ref{fig:shapeOptKoiterHalfSphereLR}, 
we compare different values for the parameter $\penaltyPerimeter$ to penalize the Modica--Mortola functional.
Indeed, for $\penaltyPerimeter \to 0$ we observe successively finer pattern for the hard phase again underpinning the emergence of 
a microstructures as the minimizer in the limit.
\begin{figure}[!htbp]
\resizebox{\textwidth}{!}{
\begin{tabular}{ m{2cm} c c c }
  undeformed config.
  & \begin{minipage}{0.25\textwidth} \centering
       \includegraphics[max width=\linewidth]{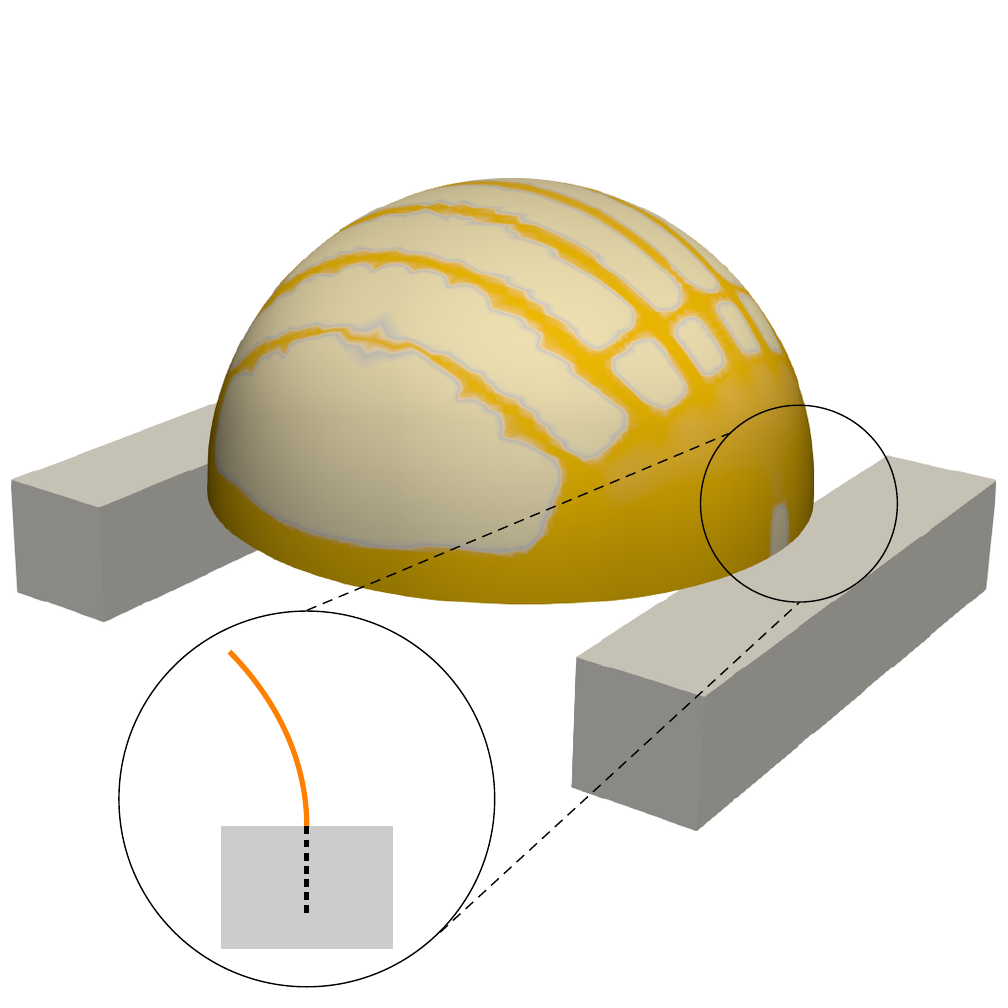}
  \end{minipage}
  & \begin{minipage}{0.25\textwidth} \centering
      \includegraphics[max width=\linewidth]{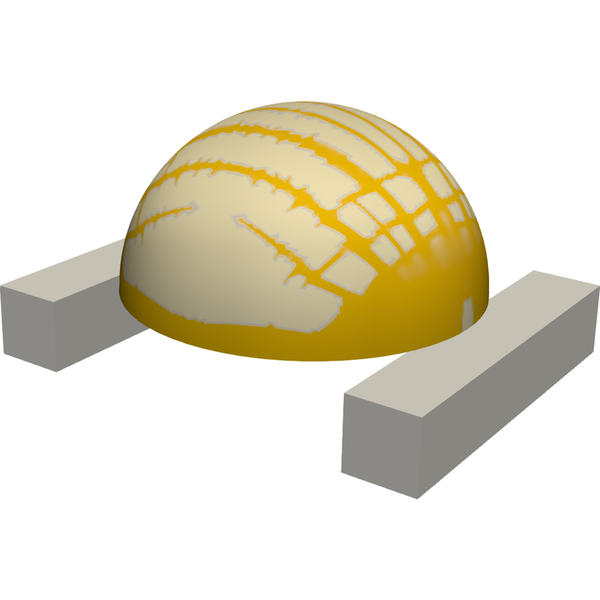}
  \end{minipage}
  & 
  \begin{minipage}{0.25\textwidth} \centering
      \includegraphics[max width=\linewidth]{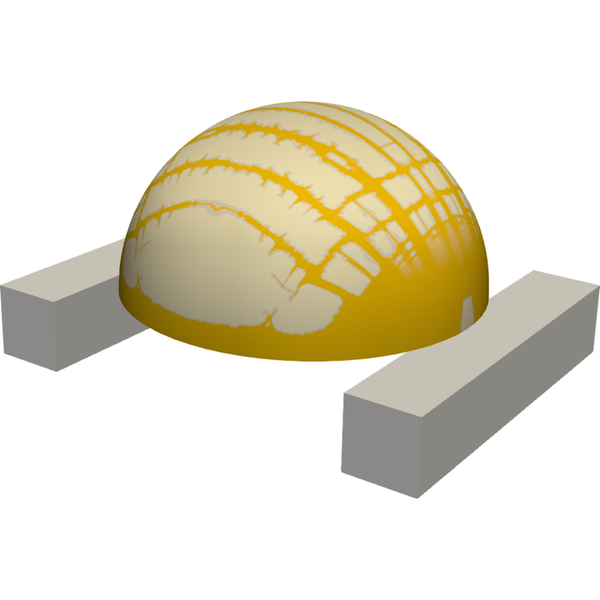}
  \end{minipage}
  \\
 chart domain
  & \begin{minipage}{0.25\textwidth} \centering 
     \includegraphics[max width=0.9\linewidth, max height=0.9\linewidth]{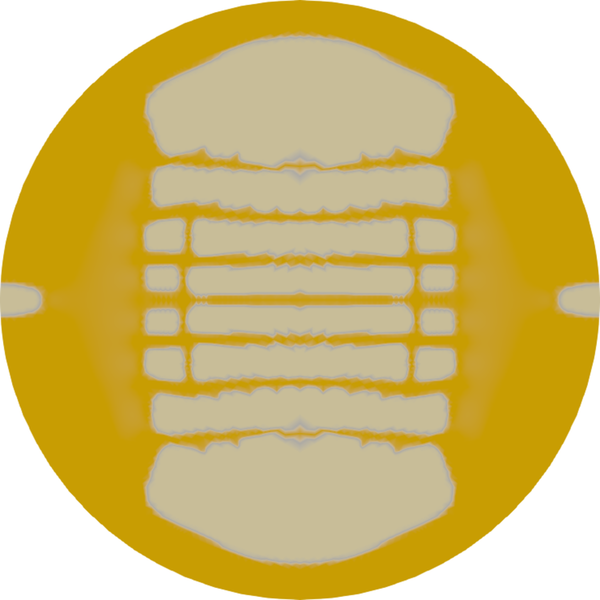}
    \end{minipage}
  & \begin{minipage}{0.25\textwidth} \centering
      \includegraphics[max width=0.9\linewidth, max height=0.9\linewidth]{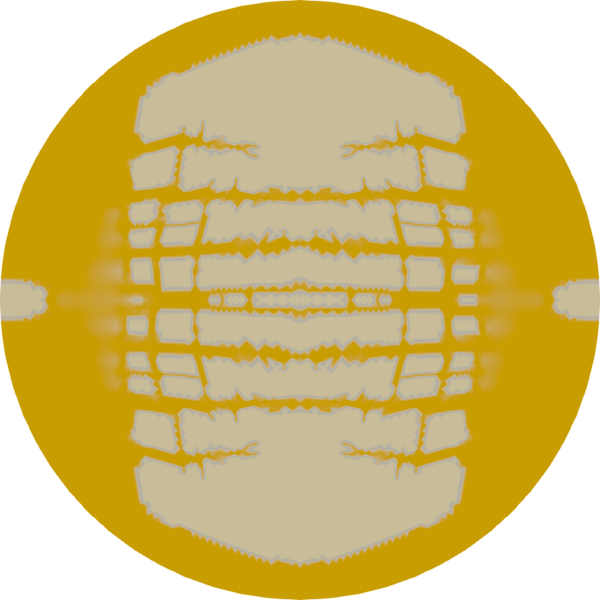}
     \end{minipage}
  & 
  \begin{minipage}{0.25\textwidth} \centering
      \includegraphics[max width=0.9\linewidth, max height=0.9\linewidth]{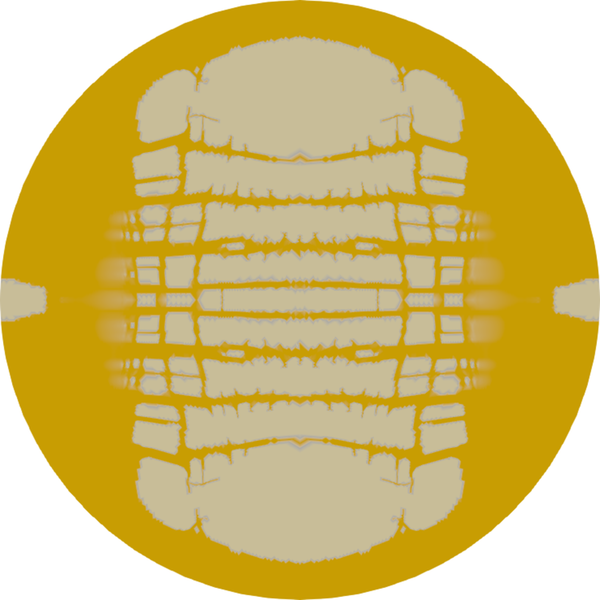}
    \end{minipage}
  \\ 
\end{tabular}
}
\caption
[Optimal material distributions on a hemisphere.]
{For a material distributions on a hemisphere
the optimal hard phase is shown for $\penaltyPerimeter=10^{-7},10^{-8},10^{-9}$ (left to right)
in the deformed configuration (top) and on the chart domain (bottom). 
In addition, the clamped boundary condition is sketched.
}
\label{fig:shapeOptKoiterHalfSphereLR}
\end{figure}

\bigskip

\emph{Constant Load on a Half Cylinder.}
Finally, we consider the half cylinder $S_{ \text{ref}}$
which is parametrized over the chart domain $\omega = [0,1]^2$ via
$
 \psi_{ \text{ref}}(\xi)
 = \left( \frac{1}{2\pi} (1 - \cos( \pi \xi_1) ), \xi_2, \frac{1}{2\pi} \sin( \pi \xi_1) \right)
$
We assume clamped boundary conditions on the left and right sides w.r.t. the $e_2$-direction, i.e., we set 
$\Gamma_D = \{ p =  \in S_{ \text{ref}} \; : \;  p_2 \in \{ 0, 1 \} \}$.
Here, we study the effect for different thickness parameters $\delta$.  
We consider a single area constraint $\AreaHardPhase = \tfrac12 \mathcal{H}^2(S_{ \text{ref}})$ and a constant 
force $f = \left( 0, 0, c \right)$ with $c = -10$.
In Figure~\ref{fig:shapeOptKoiterHalfCylinderTB}, we depict the computational results.
\begin{figure}[!htbp]
\resizebox{\textwidth}{!}{
 \begin{tabular}{  m{2.5cm} c c c c }
  {\footnotesize deformed config. for homogeneous material}
 & \begin{minipage}{0.15\textwidth} \centering
      \includegraphics[max width=\linewidth]{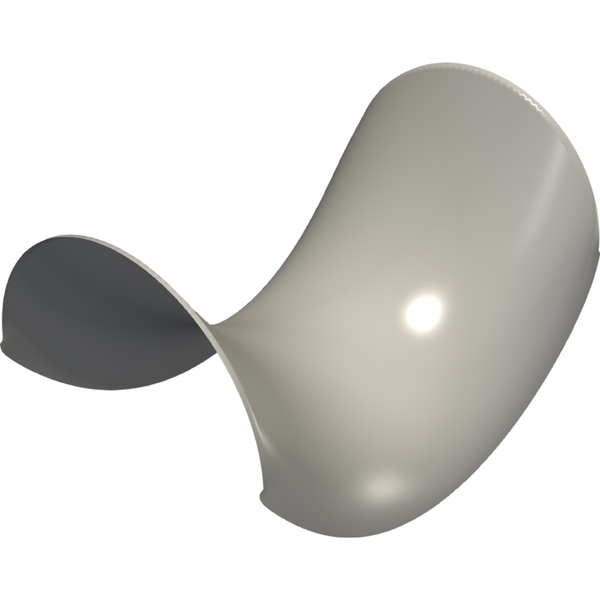}
  \end{minipage}
 & \begin{minipage}{0.15\textwidth} \centering
      \includegraphics[max width=\linewidth]{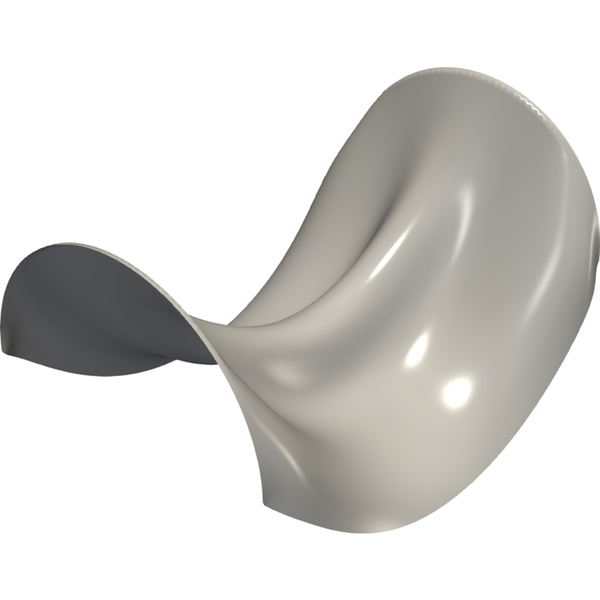}
  \end{minipage}
 & \begin{minipage}{0.15\textwidth} \centering
      \includegraphics[max width=\linewidth]{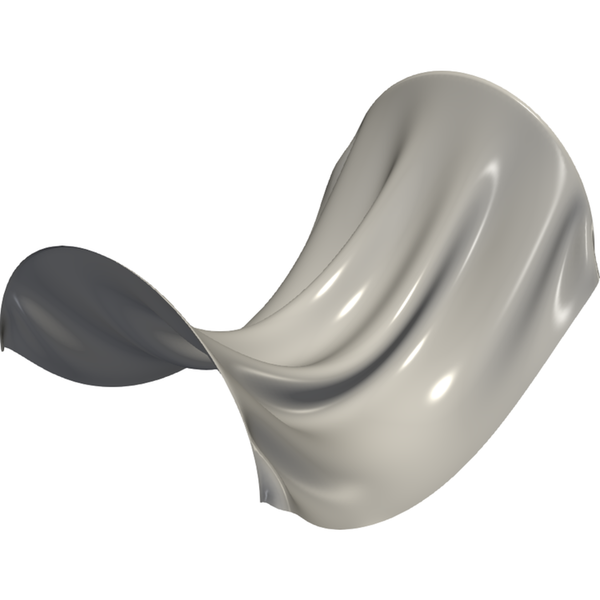}
  \end{minipage}
 & \begin{minipage}{0.15\textwidth} \centering
      \includegraphics[max width=\linewidth]{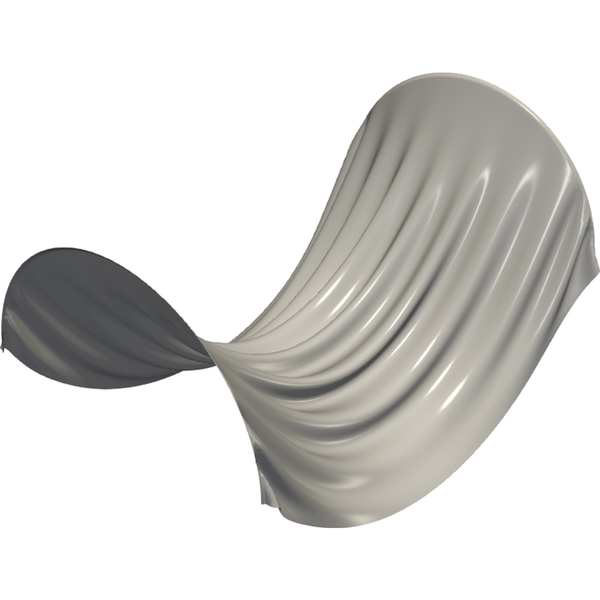}
  \end{minipage}
  \\ 
  {\footnotesize deformed config. for optimal material}
 & \begin{minipage}{0.15\textwidth} \centering
      \includegraphics[max width=\linewidth]{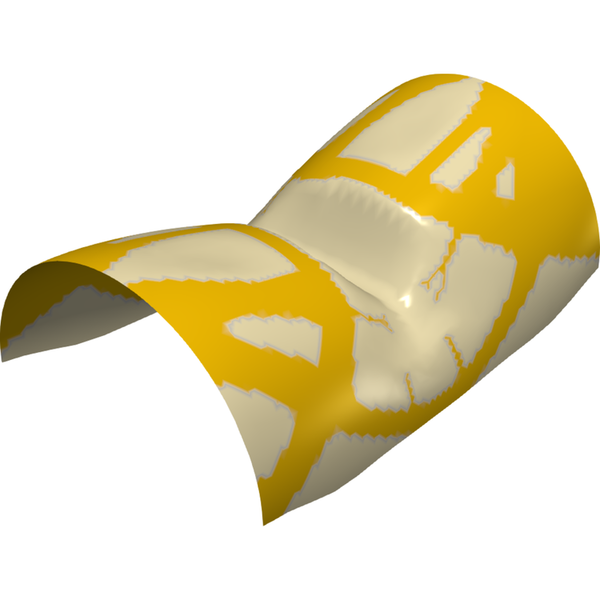}
  \end{minipage}
 & \begin{minipage}{0.15\textwidth} \centering
      \includegraphics[max width=\linewidth]{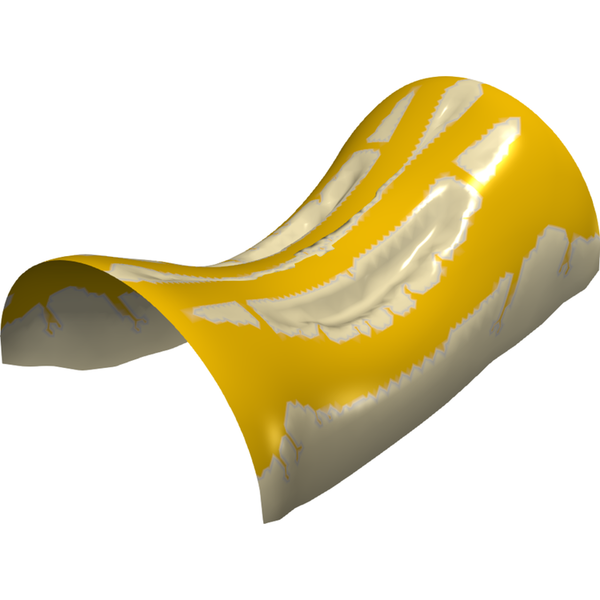}
  \end{minipage}
 & \begin{minipage}{0.15\textwidth} \centering
      \includegraphics[max width=\linewidth]{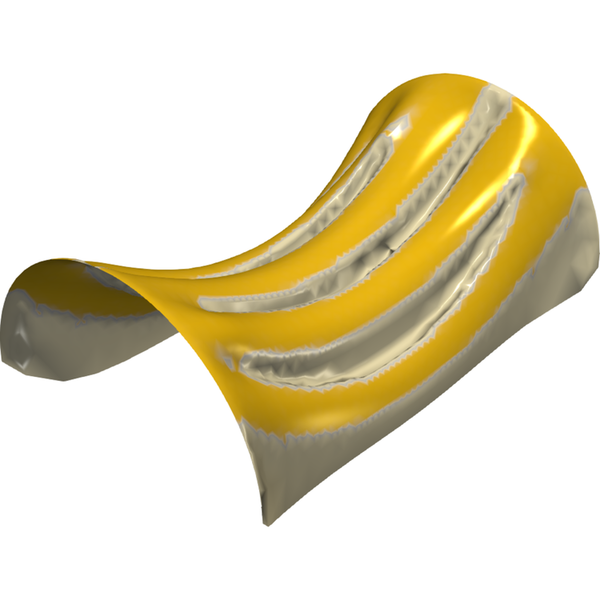}
  \end{minipage}
 & \begin{minipage}{0.15\textwidth} \centering
      \includegraphics[max width=\linewidth]{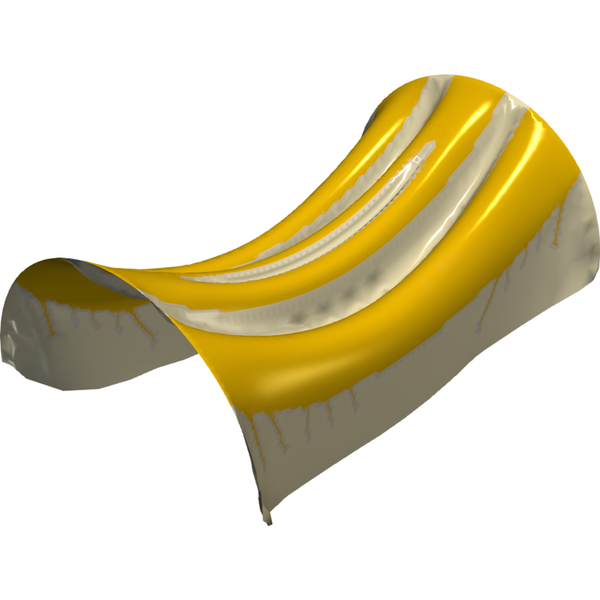}
  \end{minipage}
  \\
 {\footnotesize chart domain}
 & \begin{minipage}{0.15\textwidth} \centering
      \includegraphics[max width=\linewidth]{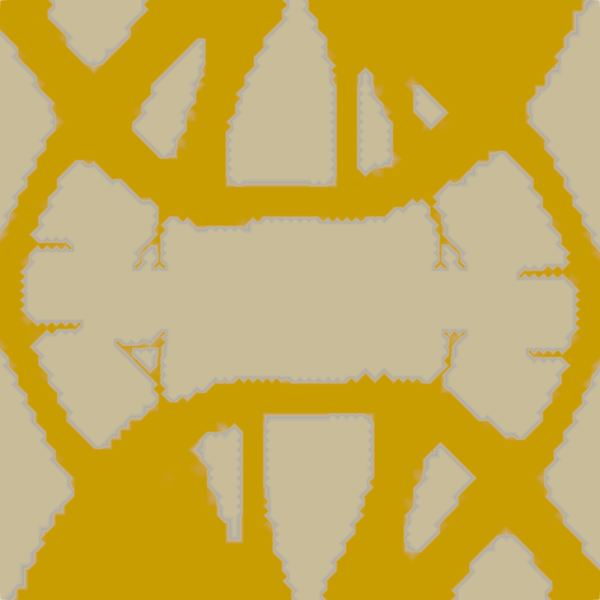}
  \end{minipage}
 & \begin{minipage}{0.15\textwidth} \centering
      \includegraphics[max width=\linewidth]{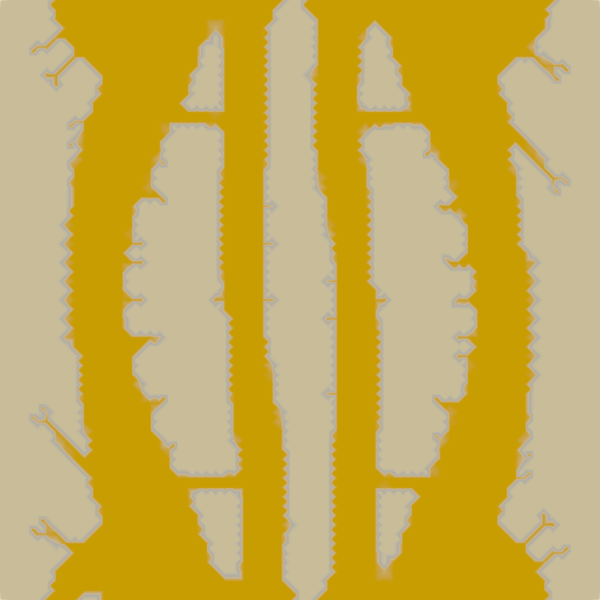}
  \end{minipage}
 & \begin{minipage}{0.15\textwidth} \centering
      \includegraphics[max width=\linewidth]{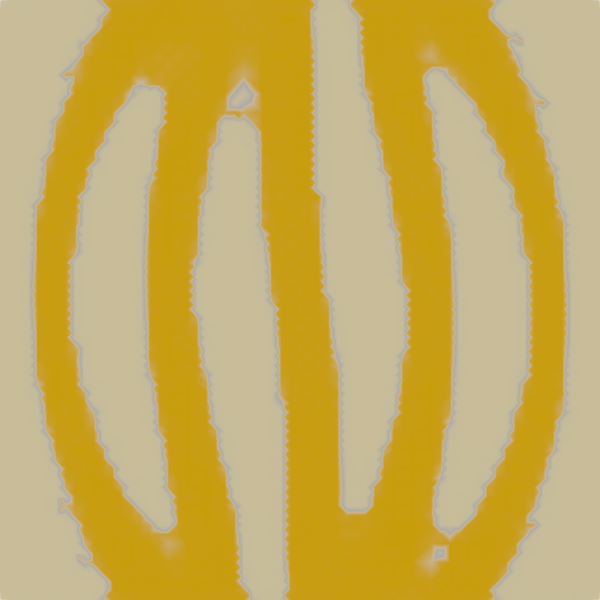}
  \end{minipage}
 & \begin{minipage}{0.15\textwidth} \centering
      \includegraphics[max width=\linewidth]{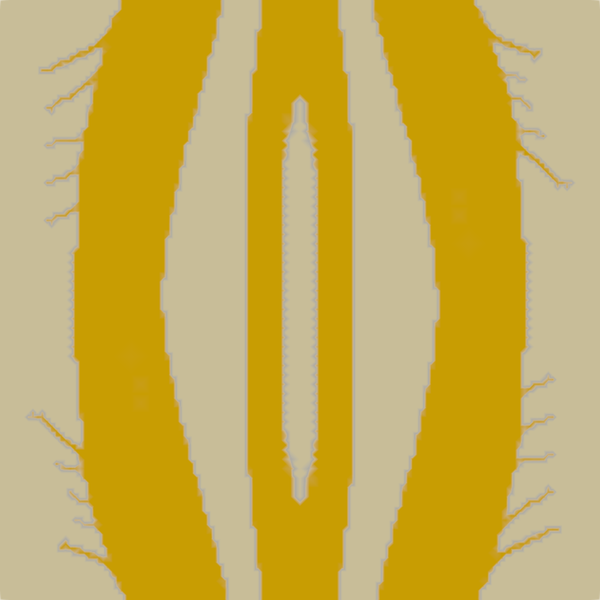}
  \end{minipage}
  \\ 
\end{tabular}
}
\caption
[Optimal material distributions on a half cylinder.]
{First, we show solutions of the state equation for a homogeneous material distribution (top)
for clamped left and right side and thickness parameters $\delta=10^{-1},10^{-1.5},10^{-2},10^{-2.5}$ (from left to right).
The optimal material distributions on a half cylinder is shown in the deformed configuration (middle) and on the chart (bottom).
}
\label{fig:shapeOptKoiterHalfCylinderTB}
\end{figure}

\paragraph*{Acknowledgements.}
Peter Hornung acknowledges support of the Deutsche Forschungsgemeinschaft (DFG, German Research Foundation). 
Martin Rumpf and Stefan Simon acknowledge support of the Collaborative Research Center 1060, 
funded by the DFG - Projektnummer 211504053 - SFB 1060 
and the Hausdorff Center for Mathematics, funded by the DFG under Germany \' s Excellence Strategy - GZ 2047/1, Project-ID 390685813.

\def\polhk#1{\setbox0=\hbox{#1}{\ooalign{\hidewidth
  \lower1.5ex\hbox{`}\hidewidth\crcr\unhbox0}}}

\end{document}